\documentclass[11pt]{amsart}
\usepackage{amssymb}
\usepackage{bm}
\usepackage[centertags]{amsmath}
\usepackage{amsfonts}
\usepackage{amsthm}
\usepackage{mathrsfs}
\usepackage[normalem]{ulem}
\usepackage{comment}
\usepackage{enumitem}
\usepackage{mathabx}
\usepackage{hyperref}
\usepackage[noabbrev]{cleveref}

\usepackage{xifthen}
\usepackage{ifthenx}
\usepackage{xargs}

\usepackage[table,svgnames,x11names]{xcolor}
\usepackage[colorinlistoftodos,prependcaption,textsize=tiny,textwidth=3.5cm]{todonotes}
\presetkeys{todonotes}{fancyline}{}
\newcommandx{\change}[2][1=]{\todo[linecolor=blue,backgroundcolor=blue!25,bordercolor=blue,#1]{#2}}
\newcommandx{\changein}[2][1=]{\change[inline, caption={change}, #1]{%
    \begin{minipage}{\textwidth-20pt}#2\end{minipage}}}
\newcommandx{\todoin}[2][1=]{\todo[inline, caption={todo}, #1]{%
    \begin{minipage}{\textwidth-20pt}#2\end{minipage}}}

\newcommandx{\remove}[2][1=]{\todo[linecolor=Plum,backgroundcolor=Plum!25,bordercolor=Plum,#1]{#2}}
\newcommandx{\removein}[2][1=]{\remove[inline, caption={todo}, #1]{%
    \begin{minipage}{\textwidth-20pt}#2\end{minipage}}}

\makeatletter
\def\l@subsection{\@tocline{2}{0pt}{2.5pc}{2.5pc}{}}%
\makeatother
\makeatletter
\def\l@subsubsection{\@tocline{3}{0pt}{5pc}{5pc}{}}%
\makeatother

\usepackage{hyperref}
\hypersetup{pdftex,colorlinks=true,allcolors=black}
\usepackage{hypcap}

\DeclareFontFamily{OT1}{pzc}{}
\DeclareFontShape{OT1}{pzc}{m}{it}{<-> s * [1.10] pzcmi7t}{}
\DeclareMathAlphabet{\mathpzc}{OT1}{pzc}{m}{it}

\theoremstyle{definition}
\newtheorem{thm}{Theorem}[section]
\newtheorem{dfn}[thm]{Definition}
\newtheorem{lem}[thm]{Lemma}
\newtheorem{prop}[thm]{Proposition}

\newtheorem{rem}[thm]{Remark}
\newtheorem{exa}[thm]{Example}

\newtheorem{ntt}[thm]{Notation}

\newtheorem*{defn*}{Definition}

\newtheorem*{thm*}{Theorem}

\newtheorem*{cor*}{Corollary}

\newtheorem*{prp*}{Proposition}

\newtheorem*{clm}{Claim}

\newcommand{\N}{\mathbb{N}}

\newcommand{\R}{\mathbb{R}}

\newcommand{\trn}[1]{{\left\vert\kern-0.25ex\left\vert\kern-0.25ex\left\vert #1 
    \right\vert\kern-0.25ex\right\vert\kern-0.25ex\right\vert}}

\newcommand{\trnsmall}[1]{{\vert\kern-0.25ex\vert\kern-0.25ex\vert #1 
    \vert\kern-0.25ex\vert\kern-0.25ex\vert}}

\DeclareMathOperator{\supp}{supp}

\long\def\symbolfootnote[#1]#2{\begingroup%
\def\thefootnote{\fnsymbol{footnote}}\footnote[#1]{#2}\endgroup}

\makeatletter
\@namedef{subjclassname@2020}{%
  \textup{2020} Mathematics Subject Classification}
\makeatother

\allowdisplaybreaks

\begin{document}

\title[Operator reduction in $R_\alpha^p$ spaces]{Coordinate systems and distributional embeddings in Bourgain-Rosenthal-Schechtman spaces: a framework for operator reduction}

\author[K. Konstantos]{Konstantinos Konstantos}

\address{Konstantinos Konstantos, Department of Mathematics and Statistics, York University, 4700 Keele Street, Toronto, Ontario, M3J 1P3, Canada}

\email{kostasko@yorku.ca}

\author[P. Motakis]{Pavlos Motakis}

\address{Pavlos Motakis, Department of Mathematics and Statistics, York University, 4700 Keele Street, Toronto, Ontario, M3J 1P3, Canada}

\email{pmotakis@yorku.ca}

\thanks{This research was supported by NSERC Grant RGPIN-2021-03639.}


\keywords{Bourgain-Rosenthal-Schechtman spaces, finite-dimensional decomposition, distributional embedding, factorization property}

\subjclass[2020]{46B09, 46B25, 46B28, 46E30, 47A68}

\begin{abstract}
For every $1\leq \alpha<\omega_1$, we construct an explicit unconditional finite-dimensional decomposition (FDD) $(X_\lambda)_{\lambda\in\mathcal{T}_\alpha}$ of the Bourgain-Rosenthal-Schechtman space $R_\alpha^{p,0}$ by blocking its standard martingale difference sequence (MDS) basis. This FDD has strong reproducing properties and supports a theory of distributional representations between the spaces $R_\alpha^{p,0}$, $1\leq \alpha<\omega_1$. We use this framework to prove an approximate orthogonal reduction: every bounded linear operator on a limit space $R_\alpha^{p,0}$ is, via a distributional embedding and up to arbitrary precision, reduced to a scalar FDD-diagonal operator. As a consequence, the standard MDS bases of the limit spaces $R_\alpha^{p,0}$ satisfy the factorization property.
\end{abstract}

\maketitle


\tableofcontents

\section{Introduction}

Let $1<p<\infty$ be fixed and let $q = p/(p-1)$ and $p^* = \max\{p,q\}$. The first example of a non-trivial complemented subspace $X_p$ of $L_p = L_p[0,1]$ was given by Rosenthal in 1970 \cite{rosenthal:1970:Xp}. This space is spanned by an independent sequence of symmetric three-valued random variables $(f_n)_{n=1}^\infty$ such that the ratio $\|f_n\|_{2}/\|f_n\|_{p^*}$ has a subsequence tending to zero sufficiently slowly. To study the properties of these spaces, such as their complementation, for $p>2$, Rosenthal introduced what is now known as Rosenthal’s inequality, which states that for any scalar sequence $(a_n)_{n=1}^\infty$,
\[
\big\|\sum_{n=1}^\infty a_n f_n\big\|_p \sim_p \max\Big\{\Big(\sum_{n=1}^\infty |a_n|^p \|f_n\|_p^p\Big)^{1/p},\Big(\sum_{n=1}^\infty |a_n|^2 \|f_n\|_2^2\Big)^{1/2}\Big\}.
\]
This equivalence makes it possible to study $X_p$ as a sequence space whose norm admits a simple and explicit description. Schechtman later proved in 1975 \cite{schechtman:1975} that the iterated tensor products $\otimes^n X_p$ are pairwise non-isomorphic, thereby showing the existence of infinitely many pairwise non-isomorphic complemented subspaces of $L_p$.

In their 1981 paper \cite{bourgain:rosenthal:schechtman:1981} Bourgain, Rosenthal, and Schechtman defined a transfinite collection of complemented subspaces of $L_p$, denoted by $R_\alpha^p$, $0\leq \alpha<\omega_1$, and proved the existence of uncountably many pairwise non-isomorphic complemented subspaces of $L_p$. Just like the space $X_p$, the definition of these spaces is based on probabilistic language. The first space $R_0^p$ consists of the constant functions. If $R_\alpha^p$ is defined, then $R_{\alpha+1}^p$ is the disjoint sum of two $1/2$-compressed (that is, $1/2$-dilated) isometric copies of $R_\alpha^p$. More formally, for $I = [a,b)\subset[0,1)$, define $T_I:L_p\to L_p$ by setting $(T_If)(t) = f((t-a)/(b-a))$ for $t\in I$ and $(T_If)(t)=0$ otherwise. This $T_I$ is an example of a scaled isometry of a special type, called a $\theta$-compression, where $\theta=b-a$. Then,
\[R_{\alpha+1}^p = T_{[0,1/2)}(R_\alpha^p) + T_{[1/2,1)}(R_\alpha^p).\]
The distributional structure of the definition is an essential part of the construction. For a limit ordinal $\alpha$, the space $R_\alpha^p$ is the independent sum of the spaces $R_\beta^p$, $0\leq \beta<\alpha$, that is, there are distributional embeddings $T_\beta^\alpha:R_\beta^p\to L_p$, $0\leq \beta<\alpha$ with independent ranges such that
\[R_\alpha^p = [(T_\beta^\alpha(R_\beta^p))_{\beta<\alpha}].\]
By construction, each space $R_\alpha^p$ is spanned by the constant function together with a symmetric three-valued martingale difference sequence (MDS), and thus by Burkholder’s martingale theorem from \cite{burkholder:1966} the space admits an unconditional basis. The proof of the complementation of these spaces relies on a clever threefold application of an inequality of Burkholder, Davis, and Gundy from \cite{burkholder:davis:gundy:1972} to a tree-martingale structure and on Burkholder’s martingale theorem. The result is that the spaces $R_\alpha^p$ are complemented via orthogonal projections. The fact that the family $(R_\alpha^p)_{0\leq \alpha<\omega_1}$ contains a cofinal subset of mutually non-isomorphic spaces is proved using descriptive set-theoretic methods, specifically an ordinal index, though this subset is not explicitly specified. The ordinal indices of a few specific subspaces of $L_p$ were computed in \cite{dutta:khurana:2016}.

The most comprehensive study of the Bourgain-Rosenthal-Schechtman spaces was carried out by Alspach in his 1999 paper \cite{alspach:1999}. This is a deep work with many components. He applied, in a broader context, the notion of $(p,2)$-sums of sequences of subspaces of $L_p$ for $p>2$. This notion had been used earlier by Rosenthal in \cite{rosenthal:1970:Xp} and by Johnson and Odell in \cite{johnson:odell:1981}, primarily to study $X_p$. Alspach combined it with extensive use of Rosenthal’s inequality to analyze the isomorphic structure of the spaces $R_\alpha^p$ for $\alpha<\omega_1$. Among other results, Alspach computed the isomorphism classes of the infinite-dimensional Bourgain-Rosenthal-Schechtman spaces, proving that they are represented by the spaces $R_\alpha^p$, where $\alpha<\omega_1$ is a limit ordinal. He also proved that the space $X_p\otimes X_p$ is not isomorphic to a complemented subspace of any of the Bourgain-Rosenthal-Schechtman spaces. A related study of this construction appears in \cite{ghawadrah:2016}, where some of the results from \cite{bourgain:rosenthal:schechtman:1981} are extended to Orlicz spaces. Other works, such as \cite{mueller:schechtman:1989}, \cite{rzeszut:wojciechowski:2017}, \cite{konstantos:motakis:2025}, and \cite{konstantos:speckhofer:2025} borrow tools from the Bourgain-Rosenthal-Schechtman construction, but its full potential may not yet have been achieved, and further applications remain to be explored.

Unlike Alspach’s work, which is based on isomorphic tools and Rosenthal’s inequality, our work focuses on the distributional structure of the Bourgain-Rosenthal-Schechtman spaces, and it is partly a continuation of \cite{konstantos:motakis:2025} by the present authors. We develop a theory around the spaces’ standard MDS bases and use it to define a class of distributional embeddings of the spaces
\[R_\alpha^{p,0} = \{f\in R_\alpha^{p} : \mathbb E(f) = 0\},\;1\leq \alpha<\omega_1,\]
into each other. The motivation behind this work comes from the primary factorization property of Banach spaces and the factorization property of bases of Banach spaces. Given bounded linear operators $T,S:X\to X$, where $X$ is a Banach space, we call $T$ a factor of $S$ if there exist bounded linear operators $A,B:X\to X$ such that $S = ATB$. A Banach space has the primary factorization property if, for every bounded linear operator $T:X\to X$, either $T$ or $id - T$ is a factor of $id:X\to X$. This property is related, for instance, to the primariness of Banach spaces ($X$ is primary if whenever $X\simeq Y\oplus Z$ then $Y\simeq X$ or $Z\simeq X$) (see, e.g., \cite{lindenstrauss:1971}, \cite{lindenstrauss:pelczynski:1971}, \cite{maurey:1975:2}, \cite{pietsch:1978}, \cite{enflo:starbird:1979}, \cite{johnson:odell:1981}, \cite{capon:1982:2}, \cite{bourgain:1983}, \cite{mueller:1988}, \cite{blower:1990}, \cite{lechner:motakis:mueller:schlumprecht:2022}) and to the study of operator ideals (see, e.g., \cite{dosev:johnson:2010}, \cite{kania:laustsen:2012}, \cite{dosev:johnson:schechtman:2013}, \cite{kania:lechner:2022}, \cite{lechner:speckhofer:2025}). A basis $(e_n)_{n=1}^\infty$ of a Banach space $X$ has the factorization property if every bounded linear operator $T:X\to X$ such that $\inf_n |e_n^*(Te_n)| > 0$ is a factor of $id:X\to X$. While the term was first introduced in \cite{laustsen:lechner:mueller:2015}, the property had been studied implicitly since Andrew’s 1977 paper \cite{andrew:1979}. It has been the focal point of several recent research articles, such as \cite{lechner:2018:SL-infty}, \cite{lechner:motakis:mueller:schlumprecht:2020}, \cite{navoyan:2023}, \cite{lechner:speckhofer:2025}, \cite{konstantos:motakis:2025}, and \cite{konstantos:speckhofer:2025}. It was used, for instance, implicitly by Alspach in \cite{alspach:1999} to show that $X_p\otimes X_p$ is not isomorphic to a complemented subspace of a Bourgain-Rosenthal-Schechtman space, and by Johnson and Schechtman in \cite{johnson:schechtman:2024} to show that the maximal operator ideal in $L_1$ does not admit a right approximate identity. A corollary of our work is that the limit Bourgain-Rosenthal-Schechtman spaces with their standard MDS bases satisfy the factorization property.

The first contribution of this paper, taking place in Section \ref{fdd section}, is the definition of an explicit unconditional FDD $(X_\lambda)_{\lambda\in\mathcal{T}_\alpha}$ for each space $R_\alpha^{p,0}$, obtained by blocking the standard MDS bases defined by Bourgain, Rosenthal, and Schechtman in \cite{bourgain:rosenthal:schechtman:1981}. It is constructed recursively, following the recursive structure of the spaces $R_\alpha^{p,0}$ for $1\leq \alpha<\omega_1$. For each $n\in\mathbb{N}$, let $\mathcal{D}^n$ denote the set of left-closed, right-open dyadic intervals of length $1/2^n$ in $[0,1)$, and let $L_p^{n,0}$ be the subspace of mean-zero functions spanned by the characteristic functions of these intervals. Then, for every $n\in\mathbb{N}$, $R_p^{n,0} = L_p^{n,0}$, and we consider the trivial FDD $(L_p^{n,0})$ for this space. For a limit ordinal $\alpha$, assume that for each $R_\beta^{p,0}$, $1\leq \beta<\alpha$, an FDD $(X_\lambda)_{\lambda\in\mathcal{T}_\beta}$ has been defined. We then define $(X_\lambda)_{\lambda\in\mathcal{T}_\alpha}$ by taking a union of the independent FDDs $(T_\beta^\alpha(X_\lambda))_{\lambda\in\mathcal{T}_\beta}$, $1\leq \beta<\alpha$. If $\alpha$ is a limit ordinal, then for every $n\in\mathbb{N}$,
\[
R_{\alpha+n}^{p,0} = L_p^{n,0} + \langle T_I(R_\alpha^{p,0}) : I\in\mathcal{D}^n \rangle.
\]
We define the FDD $(X_\lambda)_{\lambda\in\mathcal{T}_{\alpha+n}}$ by letting one component be $L_p^{n,0}$ and taking the $2^n$ disjointly supported $1/2^n$-compressed copies $(T_I(X_\lambda))_{\lambda\in\mathcal{T}_\alpha}$, $I\in\mathcal{D}^n$, of $(X_\lambda)_{\lambda\in\mathcal{T}_\alpha}$. For $1\leq \alpha<\omega_1$, we formalize the index set $\mathcal{T}_\alpha$ of the FDD as an explicit tree of finite sequences of ordinal numbers and dyadic intervals, in accordance with the above construction. Each $X_\lambda$ is an isometric copy of some finite-dimensional space $L_p^{\kappa(\lambda),0}$ via a $\theta_\lambda$-compression operator $T_\lambda$ determined, e.g., by the entries of $\lambda$. Although a study of the subsets of this structure is not part of the present paper, the index set $\mathcal{T}_\alpha$ allows for the definition of concrete subspaces of $R_\alpha^{p,0}$ spanned by subsets of the FDD, which are themselves complemented in $L_p$. This raises the question of determining the cardinality of the isomorphism classes of spaces of this type, for example when $\omega^2\leq \alpha<\omega_1$, and in particular whether this cardinality equals the continuum, that is, the maximal possible among separable spaces (it was shown by Johnson and Schechtman in \cite{johnson:schechtman:2021} that the collection of distinct operator ideals in $L_p$ already attains the maximal cardinality $2^\mathfrak{c}$). Some tools relevant to this question have already been developed in \cite{mueller:1987}, \cite{alspach:1999}, \cite{alspach:tong:2003}, and \cite{alspach:tong:2006}.

The second contribution of this paper, encompassing Sections \ref{DES section} and \ref{dist repre subsection}, is the introduction of a class of distributional embeddings between the spaces $R_\alpha^{p,0}$, later used in the factorization of operators. Let $1\leq \alpha<\omega_1$. A distributional embedding scheme of $R_\alpha^{p,0}$ into $R_\beta^{p,0}$ is a family $\Psi = (\mathcal{M}_\lambda,(J_\mu)_{\mu\in\mathcal{M}_\lambda})_{\lambda\in\mathcal{T}_\alpha}$, where $(\mathcal{M}_\lambda)_{\lambda\in\mathcal{T}_\alpha}$ is a collection of disjoint antichains of $\mathcal{T}_\beta$ that respect the tree structure and, to some extent, the distributional structure of the FDD. For instance, for each $\lambda\in\mathcal{T}_\alpha$, the associated compressions $(T_\mu)_{\mu\in\mathcal{M}_\lambda}$ have pairwise disjointly supported ranges and the family $(J_\mu)_{\mu\in\mathcal{M}_\lambda}$ consists of distributional embeddings with common domain a finite-dimensional $L_p^{\kappa(\lambda),0}$ space. A substantial portion of the paper is devoted to proving that the operator
$T_\Psi : R_\alpha^{p,0}\to R_\beta^{p,0}$ given by
\[T_\Psi|_{X_\lambda} = \sum_{\mu\in\mathcal{M}_\lambda} T_\mu J_\mu T_\lambda^{-1}\]
is a distributional embedding. While the finite-dimensional components $T_\Psi|_{X_\lambda}$, $\lambda\in\mathcal{T}_\alpha$ are almost trivially distributional embeddings, the main difficulty lies in controlling how these pieces interact with one another. Proceeding by transfinite induction on $\alpha$, we decompose $T_\Psi$ into component operators determined by distributional embedding schemes on Bourgain-Rosenthal-Schechtman spaces on smaller ordinals. Each piece satisfies the hypothesis, and reassembling the pieces recovers $T_\Psi$ and establishes it as a distributional embedding. The inductive step differs for successor and limit ordinals $\alpha$. Distributional embedding schemes capture the strong reproducing properties of the decomposition $(X_\lambda)_{\lambda\in\mathcal{T}_\alpha}$ needed for the factorization of operators. It is plausible that the strategical reproducibility framework for bases due to Lechner, M\"uller, Schlumprecht, and the second author from \cite{lechner:motakis:mueller:schlumprecht:2020} extends to FDDs and applies to $(X_\lambda)_{\lambda\in\mathcal{T}_\alpha}$ for limit ordinals.

The final part of this paper, included in Section \ref{reduction to FDD}, applies the tools developed so far. A subspace $X$ of $L_p$ is called orthogonally complemented if the orthogonal projection $P:L_2\to \overline{X\cap L_2}^{\|\cdot\|_2}$ defines a bounded linear projection from $L_p$ onto $X$. For example, the Bourgain-Rosenthal-Schechtman spaces are of this type. If $X$ is such a subspace of $L_p$ and $T,S:X\to X$ are bounded linear operators, we say that $T$ is an orthogonal factor of $S$ with error $\epsilon$ if there exists a distributional embedding $J:X\to X$ such that $\|J^\dagger TJ - S\| < \epsilon$, where $J^\dagger:X\to X$ denotes the formal adjoint of $J$, which is well defined and bounded in this context. This can be viewed as an approximate orthogonal reduction of $T$ to $S$ in a way that respects the distributional structure. A bounded linear operator $D:R_\alpha^{p,0}\to R_\alpha^{p,0}$ is called scalar FDD-diagonal if for every $\lambda\in\mathcal{T}_\alpha$, $D|_{X_\lambda} = d_\lambda\cdot id$, for some scalar family $(d_\lambda)_{\lambda\in\mathcal{T}_\alpha}$, called the collection of entries of $D$.
\begin{thm*}
Let $\alpha$ be a limit countable ordinal number. For every bounded linear operator $T:R_\alpha^{p,0}\to R_\alpha^{p,0}$ and $\epsilon>0$ there exist a distributional embedding scheme $\Psi$ of $R_\alpha^{p,0}$ into itself and a scalar FDD-diagonal bounded linear operator $D:R_\alpha^{p,0}\to R_\alpha^{p,0}$ such that
\[\big\|T_\Psi^\dagger TT_\Psi - D\big\|<\epsilon.\]
Furthermore, the entries $(d_\lambda)_{\lambda\in\mathcal{T}_\alpha}$ of $D$ are averages of the diagonal entries of the matrix representation of $T$ with respect to the standard MDS basis of the space.
\end{thm*}
As an easy conclusion, we obtain that the limit spaces $R_\alpha^{p,0}$ with their standard MDS bases satisfy the factorization property. The proof of this theorem rests on an inductive construction of the distributional embedding scheme $\Psi$, i.e., a distributional ``reproduction'' of the FDD $(X_\lambda)_{\lambda\in\mathcal{T}_\alpha}$. We enumerate $\mathcal{T}_\alpha$ as a sequence $(\lambda_n)_{n=1}^\infty$ compatible with the tree structure, and at the $n$'th step we perform a probabilistic argument with components due to Lechner (\cite{lechner:2018:1-d}) and due to Lechner, M\"uller, Schlumprecht, and the second author (\cite{lechner:motakis:mueller:schlumprecht:2022} and \cite{lechner:motakis:mueller:schlumprecht:2023}) to choose appropriate $(\mathcal{M}_{\lambda_n},(J_\mu)_{\mu\in\mathcal{M}_{\lambda_n}})$ and $d_{\lambda_n}$. The fact that $d_{\lambda_n}$ is an average of the diagonal entries of $T$ is a regularity feature of the construction, suggesting that it is, in a certain sense, canonical and of low complexity. Similar probabilistic techniques appear in other papers, e.g., \cite{speckhofer:2025} and \cite{konstantos:speckhofer:2025}, but our argument is closer to \cite{konstantos:motakis:2025}. In a broader historical context, the reduction of a factorization problem on an infinite-dimensional Banach space to finite-dimensional pieces, often referred to as the localization method, goes back to Bourgain \cite{bourgain:1983}. The question of whether a cofinal subset of the $R_\alpha^{p,0}$ spaces satisfies the primary factorization property remains open. Nevertheless, some of the tools developed in Sections \ref{DES section}, \ref{dist repre subsection}, and \ref{reduction to FDD} are stated in a more general form than is necessary for the factorization property, so as to allow for a potential proof of the primary factorization property for the spaces $R_{\omega^\alpha}^{p,0}$, $1\leq \alpha<\omega_1$.

\section{Notation and concepts} 

In this section we introduce the standard terminology and notation used throughout the paper. We begin with basic probability theory and Banach space concepts and recall key notions from the theory of operator factorizations. We then outline the framework of orthogonally complemented subspaces and orthogonal factorizations in subspaces of $L_p$, and conclude by recalling the notation for the Haar system.

Given a probability space $(\Omega, \mathcal{A}, \mu)$, the characteristic function of a set $A \subset \Omega$ is denoted by $\chi_{A} : \Omega \to \{0,1\}$. Whenever we work within a probability space, all set relations are understood in the essential sense; for instance, $A \subset B$ means that $\mu(B \setminus A) = 0$. Given a $\sigma$-subalgebra $\mathcal{B}$ of $\mathcal{A}$, the conditional expectation of an integrable function $f : \Omega \to \mathbb{R}$ with respect to $\mathcal{B}$ is denoted by $\mathbb{E}(f | \mathcal{B})$. The support of $f$ is the set $\mathrm{supp}(f) = [f \neq 0]$. If $\Omega$ is endowed with a topology, its Borel $\sigma$-algebra is denoted $\mathcal{B}(\Omega)$. The Lebesgue measure of a measurable subset $A$ of $\mathbb{R}$ is denoted $|A|$. Throughout the paper, we fix $1 < p < \infty$ and denote by $q$ its conjugate exponent. An $L_p$ space refers to a space of the form $L_p(\Omega,\mathcal A,\mu)$ for some probability space $(\Omega,\mathcal A,\mu)$; for brevity we write $L_p(\mu)$. When the underlying probability space is the unit interval $[0,1]$ with its Borel $\sigma$-algebra and the Lebesgue measure, we simply write $L_p$. When multiple such spaces are considered simultaneously, we may denote them, for instance, as $L_{p}(\mu)$, $L_{p}(\nu)$, or $L_{p}(\xi)$. For $f \in L_{p}(\mu)$ and $g \in L_{q}(\mu)$, we write $\langle f, g \rangle = \int f g \, d\mu$.

For a subset $S$ of a linear (respectively, normed) space, we denote by $\langle S \rangle$ (respectively, $[S]$) its linear (respectively, closed linear) span. We write $\mathbb{N} = \{1, 2, \dots\}$ for the set of positive integers and $\mathbb{N}_0 = \{0\} \cup \N$. A sequence $(e_{n})_{n=1}^{\infty}$ in a Banach space $X$ is a Schauder basis if every $x \in X$ admits a unique representation $x = \sum_{n=1}^{\infty} a_{n} e_{n}$ for some scalar sequence $(a_{n})_{n=1}^{\infty}$. The corresponding sequence $(e_{n}^{*})_{n=1}^{\infty}$ in $X^{*}$, defined by $e_{n}^{*}(\sum_{m=1}^{\infty} a_{m} e_{m}) = a_{n}$, is called its biorthogonal sequence. A Schauder basis $(e_{n})$ is said to be $C$-unconditional for some $C \geq 1$ if, for all scalars $a_{1}, \ldots, a_{n}$ and $b_{1}, \ldots, b_{n}$ satisfying $|a_{k}| \le |b_{k}|$ for $1 \le k \le n$,  
\[\Big\| \sum_{k=1}^{n} a_{k} e_{k} \Big\| \leq C \Big\| \sum_{k=1}^{n} b_{k} e_{k} \Big\|.\]
A basis that is $C$-unconditional for some $C \geq 1$ is simply called unconditional. A Schauder decomposition of a Banach space $X$ is a sequence of closed subspaces $(E_{n})_{n=1}^{\infty}$ such that every $x \in X$ can be uniquely written as $x = \sum_{n=1}^{\infty} x_{n}$ with $x_{n} \in E_{n}$. A Schauder decomposition induces a sequence of bounded linear projections $(P_n)_{n=1}^\infty$ defined by $P_{n}(\sum_{m=1}^{\infty} x_m) = x_{n}$. If each $E_{n}$ is finite-dimensional, $(E_{n})$ is called a finite-dimensional decomposition (FDD). If every sequence of non-zero vectors $(x_{n})_{n=1}^{\infty}$ with $x_{n} \in E_{n}$ forms an unconditional basis of its closed linear span, the decomposition is called an unconditional Schauder decomposition. 

Let $X$ be a Banach space with a Schauder basis $(e_n)_{n=1}^\infty$. We identify a bounded linear operator $T : X \to X$ with its matrix representation $(e_m^*(T e_n))_{m,n=1}^\infty$. The operator $T$ is said to be diagonal with respect to $(e_n)_{n=1}^\infty$ (or simply diagonal when the basis is understood) if $e_m^*(T e_n) = 0$ for all $m \ne n$. Equivalently, each $e_n$ is an eigenvector of $T$ corresponding to some eigenvalue $\lambda_n$. Likewise, if $X$ admits an FDD $(E_n)_{n=1}^\infty$, we call a bounded linear operator $T : X \to X$ diagonal with respect to $(E_n)_{n=1}^\infty$ (or FDD-diagonal when the decomposition is understood) if $P_m T P_n = 0$ for all $m \ne n$. If each $E_n$ is an eigenspace of $T$ corresponding to some eigenvalue $\lambda_n$, we say that $T$ is scalar diagonal with respect to $(E_n)_{n=1}^\infty$ (or simply scalar FDD-diagonal when the FDD is understood) and we call $(\lambda_n)_{n=1}^\infty$ the entries of $T$.

\subsection{Factorization of operators}
We recall the notions of factor and projectional factor of an operator, note their distinction, and prepare for the discussion of orthogonal factors in the next subsection. We also recall the factorization property for Banach spaces with a Schauder basis.

\begin{dfn}
Let $X$ be a Banach space, $T,S:X\to X$ be bounded linear operators, and $K>0$.
\begin{enumerate}[label=(\alph*)]
    
    \item We call $T$ a $K$-factor of $S$ if there exist bounded linear operators $A,B:X\to X$ such that
    \[S = ATB\text{ and }\|A\|\|B\|\leq K.\]
    When we do not need to be specific about the constant $K$, then we call $T$ a factor of $S$.

    \item We call $T$ a $K$-projectional factor of $S$ is there exist bounded linear operators $A,B:X\to X$ such that
    \[S = ATB,\;\|A\|\|B\|\leq K,\text{ and }AB = id:X\to X.\]
    When we do not need to be specific about the constant $K$, then we call $T$ a projectional factor of $S$.
    
\end{enumerate}
\end{dfn}

\begin{rem}
An operator $T:X\to X$ is a factor of the identity operator $id:X\to X$ if and only if there exist projections $P:X\to Y$ and $Q:X\to Z$, with images isomorphic to $X$, such that $QT|_Y:Y\to Z$ is an isomorphism. By contrast, $T$ is a projectional factor of the identity operator precisely when there exists a projection $P:X\to Y$, with image isomorphic to $X$, such that $PT|_Y=id:Y\to Y$. More generally, if $T$ is a projectional factor of $S$, then there exists a decomposition $X = Y\oplus Z$ such that $X\simeq Y$ and the $2\times 2$ matrix representation of $T$ with respect to this decomposition contains a matrix similar to $S$ in the upper left entry. Thus, $S$ may be viewed as a ``simpler'' operator than $T$, with $T$ effectively ``reduced'' to $S$. This highlights why the latter property is strictly stronger. Projectional factorizations enjoy stronger properties and are employed in stepwise reductions that lead to clean factorization results often related to decompositions of Banach spaces (see, e.g., \cite{lechner:motakis:mueller:schlumprecht:2022, lechner:motakis:mueller:schlumprecht:2023}).  

In this paper we impose additional structure, further tightening the notion of a projectional factor: the decomposition $X=Y\oplus Z$ must be orthogonal, and the similarity between $T$ and $S$ is witnessed by a distributional isomorphism between $X$ and $Y$. These notions are well defined for subspaces of $L_p$-spaces. This refinement grants additional properties to the factorization, notably the automatic continuity of certain operators. One of our main results is that every bounded linear operator on a limit space $R_\alpha^{p,0}$ ``orthogonally reduces'' to a scalar FDD-diagonal operator. The necessary, though elementary, theory to define orthogonal factorizations is developed in Section~\ref{orthocomplemented section}.
\end{rem}

The factorization property, formally introduced in \cite{laustsen:lechner:mueller:2015}, has driven the study of reproducing properties of bases in Banach spaces and has led to further developments in the theory of $L_p$ and other classical Banach spaces (see, e.g., \cite{lechner:motakis:mueller:schlumprecht:2021}, \cite{lechner:motakis:mueller:schlumprecht:2022}, \cite{lechner:motakis:mueller:schlumprecht:2023}, \cite{lechner:speckhofer:2025}, \cite{konstantos:motakis:2025}). A corollary of our main theorem is the factorization property of the standard basis of the limit $R_\alpha^{p,0}$ spaces.

\begin{dfn} \label{def of fact property} Let $X$ be a Banach space with a basis $(e_{n})_{n=1}^{\infty}$. We say that the basis $(e_{n})_{n=1}^{\infty}$ of the space $X$ has the factorization property if every operator $T \colon X \to X$ with large diagonal with respect to the basis $(e_{n})_{n=1}^{\infty}$ of $X$, that is for some $\delta>0$, we have $\inf_{n\in\mathbb{N}}|e_n^*(Te_n)| \geq\delta$, is a factor of the identity operator $id$ on $X$.
\end{dfn}

\subsection{Orthogonally complemented spaces and orthogonal factors}
\label{orthocomplemented section}

A subspace of a space $L_p(\mu)$ is called orthogonally complemented if it is the image of a bounded linear projection stemming from an orthogonal projection on $L_2(\mu)$. To our knowledge, this concept was initially presented in \cite{alspach:1999} for $p>2$, but it found broader application in \cite{konstantos:motakis:2025}, where it was used to create the suitable framework for examining the factorization properties of operators on the Rosenthal space, $X_p$, and the first infinite-dimensional Bourgain-Rosenthal-Schechtman space, $R_\omega^p$. The key term of orthogonal factors of operators is primarily modeled on projectional factors from \cite{lechner:motakis:mueller:schlumprecht:2022}. We revisit the concepts outlined in \cite{konstantos:motakis:2025} with some extra level of detail. While the material discussed here is rather elementary compared to the remainder of the paper, we rely on it to factorize operators.

\begin{dfn}
\label{def of a pc projection and orthog comple space}
Consider the space $L_p(\mu)$ and let $C\geq 1$.
\begin{enumerate}[label=(\alph*)]
    
    \item An orthogonal projection $P:L_2(\mu)\to L_2(\mu)$ is called $p$-$C$-bounded if, for every $f\in L_2(\mu)\cap L_p(\mu)$, $\|Pf\|_p\leq C\|f\|_p$.

    \item A subspace $X$ of $L_p(\mu)$ is called $C$-orthogonally complemented if there exists a $p$-$C$-bounded orthogonal projection $P:L_2(\mu)\to L_2(\mu)$ such that $X = \overline{P(L_2(\mu)\cap L_p(\mu))}^{\|\cdot\|_p}$.

\end{enumerate}
\end{dfn}

\begin{ntt}
Consider the space $L_p(\mu)$, $C\geq 1$, let $P:L_2(\mu)\to L_2(\mu)$ be a $p$-$C$-bounded orthogonal projection, and denote $X = \overline{P(L_2(\mu)\cap L_p(\mu))}^{\|\cdot\|_p}$.

\begin{enumerate}[label=(\alph*)]

\item By the density of $L_2(\mu)\cap L_p(\mu)$ in $L_p(\mu)$, there exists a unique bounded linear projection
\[P_X:L_p(\mu)\to L_p(\mu)\]
such that $P_X|_{L_2(\mu)\cap L_p(\mu)} = P|_{L_2(\mu)\cap L_p(\mu)}$. Furthermore, $P_X(L_p(\mu)) =X$ and $\|P_X\|\leq C$.

\item The self-adjointness of $P$ implies that it is $q$-$C$-bounded. We denote
\[X' = \overline{P(L_2(\mu)\cap L_q(\mu))}^{\|\cdot\|_q}\]
and $P_{X'}:L_q(\mu)\to L_q(\mu)$ the projection onto $X'$ stemming from $P$.

\end{enumerate}
\end{ntt}

\begin{rem}
In the context of the above notation, the following hold.
\begin{enumerate}[label=(\alph*)]
    
    \item The projection $P_{X'}:L_q(\mu)\to L_q(\mu)$ is the Banach space adjoint of $P_X:L_p(\mu)\to L_p(\mu)$. Consequently, the space $X'$ is $C$-isomorphic to $X^*$ via the linear bijection defined by $f\mapsto \langle f,\cdot\rangle$, which satisfies $C^{-1}\|f\|_q \leq \|f\|_{X^*}\leq \|f\|_q$.

    \item Because the intersection $X\cap X'$ contains $P(L_\infty(\mu))$, it is dense in both $X$ and $X'$ with respect to the norms $\|\cdot\|_p$ and $\|\cdot\|_q$, respectively.

\end{enumerate}
\end{rem}

Recall that a sequence $(f_n)_{n=1}^\infty$ in a space $L_p(\mu)$ is called a martingale difference sequence if, for all $n\geq 2$, $\mathbb E(f_n|f_1,\ldots,f_{n-1}) = 0$. When all its terms are non-zero, by the non-expansiveness of the conditional expectation operator, $(f_n)_{n=1}^\infty$ is a monotone Schauder basic sequence. Moreover, by the Burkholder inequality (\cite[Theorem  9]{burkholder:1966}), it is unconditional.

\begin{exa}
\label{mds orthocomplemented}
Consider the martingale difference sequence $(f_n)_{n=1}^\infty$ with non-zero terms in a space $L_2(\mu)$.
\begin{enumerate}[label=(\alph*)]
    
    \item The sequence $(\|f_n\|_2^{-1}f_n)_{n=1}^\infty$ is biorthogonal in $L_2(\mu)$. It yields the orthogonal projection
    \[Pf = \sum_{n=1}^\infty \|f_n\|_2^{-2}\langle f_n,f\rangle f_n.\]
    
\end{enumerate}
Assume further that $(f_n)_{n=1}^\infty$ is symmetric and $\{-1,0,1\}$-valued and let $X = [(f_n)_{n=1}^\infty]$ in $L_p(\mu)$.
\begin{enumerate}[resume,label=(\alph*)]

    \item For all $n\in\N$, $\|f_n\|_p\|f_n\|_q = \|f_n\|_2^2$ (see \cite[page 280]{rosenthal:1970:Xp}).
    
    \item For $C\geq 1$, $P$ is $p$-$C$-bounded, and thus $X$ is $C$-orthogonally complemented, if and only if $(\|f_n\|_q^{-1}f_n)_{n=1}^\infty$ is $C$-equivalent to the biorthogonal sequence of $(\|f_n\|_p^{-1}f_n)_{n=1}^\infty$.
    
    \item If the latter is indeed the case, then
    \[P_Xf = \sum_{n=1}^\infty\|f_n\|_q^{-1}\langle f_n,f\rangle \|f_n\|_p^{-1}f_n,\]
    where the above series convergences with respect to $\|\cdot\|_p$, and $X' = [(f_n)_{n=1}^\infty]$ in $L_q(\mu)$.
    
\end{enumerate}
Concrete such examples in $L_p$ are the closed linear span of the radermacher sequence (by the Khintchine inequalities) or any subsequence of the standard Haar basis (by unconditionality), and the $R_\alpha^p$ spaces, as proved in \cite{bourgain:rosenthal:schechtman:1981}.
\end{exa}

Recall that for a real-valued random variable $f$ on a probability space $(\Omega,\mathcal{A},\mu)$, the distribution of $f$, denoted $\mathrm{dist}(f)$, is the probability measure on the Borel sets of $\mathbb{R}$ defined by $(\mathrm{dist}(f))(A) = \mu([f\in A])$.

\begin{dfn}
Let $X$, $Y$ be subspaces of some $L_p$-spaces. A linear operator $T:X\to Y$ is called a distributional embedding if, for all $f\in X$, $\mathrm{dist}(Tf) = \mathrm{dist}(f)$. If $T$ is additionally onto then it is called a distributional isomorphism, and we write $X\equiv^\mathrm{dist} Y$.
\end{dfn}

A distributional embedding is an isometry, but it is much more than that. When the domain of a distributional embedding is not closed, its extension to the closure remains a distributional embedding. This can be proved through the characteristic function of random variables.

\begin{exa}
\label{example for compression part a }
If $\mathcal{A}$ is a $\sigma$-subalgebra of $\mathcal{B}{[0,1]}$ and $\varphi \colon ([0,1], \mathcal{A}) \to ([0,1], \mathcal{B}([0,1]))$ is measurable and measure-preserving, then the Koopman operator of $\varphi$, $T_{\varphi} \colon L_{p} \to L_{p}$ defined by $T_{\varphi} (f) = f \circ \varphi$ is a distributional embedding. 
\end{exa}

\begin{ntt}
\label{formal adjoint}
Consider the spaces $L_p(\mu)$, $L_p(\nu)$, a $C$-orthogonally complemented subspace $X$ of $L_p(\mu)$, and a distributional embedding $T:X\to L_p(\nu)$. We construct the formal adjoint of $T$,
\[T^\dagger:L_p(\nu)\to X,\]
as follows. Initially, we extend $T|_{X\cap X'}$ to a distributional embedding $T:X'\to L_q(\nu)$, which by slightly abusing notation, we also denoted $T$. We then identify $L_q(\nu)^*$ with $L_p(\nu)$ and $X^*$ with $X'$ and let $T^\dagger$ be the Banach-space adjoint of $T:X'\to L_q(\nu)$. Because the identification $X'\simeq X^*$ is not isometric, $\|T^\dagger\|\leq C$.
\end{ntt}

\begin{rem}
By initiating with the operator $T:X'\to L_q(\nu)$, we analogously construct its formal adjoint $T^\dagger:L_q(\nu)\to X'$. By definition, for all $f\in L_p(\nu),g\in X',u\in X,v\in L_q(\nu)$, considering the operators $T$ and $T^\dagger$ with appropriate domains,
    \[\langle g,T^\dagger f\rangle = \langle Tg,f\rangle\text{ and }\langle u,T^\dagger v\rangle = \langle Tu,v\rangle.\]
    In particular, the two distinct versions of $T^\dagger$ coincide on $L_p(\nu)\cap L_q(\nu)$, and thus
    \[T^\dagger(L_p(\nu)\cap L_q(\nu))\subset X\cap X'.\]
\end{rem}

\begin{prop}
\label{distro copy of orthocomplemented is orthocomplemented and others}
Consider the spaces $L_p(\mu)$ and $L_p(\nu)$. Let $X$ be a $C$-orthogonally complemented subspace of $L_p(\mu)$ and $T:X\to L_p(\nu)$ be a distributional embedding. Denoting $T^\dagger:L_p(\nu)\to X$ its formal adjoint, the following hold.
\begin{enumerate}[label=(\alph*)]
   
    \item\label{distro copy of orthocomplemented is orthocomplemented and others a} $T^\dagger T = id:X\to X$

    \item\label{distro copy of orthocomplemented is orthocomplemented and others b} $T T^\dagger:L_p(\nu)\to L_p(\nu)$ is a linear projection of norm at most $C$ onto $T(X)$ witnessing the $C$-orthogonal complementation of $T(X)$ in $L_p(\nu)$.
    
\end{enumerate}
In particular, $C$-orthogonal complementability is preserved under distributional isomorphisms.
\end{prop}

\begin{proof}
The first statement is shown as follows: for $f\in X$, $g\in X'$, we have $\langle g,f\rangle = \langle Tg,Tf\rangle = \langle g, T^\dagger Tf\rangle$. Therefore, $T^\dagger T=id:X\to X$. This yields that $TT^\dagger$ is a projection onto $T(X)$. It remains to show that $TT^\dagger|_{L_2(\nu)\cap L_p(\nu)}$ has norm one with respect to $\|\cdot\|_2$, because in Hilbert spaces norm-one projections are automatically orthogonal (see, e.g., \cite[Chapter 3 \S III, Proposition 3.3]{conway:1990}). For $f\in L_2(\nu)\cap L_p(\nu) = L_q(\nu)\cap L_p(\nu)$, we have $T^\dagger f\in X\cap X'\subset L_2(\mu)$, and therefore,
\[\|TT^\dagger f\|_2^2 = \|T^\dagger f\|_2^2 = \langle T^\dagger f,T^\dagger f\rangle = \langle f,TT^\dagger f\rangle\leq \|f\|_2\|TT^\dagger f\|_2.\]
We deduce $\|TT^\dagger f\|_2\leq \|f\|_2$.
\end{proof}

\begin{exa} \label{Eaxmple that gives formula of projection}
Consider the symmetric $\{-1,0,1\}$-valued martingale difference sequence $(f_n)_{n=1}^\infty$ with non-zero terms in a space $L_p(\mu)$ and assume that $X = [(f_n)_{n=1}^\infty]$ is $C$-orthogonally complemented. Let $T:X\to L_p(\nu)$ be a distributional embedding into some space $L_p(\nu)$ and, for $n\in\mathbb N$, denote $b_n = Tf_n$. Then, for all $f\in X$ and $g\in L_p(\nu)$,
\begin{align*}
Tf = \sum_{n=1}^\infty \|f_n\|_q^{-1}\langle f_n,f\rangle \|f_n\|_p^{-1}b_n\text{ and }
T^\dagger g = \sum_{n=1}^\infty \|f_n\|_q^{-1}\langle b_n,g\rangle \|f_n\|_p^{-1}f_n.
\end{align*}
In particular,
\[TT^\dagger g = \sum_{n=1}^\infty \|f_n\|_q^{-1}\langle b_n,g\rangle \|f_n\|_p^{-1}b_n\]
witnesses the $C$-orthogonal complementation on $T(X)$.
\end{exa}

\begin{rem}
\label{formal adjoint anti-automorphism}
The formal adjoint acts as an anti-automorphism. That is, if $X$ is $C$-orthogonally complemented in $L_p(\mu)$ and $T:X\to L_p(\nu)$, $S:T[X]\to L_p(\xi)$ are distributional embeddings, we have $(ST)^\dagger = T^\dagger S^\dagger$.
\end{rem}

With the necessary theory in place, we give the definition of approximate orthogonal factors of operators that is used in our main theorem.

\begin{dfn} \label{defin of orthog factor}
Let $X$, $Y$ be subspaces of some spaces $L_p(\mu)$ and $L_p(\nu)$, respectively, and assume that $Y$ is a $C$-orthogonally complemented. Let $R \colon X \to X$ and $S \colon Y \to Y$ be bounded linear operators. For $\epsilon>0$, we say that $R$ is an orthogonal factor of $S$ with error $\epsilon$ if there exists a distributional embedding $T \colon Y \to  L_p(\mu)$ with $T(Y)\subset X$ such that
\[\| T^\dagger R T -S \| \leq \epsilon.\]
If the above holds for $\epsilon = 0$, then we call $R$ an orthogonal factor of $S$.
\end{dfn}

When $R$ is an orthogonal factor of $S$, this should be understood as a reduction of $R$ to $S$, which in applications is typically the simpler operator. The following remark clarifies that such reductions can be iterated while retaining control over the error at each step.

\begin{rem}
It follows from Remark \ref{formal adjoint anti-automorphism} that being an approximate orthogonal factor is a transitive property, which means the following: Let $X$, $Y$, and $Z$ be subspaces of some $L_p$-spaces, and assume that $Y$ and $Z$ are $C$ and $D$-orthogonally complemented. Let $R:X\to X$, $S:Y\to Y$, and $U:Z\to Z$ be bounded linear operators, and assume that $R$ is an orthogonal factor of $S$ with error $\epsilon$ and $S$ is an orthogonal factor of $U$ with error $\delta$. Then $R$ is an orthogonal factor of $U$ with error $D\epsilon+\delta$.
\end{rem}

\begin{rem}
By Proposition \ref{distro copy of orthocomplemented is orthocomplemented and others} \ref{distro copy of orthocomplemented is orthocomplemented and others a}, being an orthogonal factor showcases symmetry around the identity operator, which means the following: Let $X$, $Y$ be subspaces of some $L_p$-spaces, and assume that $Y$ is $C$-orthogonally complemented. Let $R \colon X \to X$ and $S \colon Y \to Y$ be bounded linear operators. If $R$ is an orthogonal factor of $S$ with error $\epsilon$, then $id - R:X\to X$ is an orthogonal factor of $id - S:Y\to Y$ with the same error $\epsilon$.
\end{rem}

We state a general result on the factorization property of orthogonally complemented symmetric three-valued martingale difference sequences, formulated so as to apply in our context. In view of Theorem \ref{reduction to FDD diagonal}, it yields the factorization property of the standard basis of the limit spaces $R_\alpha^{p,0}$, that is, Theorem \ref{fact property of BRS spaces}.

\begin{prop} \label{Prop for factorization propery}   Let $X = [(f_{n})_{n=1}^{\infty}] \leq L_{p}(\mu)$, where $(f_{n})_{n=1}^{\infty}$ is a $\lbrace -1,0,1 \rbrace$-valued symmetric martingale difference sequence and let $C \geq 1$. Assume that  
\begin{enumerate}[label=(\alph*)]
\item $X$ is $C$-orthogonally complemented, and

\item for every $\eta>0$ and operator $T \colon X  \to X$, there exists a diagonal operator $R \colon X \to X$ with diagonal entries in $\mathrm{conv}(f_n^*(Tf_n))$ such that $T$ is an orthogonal factor of $R = id$ with error $\eta$.
\end{enumerate}
Then, for every $0 < \epsilon <1$, every operator $T$ with $\delta$-large diagonal, i.e., 
\begin{align*} \label{for delta}
\inf_{n \in \mathbb{N}} \Big| \|f_{n}\|_{2}^{-2} \langle f_{n}, T f_{n} \rangle \Big| \geq \delta,
\end{align*}
is a $C(p^*-1)(1-\epsilon)^{-1}\delta^{-1}$-factor of the identity. In particular, the basis  $(f_{n})_{n=1}^{\infty}$ of $X$ has the factorization property.
\end{prop}

\begin{proof}
Define the multiplication operator $M:X\to X$ given by
\begin{align*}
 f_{n} \mapsto \Big( \| f_{n} \|_{2}^{-2} \langle f_{n}, T f_{n} \rangle \Big)^{-1}  f_{n}, \;\;\; n \in \mathbb{N}.
 \end{align*}
Because every martigale difference sequence in $L_p$ is $(p^*-1)$-unconditional, $\|M\|\leq (p^*-1)\delta^{-1}$ and observe that the diagonal entries of $TM$ are all equal to one. From the assumptions, we have that there exists a diagonal operator $R \colon X \to X$ with diagonal entries one, i.e. $R=\mathrm{id}$, such that $TM$ is an orthogonal factor of $id$ with error $\epsilon$. Thus, there exists a distributional embedding $j \colon X \to L_{p}(\mu)$, with $j(X) \subset X$, such that
\begin{align*}
\|j^{\dagger}  TM j - id\| < \epsilon.
\end{align*}
Thus, the operator $j^{\dagger}  TM j$ is invertible with norm at most $(1-\epsilon)^{-1}$. Now, letting
\begin{align*}
L= \Big( j^{\dagger}  TM j \Big)^{-1} j^{\dagger}  \;\; \text{and} \;\; R = M j
\end{align*}
we have $LTR = id$ and
\begin{align*}
\| L \| \cdot \| R \| \leq \frac{C(p^{\ast}-1)}{(1 - \epsilon) \delta}.   
\end{align*}
\end{proof}

\subsection{Haar system notation}

In this subsection we recall basic notation for the standard Haar system. Although this system is not a basis of the ${R}_\alpha^{p,0}$ spaces, the canonical basis of each such space can be blocked into an FDD whose components are spanned by compressed copies of finite initial segments of the Haar system. For this reason, we will use Haar-related notation extensively.

We enumerate the Haar basis over the dyadic intervals of $[0,1]$, ordered lexicographically.

\begin{ntt}
\label{haar system notation}
We establish the following notation used in handling the Haar system.
\begin{enumerate}[label=(\alph*)]

\item We denote by $\mathcal{D}$ the set of all dyadic intervals $I$ of $[0,1]$, i.e., of the form
\[I = \left[ \frac{i-1}{2^{n}}, \frac{i}{2^{n}} \right),\]
where $n \in \mathbb{N}_{0}$ and $1 \leq i \leq 2^{n}$.

\item For $n\in\mathbb{N}_0$ we denote
\begin{align*}
    \mathcal{D}^{n} &= \lbrace I \in \mathcal{D}: \vert I \vert = 1/2^{n} \rbrace = \left\{\left[\frac{i-1}{2^n},\frac{i}{2^n}\right):1\leq i\leq 2^n\right\}\text{ and}\\
    \mathcal{D}_n & = \bigcup_{k=0} ^{n}\mathcal{D}^{k}.
\end{align*}

\item We enumerate $\mathcal{D}$ in the lexicographical order as follows. Let $I\in\mathcal{D}^m$ and $J\in\mathcal{D}^n$. Then $I<J$ means either $m<n$, or both $m=n$ and $\min(I)<\min(J)$. That is, $[0,1)<[0,1/2)<[1/2,1)<[0,1/4)<[1/4,1/2)<[1/2,3/4)<[3/4,1)<\ldots$.

\item Let $n\in\mathbb{N}_0$ and $I = [(i-1)/2^n,i/2^n)\in\mathcal{D}^n$. We denote the left and right half of $I$ in $\mathcal{D}^{n+1}$ by
\[I^{+} = \left[ \frac{2(i-1)}{2^{n+1}}, \frac{2i-1}{2^{n+1}} \right)\text{ and }I^{-} = \left[ \frac{2i-1}{2^{n+1}}, \frac{2i}{2^{n+1}} \right).\]
We call $I$ the immediate predecessor of $I^+$ and $I^-$ and write
\[\pi(I^+) = I\text{ and }\pi(I^-) = I.\]
In this context, the term ``predecessor'' refers to an implicit binary tree structure of $\mathcal D$, rather than the lexicographical order. We define the signs of $I^+$ and $I^-$ as 
\[\epsilon(I^+) = 1\text{ and }\epsilon(I^-) = -1.\]

\end{enumerate}
\end{ntt}

\begin{ntt}
\label{basic ntt for fd Lp spaces}
We recall the standard Haar system and remind its relation to standard subspaces of $L_p$.
\begin{enumerate}[label=(\alph*)]

\item For $I\in\mathcal{D}$, we define
\[h_I = \chi_{I^{+}} - \chi_{I^{-}}.\]
The collection $(\chi_{[0,1]}) \cup  (h_{I})_{I \in \mathcal{D}}$ is the standard Haar system, and by the Burkholder inequality, it forms an unconditional basis for $L_p$.

\item We denote
\[L_{p}^{0} = \lbrace f \in L_{p}: \mathbb{E}(f) = 0 \rbrace =  [ (h_{I})_{I \in \mathcal{D}} ].\]

\item For $n\in\mathbb{N}$ denote
\[L_{p}^{n} = \langle (\chi_{I})_{I \in \mathcal{D}^{n}} \rangle = \langle (\chi_{[0,1]}) \cup  (h_{I})_{I \in \mathcal{D}_{n-1}} \rangle\]
and
\[L_{p}^{n,0} = \lbrace f \in L_{p}^{n}: \mathbb{E}(f) = 0 \rbrace = \langle (h_{I})_{I \in \mathcal{D}_{n-1}} \rangle.\]

\end{enumerate}
\end{ntt}

\section{The Bourgain-Rosenthal-Schechtman  $R_{\alpha}^{p}$ spaces } 

\subsection{The $R_{\alpha}^{p}$ spaces}

In this subsection we recall the original definition of the Bourgain-Rosenthal-Schechtman spaces $R_{\alpha}^{p}$ and their orthogonal complementation from \cite{bourgain:rosenthal:schechtman:1981}. Each space $R_\alpha^p$ was originally defined as a subspace of $L_p(\{0,1\}^{T_\alpha})$, for some countable index set $T_\alpha$ depending on $\alpha$. While we briefly revisit this approach here, in the next section we present an equivalent formulation in which these spaces are realized as subspaces of $L_p = L_p[0,1]$.  

In the following definition we recall the notions of the disjoint sum and the independent sum.

\begin{dfn}~
\label{sum} 
\begin{enumerate}[label=(\alph*)]

\item For a subspace $B$ of $L_{p}(\Omega, \mathcal{A}, \mu)$, the disjoint sum $ (B  \oplus B)_{p}$ of $B$ is the  subspace
\[ 
\begin{split}
\Big\{  b(\omega, \epsilon)\in L_p(\Omega\times\{0,1\}) : &\text{ there exists }    b_{\epsilon} \in B \text{ with }  b(\omega, \epsilon) =   b_{\epsilon} (\omega)\\
&\text{for all } \omega \in \Omega, \;\epsilon = 0 \text{ or } 1 \Big\} ,
\end{split}
\]
of $L_p(\Omega\times\{0,1\})$, where $\lbrace 0,1 \rbrace$ is equipped with the uniform probability measure.

\item For every $n\in\mathbb{N}$, let $B_{n}$ be a subspace of $L_{p}(\Omega_{n}, \mathcal{A}_{n}, \mu_{n})$. We define the independent sum of $B_{n}$ as follows. For every $n$, let
\[ \begin{split} 
\overline{B}_{n} = \Bigg\{ b \in L_{p} \left( \prod_{n} \Omega_{n}  \right) :\text{ there exists } f \in B_{n} \text{ with } b(\omega) = f(\omega_{n}) 
\text{ for all } \omega \in \prod_{n} \Omega_{n} \Bigg\},
\end{split}
\]
of $L_{p} \left(  \prod_{n} \Omega_{n}  \right)$. The subspace $\overline{B}_{n}$ is a  distributional copy of $B_{n}$ that depends only on the $n$-th coordinate. Then the independent sum of  $(B_{n})_{n=1}^\infty$ is defined as
\[\Big( \sum_{n} B_{n} \Big) _{\text{Ind},p} = \big[\cup_{n} \overline{B}_{n}\big],\]
in $L_{p} \left(  \prod_{n} \Omega_{n}  \right)$.
\end{enumerate}
\end{dfn}

The spaces $R_{\alpha}^{p}$ are defined recursively on the countable ordinals. The first Bourgain-Rosenthal-Schechtman space is the space of constant functions, while the successor and limit steps are carried out using the disjoint and independent sums, respectively.

\begin{dfn} \label{defin of BRS Spaces} Define $R_{0}^{p} = \langle \chi_{\{0,1\}} \rangle$ in $L_{p}\lbrace 0,1 \rbrace$. Let $\alpha$ be an ordinal with $0< \alpha < \omega_{1}$. Suppose that we have defined the $R_{\beta}^{p}$ space  for all $\beta < \alpha$.
\begin{enumerate}[label=(\alph*)]

\item \label{defin of BRS Spaces, suc} If $\alpha = \beta + 1$, then $R_{\alpha}^{p}  = ( R_{\beta}^{p} \oplus R_{\beta}^{p})_{p}$.

\item \label{defin of BRS Spaces, lim} If $\alpha$ is a limit ordinal, then $  R_{\alpha}^{p} =  ( \sum_{\beta < \alpha} R_{\beta}^{p} ) _{\text{Ind},p}$.
\end{enumerate}
\end{dfn}

The orthogonal complementation of the $R_{\alpha}^{p}$ spaces is established implicitly in \cite{bourgain:rosenthal:schechtman:1981}, with the steps made explicit in \cite[Theorem 4.6]{konstantos:motakis:2025}.

\begin{thm}\label{Orth Compl spaces the BRS} Let $\alpha < \omega_{1}$. The $R_{\alpha}^{p}$ space is $(p^{\ast} -1)^{2} (p^{\ast}/2)^{3/2}$-orthogonally complemented. 
\end{thm}

\subsection{The  $R_{\alpha}^{p}$ spaces as subspaces of $L_{p}[0,1]$ }

In this subsection we give an equivalent definition of the $R_{\alpha}^{p}$ spaces as subspaces of $L_p = L_p[0,1]$. While this is not technically essential, it provided the authors with clearer intuition and oversight into this class of spaces. In this framework, disjoint sums are described in terms of compressions with disjointly supported ranges, and independent sums in terms of images under distributional embeddings with independent ranges.

The notion of a $\theta$-compression is central to this paper and will be used repeatedly throughout. It is essentially the same as the classical $\theta$-dilation associated with the Boyd indices of a rearrangement-invariant space. However, contrary to standard convention, we adopt the term ``compression'' since we restrict to values $\theta\in(0,1]$.

\begin{dfn} \label{comp dist embed}
Let $1<p < \infty$, $0 < \theta \leq 1$ and $X$, $Y$ be subspaces of $L_{p}$. A linear operator $T \colon X \to Y$ is called a $\theta$-compression if for every, $f \in X$ and $A \in \mathcal{B}(\mathbb{R})$,
\begin{align*}
\vert [Tf \in A] \vert = \theta \vert [f \in A] \vert + (1 - \theta)\chi_{A}(0), 
\end{align*}
i.e., $\mathrm{dist}(Tf) = \theta\mathrm{dist}(f) + (1-\theta)\delta_0$.
\end{dfn}

\begin{rem} Let $0 < \theta \leq 1$, $0< \eta \leq 1$, and $X$, $Y$, $Z$ be subspaces of $L_{p}$.
\begin{enumerate}[label=(\alph*)]
\item If $T \colon X \to Y$ is a $\theta$-compression and $S \colon Y \to Z$ is a $\eta$-compression, then $TS \colon X \to Z$ is a $\theta \eta$-compression. 
\item If $T \colon X \to Y$ is a $\theta$-compression, then for every $f \in X$ we have
\[\| Tf \|_{p} = \Big(\int_0^\infty |[|Tf|>t^{1/p}]|dt\Big)^{1/p} = \theta^{1/p} \| f \|_{p}.\]
\item If $T \colon X \to Y$ is a $1$-compression, then $T$ is a distributional embedding.
\end{enumerate}
\end{rem}

The following piece of notation will be used extensively to define and study the Bourgain-Rosenthal-Schechtman spaces as subspace of $L_p = L_p[0,1]$.

\begin{ntt}
\label{example for compression part b}
Let $0 \leq a < b \leq 1$ and $I = [a,b)$. Let $T_{I} \colon L_{p} \to L_{p}$ be defined by 
\begin{equation*}
T_{I}(f)(t) = 
    \begin{cases}
        f(\frac{t-a}{b-a}), &  a \leq t < b \\
        0, & \text{otherwise}.
    \end{cases}
\end{equation*}
Evidently, $T_I$ is a $(b-a)$-compression.
\end{ntt}

Recall that set relations and operations should be interpreted in the essential sense. 

\begin{rem}
\label{true form of compression}
Let $X$ be a subspace of $L_p$,  $0<\theta\leq 1$, and $T:X\to L_p$ be a $\theta$-compression.
For $f,g\in X$, we have that
    \begin{enumerate}[label=(\alph*)]
    
        \item $\supp(g)\subset\mathrm{supp}(f)\text{ if and only if }\supp(Tg)\subset\mathrm{supp}(Tf)$ and

        \item $\supp(g)\cap \mathrm{supp}(f) = \emptyset\text{ if and only if }\supp(Tg)\cap\mathrm{supp}(Tf) = \emptyset$.
    
    \end{enumerate}
The first assertion arises from  the linearity of $T$ and the characterization
\[\supp(g)\subset\mathrm{supp}(f)\text{ if and only if }\lim_{r\to\infty}\big|[rf+g\neq 0]\big| = |[f\neq 0]|.\]
The second assertion is simpler.

It can be shown that there exists a $\sigma$-subalgebra $\mathcal{A}$ of $\mathcal{B}[0,1]$ such that $X\subseteq L_p([0,1],\mathcal{A})$ and a $\theta$-compression $\widetilde T:L_p([0,1],\mathcal{A})\to L_p$ with $\widetilde T|_X = T$. In particular, for each $A\in\mathcal{A}$, $\widetilde T\chi_A = \chi_B$ for some measurable $B$ with $|B|=\theta|A|$. We do not include a proof, since this fact is not used later, but it provides helpful intuition regarding $\theta$-compressions.
\end{rem}

\begin{ntt} \label{ntt for independent case}
For every limit ordinal $\alpha<\omega_1$, we fix, for the rest of the paper, a collection
\[T_{\beta}^{\alpha}\colon L_p\to L_p,\; 0\leq \beta<\alpha\]
of distributional embeddings with independent ranges; that is, for any family $(f_{\beta})_{\beta<\alpha}$ in $L_p$, the functions $(T_{\beta}^{\alpha}f_{\beta})_{\beta<\alpha}$ are independent. For example, this can be achieved by fixing a collection $\varphi^\alpha_\beta\colon([0,1],\mathcal{B}([0,1]))\to([0,1],\mathcal{B}([0,1]))$, $\beta<\alpha$, of independent measure-preserving maps, and setting $T_{\beta}^{\alpha}(f)=f\circ\varphi_{\beta}^{\alpha}$, as in Example \ref{example for compression part a }.
\end{ntt}

Using Notation \ref{example for compression part b} and \ref{ntt for independent case}, the spaces $R_{\alpha}^{p}$ can equivalently be defined as subspaces of $L_{p}$ as follows.

\begin{dfn} \label{equiv def of BRS spaces}
Let $R_{0}^{p} = \langle \chi_{[0,1]} \rangle$, and, if $\alpha$ is a countable ordinal number such that, for all $0\leq \beta<\alpha$, $R_\beta^p$ has been defined then
\begin{enumerate}[label=(\alph*)]
\item if $\alpha = \beta+1$ let
\begin{align*}
R_{\alpha}^{p} = T_{[0,\frac{1}{2})} ( R_{\beta}^{p}) +  T_{ [\frac{1}{2},1) } ( R_{\beta}^{p} ),
\end{align*}
\item and if $\alpha$ is a limit ordinal let
\begin{align*}
R_{\alpha}^{p} = \left[ T_{\beta}^{\alpha}(R_{\beta}^{p}):\beta < \alpha \right].
\end{align*}
\end{enumerate}
\end{dfn}

Henceforth, we will always work with this version of the Bourgain-Rosenthal-Schechtman spaces.

\begin{rem}
The choice of maps at each step does not matter: at successor steps one may use any $1/2$-compressions with disjointly supported ranges, and at limit steps any distributional embeddings with independent ranges; the outcome is distributionally isomorphic. While this may be intuitively evident, it follows rigorously, e.g., from the statements of Section \ref{generic assembly of compressions}. That said, the map $T_I$ is slightly advantageous notationally relative to an abstract $|I|$-compression, since it admits a simple formal adjoint that we later use implicitly.
\end{rem}

\begin{ntt} Let $0\leq\alpha < \omega_{1}$. Denote $R_{\alpha}^{p,0} = \lbrace f \in R_{\alpha}^{p}: \mathbb{E}(f) = 0  \rbrace$.
\end{ntt}
For the purpose of this paper we will mostly work with $R_{\alpha}^{p,0}$ instead of $R_{\alpha}^{p}$, where $1\leq \alpha<\omega_1$.

\begin{rem} \label{rem for the mean zero BRS spaces }
We point out the following recursive relations satisfied by the spaces $R_\alpha^{p,0}$.
\begin{enumerate}[label=(\alph*)]

\item \label{rem for the mean zero BRS spaces, suc }
 For $1\leq \alpha < \omega_{1}$, for every $n \in \mathbb{N}$, by induction, we have that
\begin{align*}
R_{\alpha+n}^{p,0} = L_p^{n,0} + \langle T_I(R_{\alpha}^{p,0}): I \in \mathcal{D}^{n}\rangle = \langle h_{I}: I \in \mathcal{D}_{n-1}  \rangle  + \langle T_I(R_{\alpha}^{p,0}): I \in \mathcal{D}^{n}\rangle.
\end{align*}   

\item \label{rem for the mean zero BRS spaces, lim }
For a limit ordinal $\alpha$,
\begin{align*}
R_{\alpha}^{p,0} = [T_{\beta}^{\alpha}(R_{\beta}^{p,0}):1\leq\beta < \alpha]\equiv^\mathrm{dist}\Big(\sum_{1\leq \beta<\alpha}R_p^{\beta,0}\Big)_{\mathrm{Ind},p}.
\end{align*}

\end{enumerate}
\end{rem}

\begin{exa}
\label{example initial Rap spaces}
Let $n \in \mathbb{N}$.
\begin{enumerate}[label=(\alph*)]

\item\label{example initial Rap spaces a} $R_{n}^{p,0} = L_{p}^{n,0}$.

\item\label{example initial Rap spaces b} $R_{\omega}^{p,0} =  \left[ \bigcup_{k \in \mathbb{N}} T_{k}^{\omega}(R_{k}^{p,0}) \right]$.

\item\label{example initial Rap spaces c} $R_{\omega+n}^{p,0} =  \left[ L_{p}^{n,0} \bigcup \bigcup_{I \in \mathcal{D}^{n}} T_{I} ( R_{\omega}^{p,0} ) \right] = \left[ \langle h_{I}: I \in \mathcal{D}_{n-1} \rangle \bigcup \bigcup_{I \in \mathcal{D}^{n}} T_{I} ( R_{\omega}^{p,0} ) \right] $.

\end{enumerate}
\end{exa}

In the following remark we observe that the $R_{\alpha}^{p,0}$ space is orthogonally complemented.

\begin{rem}
\label{the 0 BRS spaces are orthog compl}
Let $\alpha < \omega_{1}$. By Theorem \ref{Orth Compl spaces the BRS}, $R_{\alpha}^{p}$, and by Proposition \ref{distro copy of orthocomplemented is orthocomplemented and others} every distributional copy of it, is $(p^{\ast}-1)^{2}(\tfrac{p^{\ast}}{2})^{3/2}$-orthogonally complemented. Let $P$ denote the corresponding $p$-$C$-orthogonal projection, where $C=(p^{\ast}-1)^{2}(\tfrac{p^{\ast}}{2})^{3/2}$. Since $L_{p}^{0}$ is 2-orthogonally complemented via $Q(f)=f-\mathbb{E}(f)$, it follows that $R_{\alpha}^{p,0}$ is $2C$-orthogonally complemented via $QP$.
\end{rem}

\section{Decompositions of $R_{\alpha}^{p,0}$}
\label{fdd section}
The main goal of this section is to construct an unconditional Schauder decomposition of $R_{\alpha}^{p,0}$ with the required reproducing properties, suitable for defining and studying classes of distributional embeddings of $R_{\alpha}^{p,0}$ relevant to the orthogonal reduction of operators. For every ordinal $\alpha < \omega_{1}$, we construct a countable well-founded tree $\mathcal{T}_{\alpha}$. The elements of each tree $\mathcal{T}_{\alpha}$ are finite sequences whose entries are ordinal numbers and dyadic intervals, ordered by initial segments. For each $\lambda \in \mathcal{T}_{\alpha}$, its entries encode information that determine a positive integer $\kappa(\lambda)$, a parameter $\theta_{\lambda} \in (0,1]$, and a corresponding $\theta_{\lambda}$-compression $T_{\lambda}$, which in turn define the space $X_{\lambda} = T_{\lambda}(L_{p}^{\kappa(\lambda),0})$. We then show that, under a suitable enumeration, the collection $((T_{\lambda} h_{I})_{I \in \mathcal{D}_{\kappa(\lambda)-1}})_{\lambda \in \mathcal{T}_{\alpha}}$ forms a martingale difference sequence spanning $R_{\alpha}^{p,0}$. Finally, we study the supports of the spaces $(X_{\lambda})_{\lambda \in \mathcal{T}_{\alpha}}$, a fact needed in subsequent sections where the aforementioned distributional embeddings are introduced and studied. The constructions of this section will be used throughout the remainder of the paper.

\subsection{The trees $\mathcal{T}_\alpha$}
For every $1\leq\alpha <\omega_1$ we construct a countable well-founded tree $\mathcal{T}_\alpha$. Each member $\lambda$ of $\mathcal{T}_\alpha$ is a finite sequence $(x_1,\ldots,x_l)$, where each $x_i$ is either an ordinal number or a dyadic interval, and this information encodes the position of the subspace $X_\lambda$ (which will be defined in Section \ref{subsec Tlambda}) inside $R_\alpha^{p,0}$.

\begin{ntt}
For every ordinal number $\alpha$ we define its integer part $\kappa(\alpha) \in \mathbb{N} \cup \lbrace 0 \rbrace$ as follows:
\begin{itemize}
\item $\kappa(\alpha) = 0$, for every limit ordinal $\alpha$ and $\alpha = 0$.

\item $\kappa(\alpha +1) = \kappa(\alpha) +1$, for every ordinal $\alpha$.
\end{itemize}
\end{ntt}
For example, $\kappa(0) =0$, $\kappa (n) = n$ and $\kappa (\omega +n) = n$. More generally, when we write $\alpha = \beta + \kappa(\alpha)$, then $\beta$ is a limit ordinal number.

\begin{ntt}
For a finite sequence $\lambda = (x_1,\ldots,x_l)$ we denote its length by $|\lambda| = l$ and for $1\leq m\leq l$ we denote $\lambda|_m = (x_1,\ldots,x_m)$. Let $\mu = (y_{1},...,y_{k})$ be another finite sequence. 
\begin{enumerate}[label=(\alph*)]

\item Denote the concatenation of $\lambda$ and $\mu$ by  
\begin{align*}
\lambda^{\frown} \mu = (x_{1},...,x_{l},y_{1},...,y_{k}).
\end{align*}

\item If $x_{l}=y_{1}$, then denote the gluing of $\lambda$ and $\mu$ by 
\begin{align*}
\lambda \oslash \mu = (x_{1},...,x_{l},y_{2},...,y_{k}).
\end{align*}

\end{enumerate}
\end{ntt}

The operation ``$^\smallfrown$'' is used in the definition of the trees $\mathcal{T}_{\alpha}$, whereas both ``$^\smallfrown$'' and ``$\oslash$'' will be employed in the study of distributional embedding schemes in Sections \ref{DES section} and \ref{dist repre subsection}.

We define the trees $\mathcal{T}_\alpha$, $1\leq \alpha<\omega_1$ in two equivalent ways. The first is recursive and highlights how $\mathcal{T}_\alpha$ depends on the trees for smaller ordinals; the second is direct and gives an explicit description of the elements of $\mathcal{T}_\alpha$.

The recursive definition of $\mathcal{T}_{\alpha}$ naturally mirrors the recursive structure of the spaces $R_{\alpha}^{p,0}$ described in Remark \ref{rem for the mean zero BRS spaces }. One could, at the same time, define the spaces $(X_{\lambda})_{\lambda\in\mathcal{T}_{\alpha}}$ recursively, using the operators $T_I$ and $T_{\beta}^{\alpha}$ in the evident way, as indicated in the introduction. For practical reasons, however, we postpone this until after we introduce the operators $(T_{\lambda})_{\lambda\in\mathcal{T}_{\alpha}}$ in Section~\ref{subsec Tlambda}, which allows a more convenient treatment of these spaces.

\begin{dfn} \label{definition of the indexed trees}
Let $1\leq \alpha < \omega_{1}$. 
\begin{enumerate}[label=(\alph*)]
\item \label{definition of the indexed trees, finite} If $\alpha = n$, then 
\begin{align*}
\mathcal{T}_{n} = \lbrace (n) \rbrace.
\end{align*}
This is based on $R_n^{p,0} = L_p^{n,0}$ and $\mathcal{T}_n = \{(n)\}$ because the FDD will be trivial in this case.

\item  \label{definition of the indexed trees, limit} If $\alpha$ is a limit ordinal, then 
\begin{align*}
\mathcal{T}_{\alpha} = \bigcup_{1\leq \beta<\alpha}\lbrace (\alpha)^{\frown} \lambda : \lambda \in \mathcal{T}_{\beta} \rbrace.
\end{align*}
Because $R_\alpha^{p,0} \equiv^\mathrm{dist} (\sum_{\beta<\omega_1}R_p^{\beta,0})_{\mathrm{Ind},p}$, the index set $\mathcal{T}_\alpha$ of the FDD of $R_\alpha^{p,0}$ is the disjoint union of the sets $\mathcal{T}_\beta$, $\beta<\alpha$.

\item \label{definition of the indexed trees, suc ord}  If $\alpha = \beta + \kappa(\alpha)$ is a successor, then  
\begin{align*}
\mathcal{T}_{\alpha} = \lbrace (\alpha) \rbrace \cup \bigcup_{I \in \mathcal{D}^{\kappa(\alpha)}}\lbrace (\alpha, I)^{\frown} \lambda : \lambda \in \mathcal{T}_{\beta} \rbrace.
\end{align*}
In this case, $R_\alpha^{p,0} = L_p^{\kappa(\alpha),0}+[T_I(R_\beta^{p,0}):I\in\mathcal{D}^{\kappa(\alpha)}]$. Therefore the index set $\mathcal{T}_\alpha$ for the FDD of $R_\alpha^{p,0}$ is the disjoint union of $2^{\kappa(\alpha)}$ copies  of the set $\mathcal{T}_\beta$ together with $(\alpha)$, the latter corresponding to the space $L_p^{\kappa(\alpha),0}$, viewed as lying ``above'' the spaces $T_I(R_\beta^{p,0})$, $I\in\mathcal{D}^{\kappa(\alpha)}$.

\end{enumerate}
The pair $(\mathcal{T}_\alpha,\sqsubseteq)$, where $\sqsubseteq$ is the initial-segment order, is a countable well-founded tree.
\end{dfn}

We follow standard tree terminology. For $\lambda,\mu\in\mathcal{T}_{\alpha}$, if $\mu\sqsubsetneq\lambda$, we say that $\lambda$ is a successor of $\mu$ and that $\mu$ is a predecessor of $\lambda$. If $\mu$ is the maximum predecessor of $\lambda$, then we say that $\lambda$ is an immediate successor of $\mu$. A member $\lambda\in\mathcal{T}_{\alpha}$ with no predecessors is called a root. Denote
\begin{align*}
\mathpzc{Root}(\mathcal{T}_{\alpha}) = \lbrace \lambda \in \mathcal{T}_{\alpha}: \; \lambda \; \text{is a root} \rbrace.
\end{align*}
The height of a member $\lambda\in\mathcal{T}_\alpha$, denoted $h(\lambda)$, is the minimal path distance from a root of $\mathcal{T}_\alpha$ to $\lambda$. There can be indices $m$ with $1\le m\le|\lambda|$ for which $\lambda|_m\notin\mathcal{T}_\alpha$, hence in general $h(\lambda)$ can be arbitrarily smaller than $|\lambda|$.

\begin{rem} \label{remark} Let $\alpha < \omega_{1}$.
\begin{enumerate}[label=(\alph*)]
\item If $\alpha$ is a limit ordinal, then $\mathpzc{Root} (\mathcal{T}_{\alpha})$ is countably infinite.

\item \label{remark2} If it is a successor ordinal and $\alpha =  \beta + \kappa(\alpha)$, then $\mathpzc{Root}(\mathcal{T}_{\alpha})$ is the singleton $\lbrace (\beta + \kappa(\alpha)) \rbrace$. 
\end{enumerate}
\end{rem}

\begin{exa} The trees $\mathcal{T}_\omega$, $\mathcal{T}_{\omega+n}$, and $\mathcal{T}_{\omega\cdot 2}$ can be easily described as follows:
\begin{align*}
\mathcal{T}_{\omega} &= \lbrace (\omega, n): n \in \mathbb{N} \rbrace,\\
\mathcal{T}_{\omega + n} &= \lbrace (\omega + n) \rbrace \cup \lbrace (\omega + n, I, \omega, k): I \in \mathcal{D}^{n}, k \in \mathbb{N} \rbrace,\text{ where $n\in\mathbb{N}$ is fixed, and}\\
\mathcal{T}_{\omega \cdot 2} &= \lbrace (\omega \cdot 2, \omega + n): n \in \mathbb{N} \rbrace \cup \lbrace (\omega \cdot 2, \omega + n, I, \omega, k): n \in \mathbb{N}, I \in  \mathcal{D}^{n}, k \in \mathbb{N} \rbrace \cup\\
&\mathrel{\phantom{=}}\lbrace (\omega \cdot 2, \omega, k): k \in \mathbb{N}  \rbrace \cup \lbrace (\omega \cdot 2, k): k \in \mathbb{N} \rbrace.
\end{align*}
\end{exa}

The following definition provides an explicit alternative description of all trees $\mathcal{T}_{\alpha}$, highlighting the structure of the entries of a sequence $\lambda = (x_{1}, \ldots, x_{l}) \in \mathcal{T}_{\alpha}$ and giving a clearer intuition for the form of these trees. Moreover, the entries of $\lambda$ carry a more specific meaning, as they can be used directly to construct   the operator $T_{\lambda}$ and consequently the space $X_{\lambda}$.

\begin{dfn} \label{dfn of ambient tree}
Let $\mathcal{T}$ be the collection of all finite sequences $(x_{1},...,x_{l})$ in $[1, \omega_{1}) \cup \mathcal{D}$ satisfying the following:
\begin{itemize}
\item $x_{1} \in [1, \omega_{1})$.

\item For $1 \leq i \leq l $, if $x_{i}$ is a limit ordinal, then $i < l$ and $x_{i+1}  \in [1, \omega_{1})$ with $x_{i+1} < x_{i}$.

\item For $1 \leq i \leq l $, if $x_{i}$ is a successor ordinal, then either $i = l$ or $i \leq l-2$. In the latter case $x_{i+1}  \in \mathcal{D}^{\kappa(x_{i})}$ and  $x_{i+2} + \kappa (x_{i}) =  x_{i}$.
\end{itemize}
\end{dfn}

The following remark can be shown easily by induction on the countable ordinal numbers.
\begin{rem} \label{alt def of the trees}
For every $1\leq \alpha < \omega_{1}$, we have that 
\begin{align*}
\mathcal{T}_{\alpha} = \lbrace (x_{1},...,x_{l}) \in \mathcal{T}: x_{1} = \alpha \rbrace.
\end{align*}
\end{rem}

\begin{ntt}
Let $\alpha < \omega_{1}$ and $\lambda = (x_{1},...,x_{l}) \in \mathcal{T}_{\alpha}$, Remark \ref{alt def of the trees} yields that $x_l$ is a successor ordinal number, i.e., there is a limit (or zero) ordinal $\xi$ such that $x_{l} = \xi + k(x_{l})$. Denote
\begin{align*}
b(\lambda) &= \xi\text{ and}\\
\kappa(\lambda) &= \kappa(x_l).
\end{align*}
\end{ntt}

For example, for $\lambda = (\omega \cdot 2, \omega + n) \in \mathcal{T}_{\omega \cdot 2}$ we have $\kappa(\lambda) = n$ and $b(\lambda) = \omega$.

\begin{rem} \label{characterization  of im suc} Let $\alpha < \omega_{1}$ and $\lambda \in \mathcal{T}_{\alpha}$. 
\begin{enumerate}[label=(\alph*)]

\item \label{characterization  of im suc, part 1} If $\mu$ is a predecessor of a $\lambda$, then there exist $I \in \mathcal{D}^{\kappa(\mu)}$ and $\nu \in \mathcal{T}_{b(\mu)}$ such that
\[\lambda = \mu^{\frown}I^{\frown} \nu.\]  
If, furthermore, $\lambda$ is an immediate successor of $\mu$ then $\nu$ is a root of $\mathcal{T}_{b(\mu)}$.

\item If $\nu\in\mathcal{T}_{b(\lambda)+\kappa(\lambda)}$ then $\lambda\oslash\nu\in\mathcal{T}_\alpha$.

\item \label{characterization  of im suc, part 2} $\lambda$ is a root if and only if $\lambda (i) \in [1, \omega_{1})$ for every $1 \leq i \leq  |\lambda|$. 

\end{enumerate}
\end{rem}

\subsection{Operators and spaces indexed over $\mathcal{T}_\alpha$}
\label{subsec Tlambda}

Following the recursive definition of the trees $\mathcal{T}_\alpha$, we define, for every $\lambda\in\mathcal{T}_\alpha$, the $\theta_\lambda$-compression $T_\lambda:L_p\to L_p$, which will be used to directly define $X_\lambda = T_\lambda(L_p^{\kappa(\lambda),0})$, but also later to study distributional copies of the FDD $(X_\lambda)_{\lambda\in\mathcal{T}_\alpha}$.

\begin{dfn}
\label{definition of T_lampda}
Recursively on the countable ordinal numbers $\alpha$, we define, for every $\lambda \in \mathcal{T}_{\alpha}$, a number $0< \theta_{\lambda} \leq 1$ and a $\theta_\lambda$-compression $T_{\lambda} \colon L_{p} \to L_{p}$ as follows.
\begin{enumerate}[label=(\alph*)]
\item \label{definition of T_lampda, finite} Let $\alpha = n$ and  $\lambda = (n) \in \mathcal{T}_{n}$. We let $\theta_{\lambda} = 1$ and $T_{\lambda} = id$.

\item  \label{definition of T_lampda, limit} Let $\alpha$ be a limit ordinal and $\lambda  \in \mathcal{T}_{\alpha}$. Then there are $1\leq \beta<\alpha$ and $\mu\in\mathcal{T}_\beta$ such that $\lambda = (\alpha)^{\frown} \mu$. We let $\theta_{\lambda} = \theta_{\mu}$ and $T_{\lambda} = T_{\beta}^{\alpha} \circ T_{\mu}$, where $T_\beta^\alpha$ is as in Notation \ref{ntt for independent case}.

\item \label{kostasc} Let $\alpha = \beta + \kappa(\alpha)$ be an infinite successor ordinal and $\lambda \in \mathcal{T}_{\alpha}$. Then 
\begin{enumerate}[label=(\roman*)]
\item \label{kostasc, head} if $\lambda = (\alpha)$, let $\theta_{\lambda} = 1$ and $T_{\lambda} = id$.

\item \label{kostasc, no head} if $\lambda = (\alpha, I)^{\frown} \mu$, where $I \in \mathcal{D}^{\kappa(\alpha)}$ and $\mu \in \mathcal{T}_{\beta}$, let $\theta_{\lambda} = \vert I \vert \theta_{\mu}$ and $T_{\lambda} = T_{I} T_{\mu}$, where $T_I$ is the $|I|$-compression as in Notation \ref{example for compression part b}.
\end{enumerate}
\end{enumerate}
\end{dfn}

\begin{exa}
Let $\lambda = (\omega \cdot 2, \omega + n, I, \omega, k) \in \mathcal{T}_{\omega \cdot 2}$, where $n \in \mathbb{N}$, $I \in  \mathcal{D}^{n}$ and $k \in \mathbb{N}$. Then,
\begin{align*}
\theta_{\lambda} = \theta_{(\omega + n, I, \omega, k )} = \vert I \vert \theta_{(\omega, k)} = \vert I \vert \theta_{(k)} = \vert I \vert
\end{align*} 
and
\begin{align*}
T_{\lambda} = T_{\omega + n}^{\omega \cdot 2} T_{ (\omega + n, I, \omega, k )} =  T_{\omega + n}^{\omega \cdot 2} T_{I} T_{( \omega,k)} =  T_{\omega + n}^{\omega \cdot 2} T_{I}  T_{k}^{\omega}.
\end{align*}
\end{exa}

\begin{rem}
Let $\alpha<\omega_1$ and $\lambda\in\mathcal{T}_\alpha$.
\begin{enumerate}[label=(\alph*)]

    \item If $b(\lambda)$ is infinite, for $I\in\mathcal{D}^{\kappa(\lambda)}$ and $\mu\in\mathcal{T}_{b(\lambda)}$,
    we have
    \[\theta_{\lambda^\smallfrown I^\smallfrown \mu} = \theta_\lambda|I|\theta_\mu\text{ and }T_{\lambda^\smallfrown I^\smallfrown \mu} = T_\lambda T_I T_\mu.\]

\item For $\mu\in\mathcal{T}_{b(\lambda)+\kappa(\lambda)}$, we have
\[\theta_{\lambda\oslash\mu} = \theta_\lambda\theta_\mu\text{ and }T_{\lambda\oslash\mu} = T_\lambda T_\mu.\]

\end{enumerate}
\end{rem}

In the next remark, we provide an explicit definition of $\theta_{\lambda}$ and $T_{\lambda}$, although it will not be used directly.

\begin{rem} \label{explicit def of Theta and T} Let $\alpha < \omega_{1}$ and $\lambda = (x_1,\ldots,x_l) \in \mathcal{T}_{\alpha}$.
\begin{enumerate}[label=(\alph*)]
\item \label{explicit def of Theta and T, part a} Following the convention $\prod_\emptyset = 1$,
\begin{align*}
\theta_{\lambda} = \underset {1 \leq i \leq l\atop x_i\in\mathcal{D}}{\prod} \vert x_i \vert = \prod_{1\leq i<l\atop x_i\in[1,\omega_1)}2^{-\kappa(x_i)}.
\end{align*}

\item \label{explicit def of Theta and T, part b} For $1\leq j\leq l$ define 
\begin{equation*}
T_{j} = 
    \begin{cases}
    id &:\;j=1\text{ or }x_{j-1}\in\mathcal{D},\\
        T_{x_{j}}^{x_{j-1}} &:\; x_{j-1},x_j \in [1,\omega_{1}), \\
       T_{J}  &:\; x_{j} = J \in \mathcal{D}.
    \end{cases}
\end{equation*}
Then, $T_\lambda = T_1T_2\cdots T_l$.
\end{enumerate}
\end{rem}

For example, if $\lambda = (\omega \cdot 3, \omega + n, I, \omega, k) \in \mathcal{T}_{\omega \cdot 3}$ then $\theta_{\lambda} = |I| = 2^{-n}$ and $T_{\lambda} = T_{\omega + n}^{\omega \cdot 3} T_{I} T_{k}^{\omega}$.

\begin{dfn} \label{def of Xlambda}
Let $\alpha < \omega_{1}$ and $\lambda \in \mathcal{T}_{\alpha}$.
\begin{enumerate}[label=(\alph*)]

\item For $I\in\mathcal{D}_{\kappa(\lambda)-1}$, let
\begin{align*}
h_{I}^{\lambda} = T_{\lambda}(h_{I}),
\end{align*}

\item\label{def of Xlambda 1}
and let
\begin{align*}
X_\lambda &= T_\lambda(L_p^{\kappa(\lambda),0}) = \langle  h^\lambda_{I}: I \in \mathcal{D}_{\kappa(\lambda)-1} \rangle\text{ and}\\
Y_\lambda &= [X_\mu:\mu\sqsupseteq\lambda].
\end{align*}
    
\end{enumerate}
\end{dfn}

\begin{rem}
\label{recursive structure of Xlambda}
Let $1 \leq \alpha < \omega_{1}$. Definitions \ref{definition of T_lampda} and \ref{def of Xlambda} immediately yield the following recursive structure for the spaces $(X_{\lambda})_{\lambda \in \mathcal{T}_{\alpha}}$, reflecting the recursive construction of the space $R_{\alpha}^{p,0}$. The remarks given inside Definition \ref{definition of the indexed trees} apply here as well.

\begin{enumerate}[label=(\alph*)]
\item If $\alpha = n$, then $\theta_{n} = 1$ and $X_{\lambda} = L_{p}^{n,0}$.

\item If $\alpha$ is a limit ordinal and $\lambda \in \mathcal{T}_{\alpha}$, then $\lambda = (\alpha)^{\frown}\mu$ for some $\mu \in \mathcal{T}_{\beta}$ with $\beta < \alpha$, and $\theta_\lambda = \theta_\mu$ and $X_{\lambda} = T_{\beta}^{\alpha}(X_{\mu})$.

\item If $\alpha = \beta + \kappa(\alpha)$ is a successor and $\lambda \in \mathcal{T}_{\alpha}$, there are two possibilities for $X_{\lambda}$:
\begin{enumerate}
\item If $\lambda = (\alpha)$, then $\theta_\lambda=1$ and $X_{\lambda} = L_{p}^{\kappa(\alpha),0}$.
\item If $\lambda = (\alpha, I)^{\frown}\mu$ with $I \in \mathcal{D}^{\kappa(\alpha)}$ and $\mu \in \mathcal{T}_{\beta}$, then $\theta_\lambda = |I|\theta_\mu$ and $X_{\lambda} = T_{I}(X_{\mu})$.
\end{enumerate}
\end{enumerate}
\end{rem}

\subsection{The finite-dimensional decomposition of $R_\alpha^{p,0}$}
We show that, under a suitable enumeration, the family $((h_{I}^{\lambda})_{I \in \mathcal{D}_{\kappa(\lambda)-1}})_{\lambda \in \mathcal{T}_{\alpha}}$ forms a martingale difference sequence spanning $R_{\alpha}^{p,0}$. This is essentially \cite[Proposition 2.8]{bourgain:rosenthal:schechtman:1981}, adapted to our notation and framework. In particular, $(X_{\lambda})_{\lambda \in \mathcal{T}_{\alpha}}$ is an unconditional FDD of $R_{\alpha}^{p,0}$.

We first show that $(X_\lambda)_{\lambda\in\mathcal{T}_\alpha}$ spans $R_\alpha^{p,0}$.

\begin{prop} \label{FDD proposition} Let $\alpha < \omega_{1}$. Then,
\[R_{\alpha}^{p,0}  = [(X_{\lambda})_{\lambda\in\mathcal{T}_\alpha} ] = [(Y_\lambda)_{\lambda \in \mathpzc{Root}( \mathcal{T}_\alpha)}].\]
\end{prop}

\begin{proof}
The second equality is immediate. For the first, argue by induction on $\alpha$.  Let $\alpha=n\in\mathbb{N}$. By Definition \ref{definition of the indexed trees}\ref{definition of the indexed trees, finite}, $\mathcal{T}_n=\{(n)\}$. For $\lambda=(n)$, by Remark \ref{recursive structure of Xlambda}, $X_\lambda=L_p^{n,0} = R_n^{p,0}$.

Let $\alpha$ be a limit ordinal. By Definition \ref{definition of the indexed trees} \ref{definition of the indexed trees, limit}, every $\lambda\in\mathcal{T}_\alpha$ has the form $(\alpha)^{\frown}\mu$ with $\mu\in\mathcal{T}_\beta$ for some $\beta<\alpha$. By the induction hypothesis and Remark \ref{recursive structure of Xlambda},
\[[(X_\lambda)_{\lambda\in\mathcal{T}_\alpha}]=\Big[\bigcup_{\beta<\alpha}T_\beta^\alpha\big( \bigcup_{\mu\in\mathcal{T}_\beta}  X_\mu\big)\Big] = \Big[\!\bigcup_{\beta<\alpha}T_\beta^\alpha(R_\beta^{p,0})\Big]=R_\alpha^{p,0}.\]

Let $\alpha = \beta + \kappa(\alpha)$. By the induction hypothesis and Remark \ref{recursive structure of Xlambda}, we conclude:
\begin{align*}
[(X_{\lambda})_{\lambda \in \mathcal{T}_{\alpha}}] &= \Big[X_{(\alpha)}  \bigcup \bigcup_{I \in \mathcal{D}^{\kappa (\alpha)}}\bigcup_{\mu \in \mathcal{T}_{\beta}} X_{(\alpha,I)^\smallfrown \mu}\Big] =
 L_p^{\kappa(\alpha),0} + \Big[ \bigcup_{I \in \mathcal{D}^{\kappa (\alpha)}} T_{I}\Big(\bigcup_{\mu \in \mathcal{T}_{\beta}} X_{\mu} \Big) \Big]\\
 &= L_p^{\kappa(\alpha),0}+\Big[ \bigcup_{I\in\mathcal{D}^{\kappa(\alpha)}}T_I(R_\beta^{p,0})\Big] = R_\alpha^{p,0}. 
\end{align*}
\end{proof}

The following is the main result of this section. Before proceeding to the proof, we introduce the necessary notation.

\begin{thm} \label{unconditional FDD of the R_a^p} Let $\alpha < \omega_{1}$ and $(\lambda_{n},I_n)_{n=1}^{\infty}$ be an enumeration of
\[\{(\lambda,I):\lambda\in\mathcal{T}_\alpha\text{ and }I\in\mathcal{D}_{\kappa(\lambda)-1}\}\]
such that for all $n<m$ one of the following holds:
\begin{enumerate}[label=(\alph*)]
    
    \item $\lambda_n$ and $\lambda_m$ are incomparable, or

    \item $\lambda_n\sqsubsetneq \lambda_m$, or

    \item $\lambda_n = \lambda_m$ and $I_n <I_m$ in the lexicographical order of $\mathcal{D}_{\kappa(\lambda_n)-1}$.
    
\end{enumerate}
Then, $(h_{I_n}^{\lambda_{n}})_{n=1}^\infty$ is a martingale difference sequence, and by the Burkholder inequality, it is an unconditional basis of $R_\alpha^{p,0}$, and $(X_\lambda)_{\lambda\in\mathcal{T}_\alpha}$ is an unconditional FDD of the same space.
\end{thm}

\begin{ntt}
For $0\leq a<b\leq 1$ and $I = [a,b)$ denote $q_I:[0,1]\to [0,1]$ given by $q_I(t) = a+(b-a)t$.    
\end{ntt}

\begin{lem}
\label{computations using compressions sets and integrals}
For $0\leq a<b\leq 1$, $I = [a,b)$, and $f\in L_p$ the following hold:
\begin{enumerate}[label=(\alph*)]
    
    \item\label{computations using compressions sets and integrals a} for every $A\subset \mathbb{R}\setminus\{0\}$,
    \[\big[T_If\in A\big] = q_I\big([f\in A])\text{ and}\]

    \item\label{computations using compressions sets and integrals b} for every $A\subset [0,1]$,
    \[\int_A(T_If)(t)dt = |I|\int_{q_I^{-1}(A)}f(t)dt.\]

\end{enumerate}
\end{lem}

\begin{proof}
The first assertion is proved with an elementary double inclusion, of which we showcase one direction. Let $t\in[0,1]$ such that $Tf\in A$. Since $(T_If)\neq 0$, it follows that $t\in I$, and $(T_If)(t) = f((b-a)^{-1}(t -a))$. Defining $s = (b-a)^{-1}(t-a)$ leads to $s\in[ f\in A]$, resulting in $t = q_I(s)$.

We establish the second assertion for $f =\chi_B$, where $B$ is a measurable subset of $[0,1]$. We will use $T_I\chi_B = \chi_{a+(b-a)B}$, where $a+(b-a)B\subset I$, and $q_I^{-1}(A) = (b-a)^{-1}[(A\cap I)-a]$. Indeed, by the homogeneity of the Lebesgue measure:
\begin{align*}
\int_A(T_I\chi_B)(t)dt &= \Big|A\cap \big(a+(b-a)B\big)\Big| = \Big|(A\cap I)\cap\big (a+(b-a)B\big)\Big|\\
&= \Big|[(A\cap I)-a]\cap \big((b-a)B\big)\Big|\\
&= (b-a)\Big|[(b-a)^{-1}(A\cap I)-a]\cap B\Big| = |I|\int_{q^{-1}(A)}\chi_B(t)dt.
\end{align*}
\end{proof}

\begin{prop}
\label{conditional expectation commutes with compression}
Let $n\in\mathbb{N}$ and, for $I\in\mathcal{D}^n$, $\mathcal{A}_I$ be a $\sigma$-subalgebra of $\mathcal{B}[0,1]$. Denote
\[\mathcal{A} = \Big\{A\subset[0,1]:\text{ for }I\in\mathcal{D}^n,\;q^{-1}_I(A)\in\mathcal{A}_I\Big\}.\]
Then, $\mathcal{A}$ is a $\sigma$-subalgebra of $\mathcal{B}[0,1]$ and the following hold.
\begin{enumerate}[label=(\alph*)]
    
    \item \label{conditional expectation commutes with compression,1} For every $I\in\mathcal{D}^n$ and $\mathcal{A}_I$-measurable function $f$, the function $T_If$ is $\mathcal{A}$-measurable.

    \item \label{conditional expectation commutes with compression,2} For every $f\in L_p$,
    \[ \mathbb{E}\big(T_If|\mathcal{A}\big) = T_I\mathbb{E}\big(f|\mathcal{A}_I\big).\]

\end{enumerate}
\end{prop}

\begin{proof} To prove (a), let  $I\in\mathcal{D}^n$, $f$ be a $\mathcal{A}_I$-measurable function and $A \in \mathcal{B}(\mathbb{R})$. If $A \subset \mathbb{R} \setminus \lbrace 0 \rbrace$, then from Lemma \ref{computations using compressions sets and integrals} \ref{computations using compressions sets and integrals a} we have that  $\big[T_If\in A\big] = q_I\big([f\in A]\big)\in \mathcal{A}$. Similarly, if $A = \{0\}$, then we have that $\big[T_If\in A\big] = q_I\big([f = 0]\big) \cup [0,1] \setminus I \in \mathcal{A}$, because $I \in \mathcal{A}$. To prove (b), let $f \in L_{p}$. By \ref{conditional expectation commutes with compression,1}, $T_I\mathbb{E}\big(f|\mathcal{A}_I\big)$ is $\mathcal{A}$-measurable. Therefore, it is enough to prove that, for $A\in\mathcal{A}$,
\[\int_A \Big(T_I\mathbb{E}\big(f|\mathcal{A}_I\big)\Big)(t)dt = \int_A(T_If)(t)dt.\] 
By Lemma \ref{computations using compressions sets and integrals} \ref{computations using compressions sets and integrals b} we have that
\[ \int_A \Big(T_I\mathbb{E}\big(f|\mathcal{A}_I\big)\Big)(t)dt = \vert I \vert  \int_{q_{I}^{-1}(A)} \Big(\mathbb{E}\big(f|\mathcal{A}_{I} \big) \Big)(t)dt
= \vert I \vert  \int_{q_{I}^{-1}(A)} f(t)dt =  \int_A(T_If)(t)dt.\] 
\end{proof}

\begin{proof}[Proof of Theorem \ref{unconditional FDD of the R_a^p}]
We proceed by induction on $\alpha$. For $\alpha = n$, we have $R_{n}^{p,0} = L_{p}^{n,0}$, and, as is well known, the basis $(h_{I})_{I \in \mathcal{D}_{n-1}}$ of $L_{p}^{n,0}$ forms a martingale difference sequence in the lexicographic order of $\mathcal{D}_{n-1}$. 

Let $\alpha$ be a limit ordinal. For every $n \in \mathbb{N}$, there exist $\beta < \alpha$ and $\mu_{n} \in \mathcal{T}_{\beta}$ such that $h_{I_{n}}^{\lambda_{n}} = T_{\beta}^{\alpha}(h_{I_{n}}^{\mu_{n}})$. For each $\beta < \alpha$, define $F(\beta) = \{ n \in \mathbb{N} : \mu_{n} \in \mathcal{T}_{\beta} \}$. The sets $F(\beta)$, $\beta < \alpha$, form a partition of $\mathbb{N}$. By the induction hypothesis, for every $\beta < \alpha$, the sequence $(h_{I_{n}}^{\mu_{n}})_{n \in F(\beta)}$ is a martingale difference sequence with respect to the increasing enumeration of $F(\beta)$. Because $(T^\alpha_\beta)_{\beta<\alpha}$ are distributional embeddings with independent ranges, $(h_{I_{n}}^{\lambda_{n}})_{n \in \mathbb{N}} = ((h_{I_{n}}^{\mu_{n}})_{n \in F(\beta)})_{\beta < \alpha}$ is a martingale difference sequence.

Let $\alpha = \beta + \kappa(\alpha)$, where $\beta$ is a limit ordinal, and let $n \in \mathbb{N}$ be arbitrary. Denote
\[\sigma = \sigma(h_{I_{k}}^{\lambda_{k}} : 1 \leq k \leq n - 1).\]
We will show that $\mathbb{E}(h_{I_{n}}^{\lambda_{n}} \mid \sigma) = 0$. We distinguish two cases for $\lambda_{n}$: either $\lambda_{n} = (\beta + \kappa(\alpha))$, or $\lambda_{n} \sqsupseteq (\beta + \kappa(\alpha))^{\frown} J_{n}$ for some $J_{n} \in \mathcal{D}^{\kappa(\alpha)}$. In the first case, by the enumeration of $\mathcal{T}_{\alpha}$, $\lambda_{1} = \lambda_{2} = \cdots = \lambda_{n} = (\beta + \kappa(\alpha))$. Since $(h_{I})_{I \in \mathcal{D}_{\kappa(\alpha)-1}}$ is a martingale difference sequence in the lexicographic order of $\mathcal{D}_{\kappa(\alpha)-1}$, it follows that $\mathbb{E}(h_{I_{n}}^{\lambda_{n}} \mid \sigma) = 0$. In the second case, we have $h_{I_{n}}^{\lambda_{n}} = T_{J_{n}}(h_{I_{n}}^{\mu_{n}})$ for some $\mu_{n} \in \mathcal{T}_{\beta}$. For every $J \in \mathcal{D}^{\kappa(\alpha)}$, define
\[F(J) = \{1 \leq k \leq n - 1 : \lambda_{k} \sqsupseteq (\beta + \kappa(\alpha))^{\frown} J\}.\]
If $k \in F(J)$, then $h_{I_{k}}^{\lambda_{k}} = T_{J}(h_{I_{k}}^{\mu_{k}})$ for some $\mu_{k} \in \mathcal{T}_{\beta}$. For each $J \in \mathcal{D}^{\kappa(\alpha)}$, let
\[\mathcal{A}_{J} = \sigma(h_{I_{k}}^{\mu_{k}} : k \in F(J)).\]
Observe     $\sigma \subset \mathcal{A}$, where $\mathcal{A}$ is as in Proposition \ref{conditional expectation commutes with compression}.  By the induction hypothesis, $(h_{I_{k}}^{\mu_{k}})_{k \in F(J_{n})} \cup \{h_{I_{n}}^{\mu_{n}}\}$ is a martingale difference sequence in the increasing enumeration of $F(J_{n})\cup\{n\}$. Combining this with Proposition \ref{conditional expectation commutes with compression} \ref{conditional expectation commutes with compression,2}, we obtain that
\[\mathbb{E}(h_{I_{n}}^{\lambda_{n}} \mid \mathcal{A}) = \mathbb{E}(T_{J_{n}} h_{I_{n}}^{\mu_{n}} \mid \mathcal{A}) = T_{J_{n}} \mathbb{E}(h_{I_{n}}^{\mu_{n}} \mid \mathcal{A}_{J_{n}}) = 0.\]
Therefore, $\mathbb{E}(h_{I_{n}}^{\lambda_{n}} \mid \sigma) = 0$.
\end{proof}

\begin{rem}
\label{IMPORTANT PROJECTION}
Let $\alpha < \omega_{1}$ and $(\lambda_{n}, I_{n})_{n=1}^{\infty}$ be an enumeration of $\{(\lambda, I) : \lambda \in \mathcal{T}_{\alpha},\, I \in \mathcal{D}_{\kappa(\lambda)-1}\}$ as in Theorem \ref{unconditional FDD of the R_a^p}. By Example \ref{Eaxmple that gives formula of projection}, if $(b_{I_{n}}^{\lambda_{n}})_{n=1}^{\infty}$ is a distributional copy of $(h_{I_{n}}^{\lambda_{n}})_{n=1}^{\infty}$, then the orthogonal projection $P$ onto $[(b_{I_{n}}^{\lambda_{n}})_{n=1}^{\infty}]$ is given by
\[P(f) = \sum_{n=1}^{\infty} \langle \|b_{I_{n}}^{\lambda_{n}}\|_{q}^{-1} b_{I_{n}}^{\lambda_{n}}, f \rangle\, \|b_{I_{n}}^{\lambda_{n}}\|_{p}^{-1} b_{I_{n}}^{\lambda_{n}}.\]
\end{rem}


\subsection{Independent decompositions of $R_\alpha^{p,0}$}
We show that for a limit ordinal $\alpha$, the space $R_{\alpha}^{p,0}$ is the independent sum of the successor ordinal spaces $R_{\beta}^{p,0}$, $1 \leq \beta < \alpha$, possibly with repetitions. The number of repetitions can, for instance, be determined from the Cantor normal form of $\alpha$. Although this result is not used explicitly in this form, certain parts of it, in particular the independence of the spaces $(Y_{\lambda})_{\lambda \in \mathpzc{Root}(\mathcal{T}_{\alpha})}$, are used in Section \ref{dist repre subsection}, where we prove that distributional embedding schemes induce distributional embeddings.

While the following statement holds in general, the second part is non-trivial only for limit ordinals $\alpha$, that is, when $\mathpzc{Root}(\mathcal{T}_{\alpha})$ is not a singleton.

\begin{prop}
\label{independent sum of successor ordinals}
For $\alpha<\omega_1$, the following hold.
\begin{enumerate}[label=(\alph*)]

    \item For every $\lambda\in\mathcal{T}_\alpha$, \[T_\lambda \big( R_{b(\lambda)+\kappa(\lambda)}^{p,0} \big) = Y_\lambda.\]

    \item The spaces $(Y_\lambda)_{\lambda\in\mathcal{R}_\alpha}$ are independent, and in particular,
\[R_{\alpha}^{p,0}\equiv^{\mathrm{dist}}\Big(\sum_{\lambda\in \mathpzc{Root}(\mathcal{T}_\alpha)} R^{p,0}_{b(\lambda)+\kappa(\lambda)}\Big)_{\mathrm{Ind},p}.\]

\end{enumerate}
\end{prop}

\begin{proof}
Firstly, we prove (a) by induction on $\alpha$. If $\alpha = n \in \mathbb{N}$, $\mathcal{T}_{\alpha} = \lbrace  (n) \rbrace$. For $\lambda = (n)$ we have that $T_{\lambda} = id$, and thus,
\begin{align*}
T_{\lambda} (R_{b(\lambda) + \kappa(\lambda)}^{p,0} ) = R_{n}^{p,0} = X_{(n)} = Y_{(n)}.
\end{align*}

Let $\alpha$ be a limit ordinal. Let $\lambda\in\mathcal{T}_\alpha$. Then $\lambda = (\alpha)^{\frown} \mu $, where $\mu \in \mathcal{T}_{\beta}$, for some $\beta<\alpha$. By the induction hypothesis, we observe that 
\begin{align*}
T_{\lambda} \big( R_{b(\lambda) + \kappa(\lambda)}^{p,0} \big) &= T_{\beta}^{\alpha} T_{\mu} \big( R_{b(\mu) + \kappa(\mu)}^{p,0} \big) = T_{\beta}^{\alpha} Y_{\mu} =  \big[  T_{\beta}^{\alpha} X_{\mu^{\prime}}: \mu \sqsubseteq \mu^{\prime} \big] \\
&= \big[  T_{\beta}^{\alpha} T_{\mu^{\prime}} L_{p}^{\kappa(\mu^{\prime}),0}: \mu \sqsubseteq \mu^{\prime} \big] = \big[  T_{\lambda^{\prime}} L_{p}^{\kappa(\lambda^{\prime}),0}: \lambda \sqsubseteq \lambda^{\prime} \big]  =  [X_{\lambda'}:\lambda\sqsubseteq \lambda'] = Y_{\lambda}.
\end{align*}

Now, let $\alpha = \beta + \kappa(\alpha)$. Let $\lambda\in\mathcal{T}_\alpha$. Then $\lambda = (\beta + \kappa(\alpha))$ or  $\lambda = (\beta + \kappa(\alpha))^{\frown} I^{\frown} \mu $, where $I \in \mathcal{D}^{\kappa(\alpha)}$ and $\mu \in \mathcal{T}_{\beta}$. In the first case, $T_\lambda = id$, consequently, 
\begin{align*}
 T_{\lambda} \big(R_{b(\lambda) + \kappa(\lambda)}^{p,0}\big) =  R_{\alpha}^{p,0} = Y_{\lambda}. 
\end{align*}
For the latter case, by the induction hypothesis, we observe that 
\begin{align*}
T_{\lambda} \big( R_{b(\lambda) + \kappa(\lambda)}^{p,0} \big) &= T_{I} T_{\mu} \big( R_{b(\mu) + \kappa(\mu)}^{p,0} \big) = T_{I} Y_{\mu}  = \big[  T_{I} X_{\mu^{\prime}}: \mu \sqsubseteq \mu^{\prime} \big] \\ 
&= \big[  T_{I} T_{\mu^{\prime}} L_{p}^{\kappa(\mu^{\prime}),0}: \mu \sqsubseteq \mu^{\prime} \big] = \big[  T_{\lambda^{\prime}} L_{p}^{\kappa(\lambda^{\prime}),0}: \lambda \sqsubseteq \lambda^{\prime} \big]  =[X_{\lambda'}:\lambda\sqsubseteq \lambda'] =Y_{\lambda}.
\end{align*}
Now, we prove (b) by induction on $\alpha$. From Remark \ref{remark} \ref{remark2} the finite and successor cases trivially hold. Let $\alpha$ be a limit ordinal and note $\mathpzc{Root}(\mathcal{T}_{\alpha}) = \lbrace (\alpha)^{\frown} \mu : \mu \in \mathpzc{Root}(\mathcal{T}_{\beta}), \beta < \alpha \rbrace$. For every $\beta<\alpha$, by the induction hypothesis and because $T_\beta^\alpha$ is a distributional embedding, the spaces $(Y_{(\alpha)^\smallfrown\mu})_{\mu\in\mathcal{T}_\beta} = (T_\beta^\alpha (Y_{\mu}))_{\mu\in\mathcal{T}_\beta}$ are independent. Because the operators $(T_\beta^\alpha)_{\beta<\alpha}$ have independent ranges, the spaces $((Y_{(\alpha)^\smallfrown\mu})_{\mu\in\mathcal{T}_\beta})_{\beta<\alpha}$ are independent. Proposition \ref{FDD proposition} and the above assertion (a) and (b) imply
\begin{align*}
R_{\alpha}^{p,0}\equiv^{\mathrm{dist}}\Big(\sum_{\lambda\in \mathpzc{Root}(\mathcal{T}_\alpha)} R^{p,0}_{b(\lambda)+\kappa(\lambda)}\Big)_{\mathrm{Ind},p}.
\end{align*}
\end{proof}

We do not use the following remark, but it can serve to make the last statement of Proposition \ref{independent sum of successor ordinals} more explicit. There are parallels between it and material in \cite[Section 2]{alspach:1999}.

\begin{rem}
Let $\alpha<\omega_1$ be a limit ordinal with Cantor normal form
\[\alpha = \omega^{\xi_1}k_1+\cdots+\omega^{\xi_m}k_m,\]
where $\omega_1>\xi_1>\cdots>\xi_m\geq 1$ and $k_1,\ldots,k_m\in\mathbb{N}$. By induction on $\alpha$, the following hold. 
\begin{enumerate}[label=(\alph*)]
    
    \item If $\xi_m\geq 2$, then, for every successor ordinal $1\leq \beta<\alpha$ the set
    \[\{\lambda\in\mathpzc{Root}(\mathcal{T}_\alpha):b(\lambda)+\kappa(\lambda) = \beta\}\]
    is countably infinite. In particular, $R_\alpha^{p,0}$ is distributionally isomorphic to the independent sum of infinitely many copies of itself.

    \item If $\xi_m = 1$, denote $\gamma = \omega^{\xi_1}k_1+\cdots+\omega^{\xi_{m-1}}k_{m-1}$. Then, for all $0\leq k<k_m$ and $n\in\mathbb{N}$, if $\beta =\gamma +\omega k + n$, 
    \[\#\{\lambda\in\mathpzc{Root}(\mathcal{T}_\alpha):b(\lambda)+\kappa(\lambda) = \beta\} = 2^{k_m - k-1},\]
    while for all successor ordinals $1\leq \beta<\gamma$ this set is countably infinite. In particular,
    \[R_\alpha^{p,0} \equiv^{\mathrm{dist}}\Big(R^{p,0}_\gamma \oplus \sum_{k = 0}^{k_m-1}\sum_{n=1}^\infty \big(\underbrace{R^{p,0}_{\gamma + \omega k + n}\oplus\cdots\oplus R^{p,0}_{\gamma + \omega k + n}}_{2^{k_m-k-1}\text{ terms}}\big)\Big)_{\mathrm{Ind},p}.\]
    
\end{enumerate}
\end{rem}


\subsection{The supports of the spaces $X_\lambda$, $\lambda\in\mathcal{T}_\alpha$}
For two incomparable $\lambda, \mu \in \mathcal{T}_{\alpha}$, we show that the spaces $X_{\lambda}$ and $X_{\mu}$ have disjoint supports if and only if the sequences $\lambda$ and $\mu$ first split at a position where their entries are dyadic intervals. This fact is used in Sections \ref{DES section} and \ref{dist repre subsection}, where distributional embedding schemes are studied.

Recall that all set-theoretic relations are understood in the essential sense.

\begin{dfn}
\label{support of subspace definition}
Let $X$ be a (not necessarily closed) subspace of $L_p$ and $A$ be a measurable subset of $[0,1]$.
\begin{enumerate}[label=(\alph*)]

\item  We write $\mathrm{supp}(X)\subset A$ if, for every $f\in X$, $\supp(f) \subset A$.

\item If $\mathrm{supp}(X)\subset A$ and, additionally, for every measurable subset $B$ of $[0,1]$ such that $\mathrm{supp}(X)\subset B$, we have $A\subset B$, then we write $\mathrm{supp}(X) = A$. 

\end{enumerate}

\end{dfn}

It is standard that there exists a set $A$ such that $\mathrm{supp}(X) = A$, and that $A$ is essentially unique.

\begin{prop} \label{remark for supp and lenghts} For $\alpha < \omega_{1}$ and $\lambda \in \mathcal{T}_{\alpha}$, the following hold.
\begin{enumerate}[label=(\alph*)]

\item \label{remark for supp and lenghts, part zero} We have
\[\supp(X_\lambda) = \supp(T_{\lambda} \chi_{[0,1]}),\]
and in particular $\theta_{\lambda} = \vert \supp(X_{\lambda}) \vert$.

\item \label{remark for supp and lenghts, part e} If $\mu$ is a successor of $\lambda$, then
\begin{align*}
\supp(X_{\mu}) \subset \supp(X_{\lambda} ).
\end{align*}
More precisely, if $\mu = \lambda^{\smallfrown} I^{\smallfrown} \nu$, where $I \in \mathcal{D}^{\kappa (\lambda)}$ and $\nu \in \mathcal{T}_{b(\lambda)}$, then 
\begin{align*}
\supp(X_{\mu}) \subset [T_{\lambda} \chi_{I} = 1  ].
\end{align*}

\end{enumerate}
\end{prop}

\begin{proof}
Both assertions are based on Remark \ref{true form of compression}. For $f\in X_\lambda$ there is $g\in L_p^{\kappa(\lambda),0}$ such that $f = T_\lambda g$. Recalling that $T_\lambda h_{[0,1]}\in X_\lambda$, we obtain 
\[\supp(f)\subset\supp(T_\lambda\chi_{[0,1]}) = \supp(T_\lambda h_{[0,1]})  \subset \supp(X_\lambda).\]
In other words, $\supp(X_\lambda)\subset [T_{\lambda} \chi_{[0,1]} = 1 ]\subset\supp(X_\lambda)$. Also,
\[|\supp(T_\lambda\chi_{[0,1]})| = |[T_\lambda\chi_{[0,1]} = 1]| = \theta_\lambda|[\chi_{[0,1]}=1]| = \theta_\lambda.\]
Next, for $\mu = \lambda^\smallfrown I^\smallfrown \nu$,
\[\supp(T_\mu\chi_{[0,1]}) = \supp(T_\lambda T_I T_\nu\chi_{[0,1]})\subset \supp(T_\lambda T_I\chi_{[0,1]}) = \supp(T_\lambda\chi_I).\]
\end{proof}

\begin{rem}
By Proposition \ref{remark for supp and lenghts} \ref{remark for supp and lenghts, part e}, for $\alpha < \omega_{1}$ and $\lambda \in \mathcal{T}_{\alpha}$, we have $\mathrm{supp}(Y_\lambda) = \mathrm{supp}(X_\lambda)$. 
\end{rem}

\begin{dfn} Let $\alpha < \omega_{1}$ and $\lambda, \mu \in \mathcal{T}_{\alpha}$. We say that $\lambda$ and $\mu$ are disjointly supported if 
\begin{align*}
\supp({X_\lambda}) \cap \supp(X_{\mu}) = \emptyset.
\end{align*}
\end{dfn}

\begin{prop} \label{incomparable and disj supp character via interval} Let $\alpha < \omega_{1}$.  Let $\lambda, \mu \in \mathcal{T}_{\alpha}$ be incomparable and denote
\[i_{0} = \min \lbrace i : \lambda(i) \neq \mu (i) \rbrace.\]
\begin{enumerate}[label=(\alph*)]
\item The entries $\lambda(i_{0})$ and  $\mu (i_{0})$ are  either both ordinals or both intervals.

\item \label{incomparable and disj supp character via interval, part 2} $\lambda$ and $\mu$ are disjointly supported if and only if $\lambda(i_{0})$ and  $\mu (i_{0})$ are both intervals.
\end{enumerate}
\end{prop}

\begin{proof}
We prove the second assertion, since the first follows directly from Definition \ref{dfn of ambient tree}. Assume that $\lambda(i_{0})$ and $\mu(i_{0})$ are both intervals, say $\lambda(i_{0}) = I$ and $\mu(i_{0}) = J$. Then $\lambda = \nu^{\frown} I^{\frown} \tilde{\lambda}$ and $\mu = \nu^{\frown} J^{\frown} \tilde{\mu}$, where $I \neq J \in \mathcal{D}^{\kappa(\nu)}$ and $\tilde{\lambda}, \tilde{\mu} \in \mathcal{T}_{b(\nu)}$. By Proposition \ref{remark for supp and lenghts} \ref{remark for supp and lenghts, part e}, we have $\supp(X_{\lambda}) \subset [T_{\nu} \chi_{I} = 1]$ and $\supp(X_{\mu}) \subset [T_{\nu} \chi_{J} = 1]$. Remark \ref{true form of compression}, together with the disjointness of $I$ and $J$, implies that $\supp(X_{\lambda})$ and $\supp(X_{\mu})$ are disjoint.

We prove the other direction by induction on $\alpha$, i.e., in the $\alpha$ inductive step we assume that for every $\beta<\alpha$, any $\lambda, \mu$ in $\mathcal{T}_\beta$  that split at an ordinal are not disjointly supported and we will show this also holds in $\mathcal{T}_\alpha$. The finite case $\alpha = n$ trivially holds, because $\mathcal{T}_n$ is a singleton.

Let $\alpha$ be a limit ordinal, and assume that $\lambda,\mu\in\mathcal{T}_\alpha$ split at an ordinal. We have that $\lambda = \nu^{\frown} \lambda(i_{0})^{\frown} \tilde{\lambda}$ and $\mu= \nu^{\frown} \mu(i_{0})^{\frown} \tilde{\mu}$. From Definition \ref{dfn of ambient tree} we observe that $\lambda(i_{0}-1) = \mu(i_{0}-1)$ is a limit ordinal. Thus, either $\lambda(i_{0}-1) = \alpha$ or $\lambda(i_{0}-1) < \alpha$. For the first case, we have that  $\lambda = (\alpha)^{\frown} \lambda(i_{0})^{\frown} \tilde{\lambda}$ and $\mu= (\alpha)^{\frown} \mu(i_{0})^{\frown} \tilde{\mu}$. Thus, $X_{\lambda} = T_{\lambda(i_{0})}^{\alpha}(X_{(\lambda(i_{0}))^{\frown}\tilde{\lambda}})$ and  $X_{\mu} = T_{\mu(i_{0})}^{\alpha}(X_{(\mu(i_{0}))^{\frown}\tilde{\mu}})$. Recall that $T_{\lambda(i_{0})}^{\alpha}$ and $T_{\mu(i_{0})}^{\alpha}$ have independent ranges. Because $\supp(X_{\lambda}) = [T_{\lambda} h_{[0,1]} \neq 0]$ and $\supp(X_{\mu}) = [T_{\mu} h_{[0,1]} \neq 0]$, they are independent. From Proposition \ref{remark for supp and lenghts} \ref{remark for supp and lenghts, part zero} we have that $\supp(X_{\lambda})$ and $\supp(X_{\mu})$ are of positive measure. Hence, $\lambda$ and $\mu$ are not disjointly supported. For the latter case, we have that $\lambda = \nu^{\frown} \lambda(i_{0})^{\frown} \tilde{\lambda}$ and $\mu= \nu^{\frown} \mu(i_{0})^{\frown} \tilde{\mu}$, where $\lambda(i_{0}-1) = \mu(i_{0}-1) < \alpha$. We have that $\lambda(i_{0}-1)^{\frown} \lambda(i_{0})^{\frown} \tilde{\lambda}$ and $\mu(i_{0}-1)^{\frown} \mu(i_{0})^{\frown} \tilde{\mu}$ belong to $\mathcal{T}_{\lambda(i_{0}-1)}$ and split at an ordinal. By the induction hypothesis, we have that $\lambda(i_{0}-1)^{\frown} \lambda(i_{0})^{\frown} \tilde{\lambda}$ and $\mu(i_{0}-1)^{\frown} \mu(i_{0})^{\frown} \tilde{\mu}$ are not disjointly supported. Combining this with Remark \ref{true form of compression}, we obtain that $\lambda$ and $\mu$ are not disjointly supported.

Now, let $\alpha = \beta + \kappa(\alpha)$ and assume that $\lambda, \mu \in \mathcal{T}_{\alpha}$ split at an ordinal. By the definition of $\mathcal{T}_\alpha$, there exist $I \in \mathcal{D}^{\kappa(\alpha)}$ and $\tilde{\lambda}, \tilde{\mu} \in \mathcal{T}_{\beta}$ such that $\lambda = (\beta + \kappa(\alpha))^{\frown} I^{\frown} \tilde{\lambda}$ and $\mu= (\beta + \kappa(\alpha))^{\frown} I^{\frown} \tilde{\mu}$ and $\tilde{\lambda}$ and $\tilde{\mu}$ split at an ordinal. By the induction hypothesis we have that $\tilde{\lambda}$ and $\tilde{\mu}$ are not disjointly supported. Thus, Remark \ref{true form of compression} implies that $\lambda$ and  $\mu$ are not disjointly supported.
\end{proof}


\section{Distributional embedding schemes}
\label{DES section}

{In this section, we present distributional embedding schemes, a fundamental concept in our paper. For ordinal numbers $\alpha,\beta<\omega_1$, a distributional embedding scheme from $R_\alpha^{p,0}$ to $R_\beta^{p,0}$ is an object $\Psi$ comprising a collection of finite subsets of $\mathcal{T}_\beta$ with a tree structure resembling $\mathcal{T}_\alpha$ and an accompanying collection of distributional embeddings of finite-dimensional $L_p$-spaces. The object $\Psi$ is designed so that it induces a distributional embedding $T_\Psi:R_\alpha^{p,0}\to R_\beta^{p,0}$. In Section 7, this tool will facilitate the stabilization of operators. Although $T_\Psi$ is easily defined as a linear operator, establishing that it preserves the distribution is an extensive process that will be concluded in Section \ref{dist repre subsection}. Here, we will deliver essential estimates related to $\Psi$ and introduce specific sub-distributional embedding schemes of $\Psi$, which are utilized in Section \ref{dist repre subsection}.}

{

We define distributional embedding schemes in a slightly broader setting than that described in the preceding paragraph. Specifically, we work with subtrees of $\mathcal{T}_{\alpha}$ and with the corresponding subspaces of $R_{\alpha}^{p,0}$ that they determine. This degree of generality is unnecessary for the reduction to diagonal operators but becomes essential for the scalar reduction. Although this latter step is not included in the present paper, the additional generality does not alter the proof, apart from minor adjustments in notation.

\begin{dfn}
Let $\alpha<\omega_1$, and let $\mathcal{E}$ be a subset of $\mathcal{T}_\alpha$.
\begin{enumerate}[label=(\alph*)]

\item We call $\mathcal{E}$ a subtree of $\mathcal{T}_\alpha$ if, for every $\lambda_0\in\mathcal{E}$, $\{\lambda\in\mathcal{T}_\alpha:\lambda \sqsubseteq\lambda_0\}\subset\mathcal{E}$.

\item We call $\mathcal{E}$ an interval-branching subtree of $\mathcal{T}_\alpha$ if is it is a subtree and, for every $\lambda_0\in\mathcal{E}$ that is non-maximal in $\mathcal{E}$ and $I\in\mathcal{D}^{\kappa(\lambda_0)}$, there exists $\lambda\in\mathcal{E}$ such that ${\lambda_0}^\smallfrown I\sqsubseteq \lambda$.

\item  We call $\mathcal{E}$ a perfectly branching subtree of $\mathcal{T}_\alpha$ if it is a subtree and, for every $\lambda_0\in\mathcal{E}$ that is non-maximal in $\mathcal{T}_\alpha$ and for every $I\in\mathcal{D}^{\kappa(\lambda_0)}$,
\[\sup\Big\{b(\lambda)+\kappa(\lambda): \lambda\in\mathcal{E}\text{ and }{\lambda_0}^\smallfrown I\sqsubseteq\lambda\Big\} = b(\lambda_0).\]

\end{enumerate}
Given a subtree $\mathcal{E}$ of $\mathcal{T}_\alpha$, we define the subspace $R_\mathcal{E} = [(X_\lambda)_{\lambda\in\mathcal{E}}]$ of $R_{\alpha}^{p,0}$.
\end{dfn}

\begin{rem}
A subtree $\mathcal{E}$ is a downward-closed subset of $\mathcal{T}_\alpha$. It is termed interval-branching if every element $\lambda$ either has no extensions or can be extended along all intervals $I\in\mathcal{D}^{\kappa(\lambda)}$. The class of spaces $R_\mathcal{E}$, where $\mathcal{E}$ ranges over all interval-branching subtrees of all $\mathcal{T}_\alpha$, encompasses the tree subspaces  $X_T^p$ of $L_p$, found in \cite[Section 3]{bourgain:rosenthal:schechtman:1981}, but we will not delve into this. A subtree is perfectly branching if the complexity of its extensions matches that of $\mathcal{T}_\alpha$.

We will prove the main result concerning distributional embedding schemes, namely Theorem \ref{Dist Isom theorem}, in the setting of interval-branching subtrees. For the purposes of our main Theorem \ref{reduction to FDD diagonal} this additional generality is not required, since it would have been sufficient to consider the full trees $\mathcal{E}=\mathcal{T}_\alpha$. Our reason for working in this broader setting is to establish a framework that may later serve in proving the primary factorization property for the spaces $R_{\omega^\alpha}^{p,0}$, $1\le \alpha<\omega_1$. Perfectly branching subtrees were introduced for the same purpose; although they play no role in the present argument and will not be mentioned again, we retain them to maintain continuity with a prospective extension of this work.
\end{rem}

\begin{dfn} \label{DES} 
Let $\alpha, \beta < \omega_{1}$ and consider the subtrees $\mathcal{E}$ of $\mathcal{T}_\alpha$ and $\mathcal{F}$ of $\mathcal{T}_\beta$. Let $\Psi = (\mathcal{M}_{\lambda}, (J_{\mu})_{\mu \in \mathcal{M}_{\lambda}})_{ \lambda \in  \mathcal{E}}$ be a collection such that, for every $\lambda \in  \mathcal{E}$, $\mathcal{M}_{\lambda}$ is a finite subset of $\mathcal{F}$ and, for every $\mu \in \mathcal{M}_{\lambda}$, $J_\mu: L_p^{\kappa(\lambda),0} \to L_p^{\kappa(\mu),0}$ is a distributional embedding.  We call $\Psi$ a distributional embedding scheme of $R_{\mathcal{E}}$ in $R_{\mathcal{F}}$ if the following hold.
\begin{enumerate}[label=(\alph*)]
\item \label{DES part a}  For $\lambda \in  \mathcal{E}$, 
\begin{enumerate}[label=(\roman*)]

\item \label{DES part a, part 1} the members of $\mathcal{M}_{\lambda}$ are pairwise disjointly supported and

\item \label{DES part a, part 2} $\underset{\mu \in \mathcal{M}_{\lambda}}{\sum} \theta_{\mu} = \theta_{\lambda}$.
\end{enumerate}

\item \label{DES part b} If $\lambda_{1}, \lambda_{2} \in \mathcal{E}$ are incomparable and $\mu_{1} \in \mathcal{M}_{\lambda_{1}}$, $\mu_{2} \in \mathcal{M}_{\lambda_{2}}$, then $\mu_{1}, \mu_{2}$ are incomparable.

\item \label{DES part c} Let $\lambda\in\mathcal{E}$ be a non-root, $\lambda_0$ be the immediate predecessor of $\lambda$, and let $I\in\mathcal{D}^{\kappa(\lambda_0)}$ such that ${\lambda_0}^\smallfrown I\sqsubseteq \lambda$. Then, for every $\mu \in \mathcal{M}_{\lambda}$, there exist $\mu_{0} \in \mathcal{M}_{\lambda_{0}}$ and $P\in\mathcal{D}^{\kappa(\mu_{0})}$ such that
\[{\mu_{0}}^\smallfrown P \sqsubseteq  \mu\text{ and }P \subset [J_{\mu_0}h_{\pi (I)} = \epsilon (I)].\]

\end{enumerate}
\end{dfn}

For the definitions of $\pi(I)$ and $\epsilon(I)$, see Notation \ref{haar system notation}.

\begin{rem} \label{lemma for existnce of eleme in th epredic}
A brief explanation of the definition of $\Psi$ is necessary.
\begin{enumerate}[label=(\roman*)]

\item\label{lemma for existnce of eleme in th epredic i} When focusing solely on the tree structure of the subtrees $\mathcal{E}$ and $\mathcal{F}$, then $(\mathcal{M}_{\lambda})_{\lambda\in\mathcal{E}}$ is a disjoint collection of non-empty finite antichains of $\mathcal{F}$, indexed over $\mathcal{E}$, such that, for every $\lambda_{1}, \lambda_{2}\in\mathcal{E}$, $\lambda_1\sqsubsetneq \lambda_2$ if and only if there exists $\mu_1\in\mathcal{M}_{\lambda_1}$, $\mu_2\in\mathcal{M}_{\lambda_2}$ such that $\mu_1\sqsubsetneq \mu_2$. While this aspect helps establish some properties of $\Psi$, $\Psi$ encodes significantly more information required for representing $R_{\mathcal{E}}$ in $R_{\mathcal{F}}$.

\item Each set $\mathcal{M}_\lambda$ will be used to define a distributional copy of $X_\lambda$ in $\langle X_\mu:\mu\in\mathcal{M}_\lambda\rangle$ via $(J_\mu)_{\mu\in\mathcal{M}_\lambda}$. Item \ref{DES part a} asserts that the spaces $( X_\mu)_{\mu\in\mathcal{M}_\lambda}$ are disjointly supported and the measure of the support of their union aligns with the measure of the support of $X_\lambda$.

\item While \ref{DES part b} appears to be an innocent condition related only to the tree structure, when considered alongside the other items, it guarantees the preservation of independence.

\item Item \ref{DES part c} indicates that each $\mu \in \mathcal{M}_\lambda$ is the (not necessarily immediate) successor of some $\mu_0 \in \mathcal{M}_{\lambda_0}$ along an interval $P$. Although it may seem logical to choose $P$ so that $\epsilon(P) = \epsilon(I)$, this perspective is overly simplistic because of the involvement of $J_{\mu_0}$. Therefore, $P$ is restricted by a condition linked to $J_{\mu_0}$ and $\pi(I)$ and $\epsilon(I)$.

\end{enumerate}
\end{rem}

We provide the next piece of notation for context, but we will not address it yet.

\begin{ntt}
\label{notation TPsi}
Fix $\alpha, \beta < \omega_{1}$, subtrees $\mathcal{E}$ of $\mathcal{T}_\alpha$ and $\mathcal{F}$ of $\mathcal{T}_\beta$, and a distributional embedding scheme $\Psi = (\mathcal{M}_{\lambda}, (J_{\mu})_{\mu \in \mathcal{M}_{\lambda}})_{ \lambda \in  \mathcal{E}}$ of $R_{\mathcal{E}}$ in $R_{\mathcal{F}}$. We denote by $T_\Psi:\langle (X_\lambda)_{\lambda\in\mathcal{E}}\rangle\to \langle (X_\mu)_{\mu\in\mathcal{F}}\rangle$ the linear map given by
\[T_\Psi|_{X_\lambda} = \sum_{\mu\in\mathcal{M}_\lambda}T_\mu J_\mu T_{\lambda}^{-1}|_{X_\lambda},\]
for $\lambda\in\mathcal{E}$. That is, for $\lambda\in\mathcal{E}\text{ and }I\in\mathcal{D}_{\kappa(\lambda)-1}$,
\[T_\Psi h_I^\lambda = \sum_{\mu\in\mathcal{M}_\lambda}T_\mu J_\mu h_I.\]
\end{ntt}

\begin{rem}
In Section \ref{dist repre subsection}, we will show that if $\mathcal{E}$ and $\mathcal{F}$ are interval-branching subtrees, then $T_{\Psi}$ is a distributional embedding. The same result also holds for non-interval-branching subtrees, with an identical proof apart from heavier notation needed to distinguish the intervals along which a given $\lambda \in \mathcal{E}$ may extend. Since the proof is already long and no application of the more general statement is known to us, we state and prove the result under the interval-branching assumption.
\end{rem}

\begin{rem}
It is possible to construct distributional embeddings $T : R_{\alpha}^{p,0} \to R_{\alpha}^{p,0}$ that are not induced by distributional embedding schemes. While the definition of a distributional embedding scheme could likely be extended to include all distributional embeddings $T : R_{\alpha}^{p,0} \to R_{\beta}^{p,0}$, we do not pursue this level of generality, as it would complicate the presentation. That said, it would be of independent interest to provide a combinatorial description of all distributional embeddings between Bourgain-Rosenthal-Schechtman spaces.
\end{rem}

\subsection{Estimates on distributional embedding schemes}
We will provide precise estimates for quantities of the form $\sum_{\nu\in\Gamma}\theta_\nu$, for subsets $\Gamma$ of $\mathcal{M}_\lambda$ of a special type stemming from the tree structure of $\Psi$.

\begin{ntt}
For the rest of this subsection, we fix $\alpha, \beta < \omega_{1}$, interval-branching subtrees $\mathcal{E}$ of $\mathcal{T}_\alpha$ and $\mathcal{F}$ of $\mathcal{T}_\beta$, and a distributional embedding scheme $\Psi = (\mathcal{M}_{\lambda}, (J_{\mu})_{\mu \in \mathcal{M}_{\lambda}})_{ \lambda \in  \mathcal{E}}$ of $R_{\mathcal{E}}$ in $R_{\mathcal{F}}$. The partial relation ``$\sqsubseteq$'' should be understood as referring to $\mathcal{E}$ or $\mathcal{F}$, rather than $\mathcal{T}_\alpha$ or $\mathcal{T}_\beta$.
\end{ntt}

 To facilitate the main statement of this subsection, we introduce some notation, which implicitly depends on $\Psi$.

\begin{ntt}
\label{notation for Ds with direction given by interval}
Let $\lambda\in\mathcal{E}$.
\begin{enumerate}[label=(\alph*)]

\item For $\mu\in\mathcal{F}$ denote
\[\Gamma_\lambda^{\mu} = \Big\{\nu\in\mathcal{M}_\lambda: \mu\sqsubseteq \nu\Big\},\]
and, for $J\in\mathcal{D}^{\kappa(\mu)}$, denote
\[\Gamma_\lambda^{{\mu}^\smallfrown J} = \Big\{\nu\in\mathcal{M}_\lambda: {\mu}^\smallfrown J\sqsubseteq \nu\Big\}.\]

\item For $I\in\mathcal{D}^{\kappa(\lambda)}$ and $\mu\in\mathcal{M}_\lambda$ denote
\[\mathcal{D}^{\kappa(\mu)}_I = \Big\{P\in\mathcal{D}^{\kappa(\mu)}: P\subset [J_\mu h_{\pi(I)} = \epsilon(I)]\Big\}.\]

\end{enumerate}

\end{ntt}

The above will be used extensively in the sequel. The reader may realize that $\mathcal{D}^{\kappa(\mu)}_I$ depends on $\mu$, rather than only $\kappa(\mu)$. However, no confusion should arise from this.

\begin{rem}
\label{remark about partitioning a successor}
Based on the above notation, the following hold.
\begin{enumerate}[label=(\alph*)]

\item\label{remark about partitioning a successor a} For $\lambda\in\mathcal{E}$, $I\in\mathcal{D}^{\kappa(\lambda)}$, and $\mu\in\mathcal{M}_\lambda$ we have
\[\sum_{P\in\mathcal{D}^{\kappa(\mu)}_I}|P| = |I|,\]
because $D_\mu h_{\pi(I)}$ is a distributional copy of $h_{\pi(I)}$ in $L_p^{\kappa(\mu),0}$.

\item\label{remark about partitioning a successor b} Definition \ref{DES} \ref{DES part c} is reformulated as follows. Let $\lambda_0\in\mathcal{E}$, $I\in\mathcal{D}^{\kappa(\lambda_0)}$, and $\lambda$ be an immediate successor of $\lambda_0$ with ${\lambda_0}^\smallfrown I\sqsubseteq \lambda$. Then the collections of sets
\[\Big\{\Gamma_\lambda^{\mu^\smallfrown P}: \mu\in\mathcal{M}_{\lambda_0}\text{ and }P\in\mathcal{D}^{\kappa(\mu)}_I\Big\}\]
forms a partition of $\mathcal{M}_\lambda$. By iteration, the same applies if we substitute ``immediate successor'' with ``successor''.

\item\label{remark about partitioning a successor c} There exist variations of \ref{remark about partitioning a successor b}, such as the following. Let $\lambda_0\in\mathcal{E}$, $I\in\mathcal{D}^{\kappa(\lambda_0)}$, and $\lambda$ be a successor of $\lambda_0$ with ${\lambda_0}^\smallfrown I\sqsubseteq \lambda$. Let  $\mu_0\in\mathcal{F}$ and $J\in\mathcal{D}^{\kappa(\mu_0)}$ such that $\Gamma_{\lambda_0}^{{\mu_0}^\smallfrown J}\neq\emptyset$. Then the collections of sets
\[\Big\{\Gamma_\lambda^{\mu^\smallfrown P}: \mu\in\Gamma_{\lambda_0}^{{\mu_0}^\smallfrown J}\text{ and }P\in\mathcal{D}^{\kappa(\mu)}_I\Big\}\]
forms a partition of $\Gamma_\lambda^{{\mu_0}^\smallfrown J}$.

\end{enumerate}
 
\end{rem}

The following is the main initial result about distributional embedding schemes, focusing on specific estimates. It will be used in the next subsection, where we define sub-distributional embedding schemes of $\Psi$.

\begin{prop}
\label{such and such}
Fix $\mu_0\in\mathcal{F}$ and let $\lambda_0\in\mathcal{E}$ be minimal such that $\Gamma_{\lambda_0}^{\mu_0}\neq\emptyset$, assuming it exists. The following hold.
\begin{enumerate}[label=(\alph*)]

\item\label{such and such a} For every $\lambda \sqsupseteq\lambda_0$,
\[\sum_{\nu\in\Gamma_{\lambda}^{\mu_0}}\theta_\nu = \theta_{\mu_0}\frac{\theta_\lambda}{\theta_{\lambda_0}}.\]

\item\label{such and such b} If $\mu_0\notin\mathcal{M}_{\lambda_0}$ then, for every $J\in\mathcal{D}^{\kappa(\mu_0)}$ and $\lambda \sqsupseteq\lambda_0$,
\[\sum_{\nu\in\Gamma_{\lambda}^{{\mu_0}^\smallfrown J}}\theta_\nu = \theta_{\mu_0}|J|\frac{\theta_\lambda}{\theta_{\lambda_0}}.\]

\item\label{such and such c} If $\mu_0\in\mathcal{M}_{\lambda_0}$ then, for every $I\in\mathcal{D}^{\kappa(\lambda_0)}$ and $J\in\mathcal{D}^{\kappa(\mu_0)}_I$,  and for every  $\lambda \sqsupseteq {\lambda_0}^\smallfrown I$,
\[\sum_{\nu\in\Gamma_{\lambda}^{{\mu_0}^\smallfrown J}}\theta_\nu = \theta_{\mu_0}\frac{|J|}{|I|}\frac{\theta_\lambda}{\theta_{\lambda_0}}.\]

\end{enumerate}
\end{prop}

We will prove Proposition \ref{such and such} in steps, by first showing \ref{such and such c}, then \ref{such and such b}, and finally we will conclude \ref{such and such a}. We start with a lemma.

\begin{lem}
\label{measure bounded by disjointness and inclusion}
Let $\Gamma$ be a set of disjointly supported members of $\mathcal{F}$.
\begin{enumerate}[label=(\alph*)]

\item\label{measure bounded by disjointness and inclusion a} Assume that there is $\mu\in\mathcal{F}$ such that, for all $\nu\in\Gamma$, $\mu\sqsubseteq \nu$. Then,
\[\sum_{\nu\in\Gamma}\theta_\nu\leq\theta_\mu.\]

\item\label{measure bounded by disjointness and inclusion b} Assume that there are $\mu\in\mathcal{F}$ and $I\in\mathcal{D}^{\kappa(\mu)}$ such that, for all $\nu\in\Gamma$, $\mu^\smallfrown J\sqsubseteq \nu$. Then,
\[\sum_{\nu\in\Gamma}\theta_\nu\leq\theta_\mu|J|.\]

\end{enumerate}
\end{lem}

\begin{proof}
The first statement directly follows from Proposition \ref{remark for supp and lenghts} \ref{remark for supp and lenghts, part e}, according to which, $\mathrm{supp}(X_\nu)$, $\nu\in\Gamma$ are disjoint subsets of $\mathrm{supp}(X_\mu)$, and thus,
\[\underset {\nu \in \Gamma}{\sum} \theta_{\nu} = \underset {\nu \in \Gamma}{\sum}  \vert \supp(X_{\nu}) \vert = \bigg| \underset {\nu \in \Gamma}{\bigcup}  \supp(X_{\nu}) \bigg| \leq |\mathrm{supp}(X_\mu)| = \theta_\mu.\]

The proof of the second statement is a variation of the above, also involving Proposition \ref{remark for supp and lenghts} \ref{remark for supp and lenghts, part e} stating, for $\nu\in\Gamma$, $\supp(X_\nu)\subset [T_\mu\chi_J = 1]$. Therefore,
\[\underset {\nu \in \Gamma}{\sum} \theta_{\nu} = \underset {\nu \in \Gamma}{\sum}  \vert \supp(X_{\nu}) \vert = \bigg| \underset {\nu \in \Gamma}{\bigcup}  \supp(X_{\nu}) \bigg| \leq |[T_\mu\chi_J = 1]| = \theta_\mu|J|,\]
because  $T_\mu$ is a $\theta_\mu$-compression.
\end{proof}

\begin{proof}[Proof of Proposition \ref{such and such} \ref{such and such c}]
We proceed by induction on $n = h(\lambda) - h(\lambda_0)$, i.e., the path distance from $\lambda_0$ to $\lambda$ in the tree $\mathcal{E}$. To elaborate further, the inductive hypothesis in the $n$'th step will state that the conclusion holds for all $\mu_0$, $\lambda_0$, and $\lambda$, as outlined, for which $h(\lambda) - h(\lambda_0) < n$.  We initiate with the base case, $n=1$. Let  $\lambda_0\in\mathcal{E}$, $\mu_0\in\mathcal{M}_{\lambda_0}$, $I\in\mathcal{D}^{\kappa(\lambda_0)}$ and $J\in\mathcal{D}^{\kappa(\mu_0)}_I$, and let $\lambda$ be an immediate successor of $\lambda_0$ with ${\lambda_0}^\smallfrown I \sqsubseteq \lambda$. Note that, in particular, $\theta_\lambda = \theta_{\lambda_0}|I|$. We will prove
\[\sum_{\nu\in\Gamma_{\lambda}^{{\mu_0}^\smallfrown J}}\theta_\nu = \theta_{\mu_0}|J| = \theta_{\mu_0}\frac{|J|}{|I|}\frac{\theta_\lambda}{\theta_{\lambda_0}}.\]
By Lemma \ref{measure bounded by disjointness and inclusion} \ref{measure bounded by disjointness and inclusion b}, for any $\mu\in\mathcal{M}_\lambda$ and $P\in\mathcal{D}^{\kappa(\mu)}_I$, we have
\[\sum_{\nu\in\Gamma_{\lambda}^{{\mu}^\smallfrown P}}\theta_\nu \leq \theta_{\mu}|P|.\]
If we assume, towards a contradiction, that $\sum_{\nu\in\Gamma_{\lambda}^{{\mu_0}^\smallfrown J}}\theta_\nu < \theta_{\mu_0}|J|$, then
\begin{align*}
\theta_\lambda  &= \sum_{\nu\in\mathcal{M}_\lambda}\theta_\nu \text{ (by Definition \ref{DES} \ref{DES part a})}\\
&=\sum_{\mu\in\mathcal{M}_{\lambda_0}}\sum_{P\in\mathcal{D}^{\kappa(\mu)}_I}\sum_{\nu\in \Gamma^{\mu^\smallfrown P}_{\lambda_0}}\theta_\nu\text{ (by Remark \ref{remark about partitioning a successor} \ref{remark about partitioning a successor b})}\\
&< \sum_{\mu\in\mathcal{M}_{\lambda_0}}\theta_\mu\sum_{P\in\mathcal{D}^{\kappa(\mu)}_I}|P| = \sum_{\mu\in\mathcal{M}_{\lambda_0}}\theta_\mu|I| \text{ (by Remark \ref{remark about partitioning a successor} \ref{remark about partitioning a successor a})}\\
& = \theta_{\lambda_0}|I| = \theta_\lambda.
\end{align*}

We now perform the $n$'th inductive step. Let  $\lambda_0\in\mathcal{T}_\alpha$, $\mu_0\in\mathcal{M}_{\lambda_0}$, $I\in\mathcal{D}^{\kappa(\lambda_0)}$ and $J\in\mathcal{D}^{\kappa(\mu_0)}_I$, and let $\lambda\sqsupseteq{\lambda_0}^\smallfrown I$ with $h(\lambda) - h(\lambda_0) = n$. Let $\lambda_{n-1}$ denote the immediate predecessor of $\lambda$ and let $I_n\in\mathcal{D}^{\kappa(\lambda_{n-1})}$ such that ${\lambda_{n-1}}^\smallfrown I_n\sqsubseteq \lambda$. By Remark \ref{remark about partitioning a successor} \ref{remark about partitioning a successor c}, the sets
\[\Big\{\Gamma_\lambda^{\mu^\smallfrown P}:\mu\in\Gamma_{\lambda_{n-1}}^{{\mu_0}^\smallfrown J}\text{ and }P\in\mathcal{D}^{\kappa(\mu)}_{I_n}\Big\}\]
form a partition of $\Gamma_{\lambda}^{{\mu_0}^\smallfrown J}$, and thus,
\begin{align*}
\sum_{\nu\in\Gamma_{\lambda}^{{\mu_0}^\smallfrown J}}\theta_\nu = \sum_{\mu\in\Gamma_{\lambda_{n-1}}^{{\mu_0}^\smallfrown J}}\sum_{P\in\mathcal{D}^{\kappa(\mu)}_{I_n}}\sum_{\nu\in\Gamma_{\lambda}^{{\mu}^\smallfrown P}}\theta_\nu.
\end{align*}
By the base case, we have that, for each $\mu\in\Gamma_{\lambda_{n-1}}^{{\mu_0}^\smallfrown J}$ and $P\in\mathcal{D}^{\kappa(\mu)}_{I_n}$,
\[\sum_{\nu\in\Gamma_{\lambda}^{{\mu}^\smallfrown P}}\theta_\nu = \theta_\mu|P|,\]
and therefore
\begin{align*}
\sum_{\nu\in\Gamma_{\lambda}^{{\mu_0}^\smallfrown J}}\theta_\nu &= \sum_{\mu\in\Gamma_{\lambda_{n-1}}^{{\mu_0}^\smallfrown J}}\theta_\mu\sum_{P\in\mathcal{D}^{\kappa(\mu)}_{I_n}}|P| = \sum_{\mu\in\Gamma_{\lambda_{n-1}}^{{\mu_0}^\smallfrown J}}\theta_\mu|I_n| \text{ (by Remark \ref{remark about partitioning a successor} \ref{remark about partitioning a successor a})}\\
&= \Big(\theta_{\mu_0}\frac{|J|}{|I|}\frac{\lambda_{n-1}}{\lambda_0}\Big)|I_n| \text{ (by the case $n-1$)}\\
&= \theta_{\mu_0}\frac{|J|}{|I|}\frac{\lambda_{n}}{\lambda_0}.
\end{align*}
\end{proof}

Before proving Proposition \ref{such and such} \ref{such and such b}, we establish a lemma that will also be used frequently later, in Section \ref{dist repre subsection}.

\begin{lem}
\label{roots associated to M lambda}
Let $\lambda\in\mathcal{E}$.
\begin{enumerate}[label=(\alph*)]

\item\label{roots associated to M lambda a} There exists a unique root $\mu\in\mathcal{F}$ such that $\Gamma^{\mu}_\lambda = \mathcal{M}_\lambda$.

\item\label{roots associated to M lambda b} Let $\mu\in\mathcal{F}$ and $J\in\mathcal{D}^{\kappa(\mu)}$ such that $\Gamma_\lambda^{\mu^\smallfrown J}\neq\emptyset$. Then there exists a unique root $\xi\in\mathcal{T}_{b(\mu)}$ such that $\mu^\smallfrown J^\smallfrown\xi\in\mathcal{F}$ and $\Gamma_\lambda^{\mu^\smallfrown J} = \Gamma_\lambda^{\mu^\smallfrown J^\smallfrown\xi}$.

\end{enumerate}
\end{lem}

\begin{proof}
The two statements are shown with virtually the same argument. We start with statement \ref{roots associated to M lambda a}. We fix an arbitrary $\nu\in\mathcal{M}_\lambda$ and let $\mu$ be the unique root such that $\mu\sqsubseteq\nu$. By Remark \ref{characterization  of im suc} \ref{characterization  of im suc, part 2} all the entries of $\mu$ are ordinal numbers. For $\nu'\in\mathcal{M}_\lambda\setminus\{\nu\}$, we will show $\nu'\in\Gamma_\lambda^\mu$. Because $\nu,\nu'$ are disjointly supported, by Proposition \ref{incomparable and disj supp character via interval}, for the minimum entry $i$ for which $\nu(i)\neq\nu'(i)$, $\nu(i)$ is some interval $I$, and since $\mu\sqsubseteq \nu$, $i>|\mu|$. In conclusion, $\nu'|_{|\mu|} = \nu|_{|\mu|} = \mu$.

We now demonstrate \ref{roots associated to M lambda b}. Fix an arbitrary $\nu\in\Gamma_\lambda^{\mu^\smallfrown J}$. Because $\mu^\smallfrown J\sqsubseteq \nu$, there exists a unique root $\xi\in\mathcal{T}_{b(\mu)}$ such that $\mu^\smallfrown J^\smallfrown\xi\sqsubseteq\nu$. In particular, for $1\leq i\leq |\xi|$, $\xi(i)$ is an ordinal number. Clearly, $\Gamma_\lambda^{\mu^\smallfrown J^\smallfrown\xi} \subset \Gamma_\lambda^{\mu^\smallfrown J}$. For $\nu'\in\Gamma_\lambda^{\mu^\smallfrown J}\setminus\{\nu\}$, we will show $\nu'\in\Gamma_\lambda^{\mu^\smallfrown J^\smallfrown\xi}$. Let $i$ denote the minimum entry for which $\nu(i)\neq\nu'(i)$. Because $\nu,\nu'\in\Gamma_\lambda^{\mu^\smallfrown J}$, $i\notin\{1,\ldots, |\mu|+1\}$. Since $\nu,\nu'$ are disjointly supported, for the minimum entry $i$ for which $\nu(i)\neq\nu'(i)$, $\nu(i)$ is not an ordinal number, and thus, $i\notin\{|\mu|+2,\ldots,|\mu|+1+|\xi|\}$. Therefore, $\nu'|_{|\mu|+1+|\xi|} = \nu|_{|\mu|+1+|\xi|} =\mu^\smallfrown J^\smallfrown \xi$.
\end{proof}

\begin{proof}[Proof of Proposition \ref{such and such} \ref{such and such b}]
We will first prove the conclusion only in the case $\lambda = \lambda_0$. More precisely, we will prove that for every $\mu_0\in\mathcal{F}$, for which there is a minimal $\lambda_0\in\mathcal{T}_\alpha$ with $\Gamma^{\mu_0}_{\lambda_0}\neq\emptyset$, if $\mu_0\notin\mathcal{M}_{\lambda_0}$ then, for every $J\in\mathcal{D}^{\kappa(\mu_0)}$, we have
\begin{equation}
\label{such and such proof b equation 1}
\sum_{\nu\in\Gamma^{{\mu_0}^\smallfrown J}_{\lambda_0}} \theta_\nu = \theta_{\mu_0}|J|.
\end{equation}
This will be proved by induction on $h(\mu_0)$, i.e., the height of $\mu_0$ in the tree $\mathcal{F}$. In the base case, $\mu_0$ is a root. By Remark \ref{lemma for existnce of eleme in th epredic} \ref{lemma for existnce of eleme in th epredic i}, $\lambda_0$ is a root, and by Lemma \ref{roots associated to M lambda} \ref{roots associated to M lambda a}, $\Gamma_{\lambda_0}^{\mu_0} = \mathcal{M}_{\lambda_0}$. We may assume $\mu_0\notin\mathcal{M}_{\lambda_0}$, as, otherwise, the conclusion vacuously holds. This, in particular, suggests that the collection of sets
\[\Big\{\Gamma_{\lambda_0}^{{\mu_0}^\smallfrown P}:P\in\mathcal{D}^{\kappa(\mu_0)}\Big\}\]
form a partition of $\Gamma_{\lambda_0}^{\mu_0} = \mathcal{M}_{\lambda_0}$. By Lemma \ref{measure bounded by disjointness and inclusion}, for any $J\in\mathcal{D}^{\kappa(\mu_0)}$,
\[\sum_{\nu\in\Gamma_{\lambda_0}^{{\mu_0}^\smallfrown J}}\theta_\nu \leq \theta_{\mu_0}|J| = |J|.\]
If, for some $J\in\mathcal{D}^{\kappa(\mu_0)}$, we had $\sum_{\nu\in\Gamma_{\lambda_0}^{{\mu_0}^\smallfrown J}}\theta_\nu <|J|$, then,
\[1 = \theta_{\lambda_0} = \sum_{\nu\in\mathcal{M}_{\lambda_0}}\theta_\nu = \sum_{\nu\in\Gamma_{\lambda_0}^{{\mu_0}}}\theta_\nu = \sum_{J\in\mathcal{D}^{\kappa(\mu_0)}} \sum_{\nu\in\Gamma_{\lambda_0}^{{\mu_0}^\smallfrown J}}\theta_\nu< \sum_{J\in \mathcal{D}^{\kappa(\mu_0)}}|J|=1.\]
This contradiction concludes the base case. Assume now that $\mu_0$ is a non-root with immediate predecessor $\mu$, for which the conclusion is assumed to hold. Let $P_0\in\mathcal{D}^{\kappa(\mu)}$ such that $\mu^\smallfrown P_0\sqsubseteq \mu_0$. By Lemma \ref{roots associated to M lambda} \ref{roots associated to M lambda b}, we have
\begin{equation}
\label{such and such proof b equation 2}
\sum_{\nu\in \Gamma_{\lambda_0}^{\mu_0}}\theta_\nu = \sum_{\nu\in \Gamma_{\lambda_0}^{\mu^\smallfrown P_0}}\theta_\nu.
\end{equation}

\noindent {\bf Claim:}
\[\sum_{\nu\in \Gamma_{\lambda_0}^{\mu_0}}\theta_\nu = \theta_{\mu_0}.\]

To establish the claim, we distinguish two cases regarding $\mu$, the immediate predecessor of $\mu_0$. In the first case, $\lambda_0$ is minimal such that $\Gamma_{\lambda_0}^\mu\neq\emptyset$. Then, by the inductive hypothesis,
\[\sum_{\nu\in \Gamma_{\lambda_0}^{\mu_0}}\theta_\nu = \sum_{\nu\in \Gamma_{\lambda_0}^{\mu^\smallfrown P_0}}\theta_\nu = \theta_\mu|P_0| = |\theta_{\mu_0}|.\]
In the second case, the immediate predecessor $\lambda_{p}$ of $\lambda_0$ is minimal such that $\Gamma_{\lambda_{p}}^{\mu}\neq\emptyset$, and if this is the case, $\mu\in\Gamma_{\lambda_{p}}^\mu$. Let $I_0\in\mathcal{D}^{\kappa(\lambda_p)}$ such that ${\lambda_p}^\smallfrown I_0\sqsubseteq \lambda_0$. By Proposition \ref{such and such} \ref{such and such c},
\[\sum_{\nu\in \Gamma_{\lambda_0}^{\mu_0}}\theta_\nu = \sum_{\nu\in \Gamma_{\lambda_0}^{\mu^\smallfrown P_0}}\theta_\nu = \theta_\mu|P_0|\frac{\theta_{\lambda_0}}{|I_0|\theta_{\lambda_p}} = |\theta_{\mu_0}|.\]
The proof of the claim is complete.

We use the claim as follows. By Lemma \ref{measure bounded by disjointness and inclusion}, for any $J\in\mathcal{D}^{\kappa(\mu_0)}$,
\[\sum_{\nu\in\Gamma_{\lambda_0}^{{\mu_0}^\smallfrown J}}\theta_\nu \leq \theta_{\mu_0}|J|.\]
If, for some $J\in\mathcal{D}^{\kappa(\mu_0)}$, we had $\sum_{\nu\in\Gamma_{\lambda_0}^{{\mu_0}^\smallfrown J}}\theta_\nu <|J|\theta_{\mu_0}$, then,
\[\theta_{\mu_0} = \sum_{\nu\in\Gamma_{\lambda_0}^{{\mu_0}}}\theta_\nu = \sum_{J\in\mathcal{D}^{\kappa(\mu_0)}}\sum_{\nu\in\Gamma_{\lambda_0}^{{\mu_0}^\smallfrown J}}\theta_\nu < \sum_{J\in\mathcal{D}^{\kappa(\mu_0)}}|J|\theta_{\mu_0}| = \theta_{\mu_0}.\]
This contradiction concludes the proof of the case $\lambda = \lambda_0$. 

We finally prove the general case. Let $\mu_0\in\mathcal{F}$ and let $\lambda_0\in\mathcal{E}$ be minimal such that $\Gamma^{\mu_0}_{\lambda_0}\neq\emptyset$, assuming it exists, and assume $\mu_0\notin\mathcal{M}_{\lambda_0}$, as, otherwise, the conclusion vacuously holds. Let $J\in\mathcal{D}^{\kappa(\mu_0)}$ and $\lambda \sqsupsetneq\lambda_0$. Let $I\in\mathcal{D}^{\kappa(\lambda_0)}$ such that $\lambda\sqsupseteq {\lambda_0}^\smallfrown I$, and note that, by Remark \ref{remark about partitioning a successor} \ref{remark about partitioning a successor c}, the sets
\[\Big\{\Gamma_\lambda^{{\mu}^\smallfrown P}:\mu\in\Gamma^{{\mu_0}^\smallfrown J}_{\lambda_0}\text{ and }P\in\mathcal{D}^{\kappa(\mu)}_I\Big\}\]
form a partition of $\Gamma_{\lambda}^{{\mu_0}^\smallfrown J}$. We now compute
\begin{align*}
\sum_{\nu\in\Gamma_{\lambda}^{{\mu_0}^\smallfrown J}}\theta_\nu &= \sum_{\mu\in\Gamma^{{\mu_0}^\smallfrown J}_{\lambda_0}}\sum_{P\in \mathcal{D}^{\kappa(\mu)}_I}\sum_{\nu\in\Gamma^{\mu^\smallfrown P}_\lambda}\theta_\nu = \sum_{\mu\in\Gamma^{{\mu_0}^\smallfrown J}_{\lambda_0}}\sum_{P\in \mathcal{D}^{\kappa(\mu)}_I} \theta_\mu\frac{|P|}{|I|}\frac{|\lambda|}{|\lambda_0|} \text{ (by Proposition \ref{such and such} \ref{such and such c})}\\
& = \sum_{\mu\in\Gamma^{{\mu_0}^\smallfrown J}_{\lambda_0}}\theta_\mu \frac{|I|}{|I|}\frac{|\lambda|}{|\lambda_0|}\text{ (by Remark \ref{remark about partitioning a successor} \ref{remark about partitioning a successor a})}\\
& = \theta_{\mu_0}|J|\frac{|\lambda|}{|\lambda_0|} \text{ (by the case $\lambda = \lambda_0$)}.
\end{align*}

\end{proof}

\begin{rem}
Proposition \ref{such and such} \ref{such and such a} follows from the argument in the Claim inside the proof of Proposition \ref{such and such} \ref{such and such b}.
\end{rem}

{
\subsection{Sub-distributional embedding schemes associated with $\Psi$} \label{new from old subsection}
We construct three new types of distributional embedding schemes by isolating substructures of $\Psi$ emanating from nodes $\lambda_0\in\mathcal{E}$ and $\mu_0\in\mathcal{F}$. These new schemes have lower ordinal numbers associated with them. This reduction in complexity is important, and it is used in Section \ref{dist repre subsection} to prove, by induction on $\alpha$, that $T_\Psi$ is a distributional embedding.

\begin{ntt}
Let $\alpha<\omega_1$ and $\mathcal{E}$ be an interval-branching subtree  of $\mathcal{T}_\alpha$. For $\lambda_0\in\mathcal{E}$ and $I\in\mathcal{D}^{\kappa(\lambda)}$, we define the interval-branching subtrees
\[\mathcal{E}^{{\lambda_0}^\smallfrown I} = \Big\{\lambda\in\mathcal{T}_{b(\lambda_0)}: \lambda_{0}^{\frown} I^{\frown} \lambda \in\mathcal{E}\Big\}\text{ and }\mathcal{E}^{{\lambda_0}} = \Big\{\lambda\in\mathcal{T}_{b(\lambda_0)+\kappa(\lambda_0)}: \lambda_{0}\oslash \lambda \in\mathcal{E}\Big\}\]
of $\mathcal{T}_{b(\lambda_0)}$ and $\mathcal{T}_{b(\lambda_0)+\kappa(\lambda_0)}$, respectively.
\end{ntt}

\begin{ntt}
For the remainder of this subsection, we fix $\alpha, \beta < \omega_{1}$, interval-branching subtrees $\mathcal{E}$ of $\mathcal{T}_\alpha$ and $\mathcal{F}$ of $\mathcal{T}_\beta$, and a distributional embedding scheme $\Psi = (\mathcal{M}_{\lambda}, (J_{\mu})_{\mu \in \mathcal{M}_{\lambda}})_{ \lambda \in  \mathcal{E}}$ of $R_{\mathcal{E}}$ in $R_{\mathcal{F}}$.
\end{ntt}

In the following proposition, we construct the first type of a sub-distributional embedding scheme of $\Psi$, which is used in the successor case of the inductive step in Section \ref{dist repre subsection}. 
 
\begin{prop} \label{first dist from old one}
Let $\lambda_{0} \in  \mathcal{E}$, $\mu_0\in\mathcal{M}_{\lambda_0}$, $I \in \mathcal{D}^{\kappa(\lambda_{0})}$, and $J \in  \mathcal{D}^{\kappa(\mu_{0})}_I$. For every $\lambda \in  \mathcal{E}^{{\lambda_0}^\smallfrown I}$, denote 
\begin{align*}
 \mathcal{M}^{({\lambda_0}^\smallfrown I,{\mu_0}^\smallfrown J)}_{\lambda} = \lbrace \mu \in \mathcal{F}^{{\mu_0}^\smallfrown J}: \mu_{0}^{\frown} J^{ \frown} \mu \in \mathcal{M}_{\lambda_{0}^{\frown} I^{\frown} \lambda } \rbrace
\end{align*}
and, for every $\mu \in \mathcal{M}^{({\lambda_0}^\smallfrown I,{\mu_0}^\smallfrown J)}_{\lambda}$, denote
\begin{align*}
J^{({\lambda_0}^\smallfrown I,{\mu_0}^\smallfrown J)}_{\mu} = J_{\mu_{0}^{ \frown} J^{ \frown} \mu}.
\end{align*}
Then the collection 
\begin{align*}
\Psi_{{\lambda_0}^\smallfrown I}^{{\mu_0}^\smallfrown J} = (\mathcal{M}^{({\lambda_0}^\smallfrown I,{\mu_0}^\smallfrown J)}_{\lambda}, (J^{({\lambda_0}^\smallfrown I,{\mu_0}^\smallfrown J)}_{\mu})_{\mu \in \mathcal{M}^{({\lambda_0}^\smallfrown I,{\mu_0}^\smallfrown J)}_{\lambda}})_{ \lambda \in  \mathcal{E}^{{\lambda_0}^\smallfrown I}}
\end{align*}
is a distributional embedding scheme of $R_{\mathcal{E}^{{\lambda_0}^\smallfrown I}}$ in  $R_{\mathcal{F}^{{\mu_0}^\smallfrown J}}$.
\end{prop}

To aid the proof of Proposition \ref{first dist from old one}, we note the following observation, which is evident from the relevant definitions.

\begin{rem}
\label{bijection between M and Gamma}
For every $\lambda\in\mathcal{E}^{{\lambda_0}^\smallfrown I}$, considering the set $\Gamma_{\lambda} = \Gamma^{{\mu_0}^\smallfrown J}_{{\lambda_0}^\smallfrown I^\smallfrown \lambda}$, the map
\[\phi_\lambda: \mathcal{M}^{({\lambda_0}^\smallfrown I,{\mu_0}^\smallfrown J)}_\lambda\to \Gamma_\lambda\]
given by $\phi_\lambda(\mu) = {\mu_0}^\smallfrown J^\smallfrown \mu$ is a bijection and, for all $\mu\in\mathcal{M}^{({\lambda_0}^\smallfrown I,{\mu_0}^\smallfrown J)}_\lambda$,
\[\theta_{\phi_\lambda(\mu)} = \theta_{\mu_0}|J|\theta_\mu.\]
\end{rem}

\begin{proof}[Proof of Proposition \ref{first dist from old one}]
For the rest of the poof, for $\lambda \in  \mathcal{E}^{{\lambda_0}^\smallfrown I}$, we denote $\mathcal{M}_\lambda^{({\lambda_0}^\smallfrown I,{\mu_0}^\smallfrown J)} = \mathcal{M}_\lambda'$ and, for $\mu\in\mathcal{M}_\lambda'$, we denote $J_\mu^{({\lambda_0}^\smallfrown I,{\mu_0}^\smallfrown J)} = J_\mu'$. For $\lambda\in\mathcal{E}^{{\lambda_0}^\smallfrown I}$, have that  $\mathcal{M}'_{\lambda}$ is finite and $\mathcal{M}'_{\lambda} \subset  \mathcal{F}^{{\mu_{0}}^\smallfrown J}$. Clearly, for every $\mu \in \mathcal{M^{\prime}}_{\lambda}$ we have that $J'_{\mu}:L_{p}^{\kappa(\lambda),0}\to L_{p}^{\kappa(\mu),0}$ is distributional embedding. We need to verify Definition \ref{DES} \ref{DES part a} \ref{DES part a, part 1} and \ref{DES part a, part 2}, \ref{DES part b}, and \ref{DES part c}.

First, we establish Definition \ref{DES} \ref{DES part a} \ref{DES part a, part 1}, so, let $\lambda \in  \mathcal{E}^{{\lambda_0}^\smallfrown I}$ and $\mu, \mu^{\prime} \in \mathcal{M}'_{\lambda}$ with $\mu \neq \mu^{\prime}$. Hence, we have that $\mu_{0}^{\frown} J^{ \frown} \mu \neq \mu_{0}^{\frown} J^{ \frown} \mu^{\prime} \in \mathcal{M}_{\lambda_{0}^{\frown} I^{\frown} \lambda}$. Thus, $\mu_{0}^{\frown} J^{ \frown} \mu$ and $\mu_{0}^{\frown} J^{ \frown} \mu^{\prime}$ are disjointly supported. Applying Proposition \ref{incomparable and disj supp character via interval} \ref{incomparable and disj supp character via interval, part 2} we easily obtain that $\mu$ and $\mu^{\prime}$ are disjointly supported. The property in Definition \ref{DES} \ref{DES part b} is a clear consequence of the tree structure of the distributional described in Remark \ref{lemma for existnce of eleme in th epredic} \ref{lemma for existnce of eleme in th epredic i}. We use Remark \ref{bijection between M and Gamma} and Proposition \ref{such and such} \ref{such and such c} to verify Definition \ref{DES} \ref{DES part b} as follows. For $\lambda\in\mathcal{E}^{{\lambda_0}^\smallfrown I}$,
\[\sum_{\mu\in\mathcal{M}'_\lambda}\theta_\mu = \theta_{\mu_0}^{-1}|J|^{-1}\sum_{\nu\in \Gamma_\lambda}\theta_\nu = \theta_{\mu_0}^{-1}|J|^{-1}\theta_{\mu_0}\frac{|J|}{|I|}\frac{\theta_{{\lambda_0}^\smallfrown I^\smallfrown \lambda}}{\theta_{\lambda_0}} = \theta_\lambda.\]

Finally, we prove that Definition \ref{DES} \ref{DES part c} is satisfied. Let $\lambda \in \mathcal{E}^{{\lambda_0}^\smallfrown I}$ be an immediate successor of $\pi(\lambda)$ and $J \in \mathcal{D}^{\kappa( \pi(\lambda) )}$ such that $\pi(\lambda)^\smallfrown J\sqsubseteq \lambda$. We will show that, for $\mu \in \mathcal{M}^{\prime}_{\lambda}$, there exist $\nu \in \mathcal{M}^{\prime}_{\pi( \lambda )}$ and $Q\in\mathcal{D}_J^{\kappa(\nu)}$ such that ${\nu}^\smallfrown Q\sqsubseteq \mu$. It is clear that the immediate predecessor of $\bar\lambda = {\lambda_0}^\smallfrown I^\smallfrown \lambda$ is $\pi(\bar\lambda) = {\lambda_0}^\smallfrown I^\smallfrown \pi(\lambda)$ and $\pi(\bar\lambda)^\smallfrown J\sqsubseteq\bar\lambda$. By Remark \ref{remark about partitioning a successor} \ref{remark about partitioning a successor c}, the collection of sets
\[\Big\{\Gamma_{\bar\lambda}^{\nu^\smallfrown Q}:\nu\in\Gamma^{{\mu_0}^\smallfrown J}_{\pi(\bar \lambda)}\text{ and }Q\in\mathcal{D}^{\kappa(\nu)}_J\Big\}\]
form a partition of $\Gamma_{\bar\lambda}^{{\mu_0}^\smallfrown J}$. Therefore, for every $\mu\in\mathcal{M}'_\lambda$, there exists $\nu\in\mathcal{M}'_{\pi(\lambda)}$ and $Q\in \mathcal{D}^{\kappa(\phi_{\pi(\lambda)}(\nu))}_J = \mathcal{D}^{\kappa(\nu)}_J$ such that ${\phi_{\pi(\lambda)}(\nu)}^\smallfrown Q\sqsubseteq \phi_\lambda(\mu)$, and thus, $\nu^\smallfrown J\sqsubseteq \mu$.
\end{proof}

The last two types of distributional embedding schemes from $\Psi$ are utilized in the limit case of the inductive step in Section \ref{dist repre subsection}. The proofs closely resemble that of Proposition \ref{first dist from old one}. Although we highlight some differences, we choose to omit specific details to avoid repetition.

\begin{prop} \label{second DES from old one}
Let $\lambda_{0} \in  \mathcal{E}$ be a root, $\mu_0\in\mathcal{F}$, and $J\in\mathcal{D}^{\kappa(\mu_0)}$ such that $\Gamma^{{\mu_0}^\smallfrown J}_{\lambda_0}\neq\emptyset$. For every $\lambda \in \mathcal{E}^{{\lambda_0}}$, denote  
\begin{align*}
 \mathcal{M}^{(\lambda_0,{\mu_0}^\smallfrown J)}_{\lambda} = \lbrace \mu \in \mathcal{F}^{{\mu_0}^\smallfrown J}: {\mu_0}^\smallfrown {J}^\smallfrown \mu \in \mathcal{M}_{\lambda_{0} \oslash \lambda } \rbrace
\end{align*}
and, for every $\mu \in  \mathcal{M}^{(\lambda_0,{\mu_0}^\smallfrown J)}_{\lambda}$, denote 
\begin{align*}
J^{(\lambda_0,{\mu_0}^\smallfrown J)}_{\mu} = J_{{\mu_0}^\smallfrown {J}^\smallfrown \mu}.
\end{align*}
Then the collection 
\begin{align*}
\Psi_{\lambda_0}^{{\mu_0}^\smallfrown J} = \left( \mathcal{M}_{\lambda}^{(\lambda_0,{\mu_0}^\smallfrown J)}, ( J^{(\lambda_0,{\mu_0}^\smallfrown J)}_{\mu})_{\mu \in \mathcal{M}^{(\lambda_0,{\mu_0}^\smallfrown J)}_{\lambda}} \right)_{ \lambda \in  \mathcal{E}^{{\lambda_0}}}
\end{align*}
is a distributional embedding scheme of $R_{\mathcal{E}^{{\lambda_0}}}$ in  $R_{\mathcal{F}^{{\mu_0}^\smallfrown J}}$.
\end{prop}

The key distinction between the proof of Proposition \ref{second DES from old one} from Proposition \ref{first dist from old one} lies in its dependence on Proposition \ref{such and such} \ref{such and such b} instead of \ref{such and such c}. We point out an observation that is analogous to Remark \ref{bijection between M and Gamma}.

\begin{rem}
\label{second bijection between M and Gamma}
For every $\lambda\in\mathcal{E}^{{\lambda_0}}$, considering the set $\Gamma_{\lambda} = \Gamma^{{\mu_0}^\smallfrown J}_{\lambda_0\oslash \lambda}$, the map
\[\phi_\lambda: \mathcal{M}^{(\lambda_0,{\mu_0}^\smallfrown J)}_\lambda\to \Gamma_\lambda\]
given by $\phi_\lambda(\mu) = {\mu_0}^\smallfrown J^\smallfrown \mu$ is a bijection and, for all $\mu\in\mathcal{M}^{(\lambda_0,{\mu_0}^\smallfrown J)}_\lambda$,
\[\theta_{\phi_\lambda(\mu)} = \theta_{\mu_0}|J|\theta_\mu.\]
\end{rem}

\begin{proof}[Proof of Proposition \ref{second DES from old one}]
We will only verify Definition \ref{DES} \ref{DES part b}, as the rest are shown in the exact same way as in the proof of Proposition \ref{first dist from old one}. By Remark \ref{second bijection between M and Gamma} and Proposition \ref{such and such} \ref{such and such b}. For $\lambda\in\mathcal{E}^{{\lambda_0}}$,
\[\sum_{\mu\in\mathcal{M}^{(\lambda_0,{\mu_0}^\smallfrown J)}_\lambda}\theta_\mu = \theta_{\mu_0}^{-1}|J|^{-1}\sum_{\nu\in \Gamma_\lambda}\theta_\nu = \theta_{\mu_0}^{-1}|J|^{-1}\theta_{\mu_0}|J|\frac{\theta_{\lambda_0\oslash\lambda}}{\theta_{\lambda_0}} = \theta_\lambda.\]
\end{proof}

The proof of the final proposition is based on a remark analogous to Remarks \ref{bijection between M and Gamma} and \ref{second bijection between M and Gamma}, as well as Proposition \ref{such and such} \ref{such and such a}. The process is identical to the previous two propositions; therefore, we will not repeat it.

\begin{prop} \label{third DES from old one}
Let $\lambda_{0} \in  \mathcal{E}$ and $\mu_0\in\mathcal{F}$ such that $\Gamma^{\mu_0}_{\lambda_0}\neq\emptyset$. For every $\lambda \in  \mathcal{E}^{{\lambda_0}}$, denote  
\begin{align*}
 \mathcal{M}^{(\lambda_0,\mu_0)}_{\lambda} = \lbrace \mu \in \mathcal{F}^{{\mu_0}}: {\mu_0}\oslash \mu \in \mathcal{M}_{\lambda_{0} \oslash \lambda } \rbrace
\end{align*}
and, for every $\mu \in  \mathcal{M}^{(\lambda_0,\mu_0)}_{\lambda}$, denote 
\begin{align*}
J^{(\lambda_0,\mu_0)}_{\mu} = J_{{\mu_0}\oslash\mu}.
\end{align*}
Then the collection 
\begin{align*}
\Psi^{\mu_0}_{\lambda_0} = \left( \mathcal{M}_{\lambda}^{(\lambda_0,\mu_0)}, ( \mathcal{H}^{(\lambda_0,\mu_0)}_{\mu})_{\mu \in \mathcal{M}^{(\lambda_0,\mu_0)}_{\lambda}} \right)_{ \lambda \in  \mathcal{E}^{{\lambda_0}}}
\end{align*}
is a distributional embedding scheme of $R_{\mathcal{E}^{{\lambda_0}}}$ in  $R_{\mathcal{F}^{{\mu_0}}}$.
\end{prop}

\section{Distributional representation of $R_{\alpha}^{p,0}$ in $R_{\beta}^{p,0}$} \label{dist repre subsection}
The objective of this section is to establish the following theorem about the map $T_\Psi$ given in Notation \ref{notation TPsi}.

\begin{thm}
 \label{Dist Isom theorem}
Let $\alpha, \beta < \omega_{1}$. For interval-branching subtrees $\mathcal{E}$ of $\mathcal{T}_\alpha$ and $\mathcal{F}$ of $\mathcal{T}_\beta$ and a distributional embedding scheme $\Psi = (\mathcal{M}_{\lambda}, (J_{\mu})_{\mu \in \mathcal{M}_{\lambda}})_{ \lambda \in  \mathcal{E}}$ of $R_{\mathcal{E}}$ in $R_{\mathcal{F}}$, we define the linear map $T_\Psi:\langle (X_\lambda)_{\lambda\in\mathcal{E}}\rangle\to \langle (X_\mu)_{\mu\in\mathcal{F}}\rangle$ given by
\[T_\Psi|_{X_\lambda} = \sum_{\mu\in\mathcal{M}_\lambda}T_\mu J_\mu T_{\lambda}^{-1},\]
 where $\lambda\in\mathcal{E}$. Then, $T_\Psi$ extends to a distributional embedding $T_{\Psi} \colon R_{\mathcal{E}} \to R_{\mathcal{F}}$.
 \end{thm}

The proof is carried out by transfinite induction on $\alpha$, and to aid the process, we introduce the following notation.

\begin{ntt}
For $\alpha < \omega_{1}$, we say that property $P(\alpha)$ is true if the conclusion of Theorem \ref{Dist Isom theorem} holds for the given $\alpha$ and all $\beta<\omega_1$.
\end{ntt}

We have structured the proof by breaking it down into three parts. The first part, Subsection \ref{generic assembly of compressions}, addresses the general theory of distributional embeddings and compressions, and the ways they can be assembled to create new such operators. Its results will enable an increase in the complexity of distributional embeddings in the inductive step of the proof of Theorem \ref{Dist Isom theorem}. In the second part, Subsection \ref{successor step distributional embedding}, we prove the base case and develop the specialized tools required to carry out the successor inductive step. In the third, and lengthiest, part, Subsection \ref{limit step distributional embedding}, we prove a series of operator reductions and use them to carry out the limit inductive step, thus completing the proof of Theorem \ref{Dist Isom theorem}.

\subsection{Assembling abstract $\theta$-compressions}
\label{generic assembly of compressions}
Here, we establish the necessary elementary tools used to assemble compressions and distributional embeddings, creating new operators of the same type. We initially recall that distributional embeddings with independent domains and independent ranges can define an ambient distributional embedding. We then define disjointly supported compressions and demonstrate how they can be utilized to create distributional embeddings.

The following is a standard result proved using the characteristic function of a random variable, which characterizes its distribution. Recall that for a linear operator $T$ we let $\mathrm{R}(T)$ denote its range.

\begin{prop}
\label{independent extension to distributional isomorphism}
Let $(X_n)_{n=1}^\infty$ be a sequence of independent subspaces of $L_p$ and $(T_n:X_n\to L_p)_{n=1}^\infty$, be a sequence distributional embeddings such that the spaces $(\mathrm{R}(T_n))_{n=1}^\infty$ are independent. Then, there exists a distributional embedding
\[T:[(X_n)_{n=1}^\infty]\to L_p\]
such that, for all $n\in \mathcal \mathbb{N}$,
\[T|_{X_n} = T_n.\]
\end{prop}

The notion of the support of a linear subspace of $L_p$ was given in Definition \ref{support of subspace definition}. This leads to the natural notion of compressions with disjointly supported ranges. We will prove that the sum of such compressions with a common domain is still a compression, a result which is visually rather obvious. At the end of this section, we formulate and prove another intuitive statement of how to define a distributional embedding on a disjoint sum $\big(\oplus_{I\in\mathcal{D}^n}X_I\big) = \langle \{T_I[X_I]:I\in\mathcal{D}^n\}\rangle$, for a collection of subspaces $(X_I)_{I\in\mathcal{D}^n}$ of $L_p$.

\begin{dfn}
Let $n\in\N$, $X_1,\ldots,X_n$ be subspaces of $L_p$, and $T_i:X_i\to L_p$, $1\leq i\leq n$, be bounded linear operators. We say that $(T_i)_{i=1}^n$ have disjointly supported ranges if the sets $(\mathrm{supp}(\mathrm{R}(T_i))_{i=1}^n$ are essentially disjoint.
\end{dfn}

\begin{exa}
These are disjointly supported compressions that will be used later.
\begin{enumerate}[label=(\alph*)]

\item For $0\leq a<b\leq 1$, $\mathrm{supp}(\mathrm{R}(T_{[a,b]})) = [a,b]$ and, for $n\in\N$, the $1/2^n$-compressions $(T_I)_{I\in\mathcal{D}^n}$ have disjointly supported ranges.

\item For $\alpha<\omega_1$ and $\mu\in\mathcal{T}_\alpha$, $\mathrm{supp}(\mathrm{R}(T_\mu)) = \supp(X_\mu)$. Therefore, if $\Psi = (\mathcal{M}_\lambda,(\mathcal{H}_\mu)_{\mu\in \mathcal{M}_\lambda})_{\lambda\in\mathcal{E}}$ is a distributional embedding scheme, then, for every $\lambda\in\mathcal{T}_\alpha$, the compressions $(T_\mu)_{\mu\in\mathcal{M}_\lambda}$ have disjointly supported ranges.

\end{enumerate}
\end{exa}

\begin{lem}
\label{lemma disjointly supported ranges put together}
The following identities hold.
\begin{enumerate}[label=(\alph*)]

\item\label{lemma disjointly supported ranges put together a} Let $X$ be a subspace of $L_p$ and $T:X\to L_p$ be a $\theta$-compression. Then, for every bounded measurable function $\phi:\mathbb R\to\mathbb R$ and $f\in X$,
\[\int\phi\big(\big(Tf\big)(t)\big) dt = \theta\int \phi\big(f(t)\big)dt + (1-\theta)\phi(0).\]

\item\label{lemma disjointly supported ranges put together b} Let $\phi:\mathbb R\to\R$ be a bounded measurable function. Let $X_1,\ldots,X_n$ be subspaces of $L_p$, and, for $1\leq i\leq n$, $T_i:X\to L_p$ be a $\theta_i$-compression such that $(T_i)_{i=1}^n$ have disjointly supported ranges. Then, for $f_1\in X_1,\ldots,f_n\in X_n$,
\[\int \phi\Big(\sum_{i=1}^n(T_if_i\big)(t)\Big) dt = \sum_{i=1}^n\theta_i\int \phi\big(f_i(t)\big)d t + \big(1-\sum_{i=1}^n\theta_i\big)\phi(0).\]

\end{enumerate}

\end{lem}

\begin{proof}
Let us promptly prove the first assertion. Assuming that $T$ is a $\theta$-compression,
\begin{align*}
\int\phi\big(\big(Tf\big)(t)\big) dt &= \int \phi(x) \;d\big(\mathrm{dist}(Tf)\big)(x) = \theta \int \phi(x) d\big(\mathrm{dist}(f)\big)(x) + (1-\theta)\int \phi(x)d\delta_0(x)\\
&= \theta\int \phi\big(f(t)\big)d t + (1-\theta)\phi(0).
\end{align*}
In order to deduce the second assertion, observe
\begin{align*}
\int \phi\Big(\sum_{i=1}^n\big(T_if_i\big)(t)\Big) dt &= \sum_{j=1}^n\int_{\mathrm{supp}(\mathrm{R}(T_j))} \phi\Big(\sum_{i=1}^n\big(T_if_i\big)(t)\Big) dt = \sum_{i=1}^n\int_{\mathrm{supp}(\mathrm{R}(T_i))} \phi\big(\big(T_if_i\big)(t)\big) dt\\
&= \sum_{i=1}^n\int \phi\big(\big(T_if_i\big)(t)\big) dt.
\end{align*}
Apply the first statement to conclude the proof.
\end{proof}

\begin{prop}
\label{common domain disjointly supported compressions can be added}
Let $X$ be a subspace of $L_p$, $n\in\N$, and, for $1\leq i\leq n$, $T_i:X\to L_p$ be a $\theta_i$-compression such that $(T_i)_{i=1}^n$ have disjointly supported ranges. Then,
\[T = \sum_{i=1}^nT_i:X\to L_p\]
is a $(\sum_{i=1}^n\theta_i)$-compression.
\end{prop}

\begin{proof}
Let $A$ be a measurable subset of $\mathbb{R}$ and $f\in X$. We apply Lemma \ref{lemma disjointly supported ranges put together} \ref{lemma disjointly supported ranges put together b} for $X_1 = \cdots = X_n = X$, $f_1=\cdots=f_n = f$, and $\phi = \chi_A$, to obtain
\[\Big|\Big[\sum_{i=1}^nT_if\in A\Big]\Big| = \sum_{i=1}^n\theta_i|[f_i\in A]| + (1-\sum_{i=1}^n\theta_i)\chi_A(0).\]
\end{proof}

\begin{ntt}
For $0\leq a<b\leq 1$ and $I = [a,b)$ we define the operator $Q_I:L_p\to L_p$ given by
\[\big(Q_If\big)(t) = f\big(q_I(t)\big) = f\big(a+t(b-a)\big).\]
\end{ntt}

\begin{rem}
\label{Q_I remarks}
For $n\in\mathbb{N}$ and $I,J\in\mathcal{D}^n$, $Q_IT_Jf = \delta_{I,J}f.$
\end{rem}

\begin{prop}
\label{assembly that increases order}
Let $n\in\mathbb{N}$ and, for $I\in\mathcal{D}^n$, let $X_I$ be a subspace of $L_p$ and $S_I:X_I\to L_p$ be a $(1/2^n)$-compression such that $(S_I)_{I\in\mathcal{D}^n}$ have disjointly supported ranges. Then,
\[S = \sum_{I\in\mathcal{D}^n}S_IQ_I:\langle(T_I[X_I])_{I\in\mathcal{D}^n}\rangle\to L_p\]
is a distributional embedding.
\end{prop}

\begin{proof}
Let $f\in\langle(T_I[X_I])_{I\in\mathcal{D}^n}\rangle$, i.e., for $I\in\mathcal{D}^n$, there is $f_I\in X_I$ such that $f = \sum_{I\in\mathcal{D}^n}T_If_I$. By Remark \ref{Q_I remarks}, $Sf = \sum_{I\in\mathcal{D}^n}S_If_I$. Let now $A$ be a measurable subset of $\mathbb{R}$, and apply Lemma \ref{lemma disjointly supported ranges put together} \ref{lemma disjointly supported ranges put together b} twice for $\phi =\chi_A$ to obtain
\begin{align*}
|[Sf \in A]| = \Big|\Big[\sum_{I\in\mathcal{D}^n}S_If_I\Big]\Big| = \sum_{I\in\mathcal{D}^n}\frac{1}{2^n}|[f_I\in A]| =  \Big|\Big[\sum_{I\in\mathcal{D}^n}T_If_I\Big]\Big| = |[f\in A]|,
\end{align*}
i.e., $\mathrm{dist}(Sf) = \mathrm{dist}(f)$.
\end{proof}

\subsection{The base case and the inductive step for successor ordinals}
\label{successor step distributional embedding}
We initially establish the base case, i.e., when $\alpha$ is a positive integer, and then we prepare to prove the successor case of the inductive step in the proof of Theorem \ref{Dist Isom theorem}. If $\alpha = \alpha_0+n$, where $\alpha_0$ is a non-zero limit ordinal number and $n\in\mathbb{N}$, and $\mathcal{E}$ is a interval-branching subtree of $\mathcal{T}_\alpha$, we analyze $T_\Psi$ as a sum of operators that includes sub-distributional embedding schemes with domains over interval-branching subtrees of $\mathcal{T}_{\alpha_0}$, meaning schemes of lesser complexity. We combine this with the inductive hypothesis $P(\alpha_0)$ and Proposition \ref{assembly that increases order} to achieve the desired conclusion.

\begin{prop}
Let $n\in\mathbb{N}$, $\beta<\omega_1$, $\mathcal{E}$ be an interval-branching subtree of $\mathcal{T}_n$, $\mathcal{F}$ be an interval-branching subtree of $\mathcal{T}_\beta$, and $\Psi$ be a distributional embedding scheme from $R_\mathcal{E}$ to $R_\mathcal{F}$. Then $T_\Psi$ is a distributional embedding.
\end{prop}

\begin{proof}
In this case, $\mathcal{T}_n = \{(n)\}$, and thus $\mathcal{E} = \{(n)\}$ and $R_\mathcal{E} = L_p^{n,0}$. Because $T_{(n)} = id$,
\[T_\Psi = \sum_{\mu\in\mathcal{M}_{(n)}}T_\mu J_\mu: R_\mathcal{E}\to R_\mathcal{F},\]
and thus, by Lemma \ref{common domain disjointly supported compressions can be added}, it is a distributional embedding.
\end{proof}

\begin{rem}
The same argument shows that, for any $\alpha<\omega_1$ and trivial subtree $\mathcal{E} = \{\lambda_0\}$ of $\mathcal{T}_\alpha$, any distributional embedding scheme $\Psi$ from $R_\mathcal{E}$ to some space $R_\mathcal{F}$, the operator $T_\Psi$ is a distributional embedding.
\end{rem}

We augment the operator $T_\Psi$ by incorporating the constant functions into its domain. This adjustment simplifies the proof of the successor case in the inductive step, making it more intuitive and eliminating some technicalities that would arise otherwise.

\begin{ntt}
We set the following notation.
\begin{enumerate}[label=(\alph*)]

\item For $\alpha<\omega_1$ and an interval-branching subtree $\mathcal{E}$ of $\mathcal{T}_\alpha$ we denote $\widetilde R_\mathcal{E} = \langle \{\chi_{[0,1]}\}\cup R_\mathcal{E}\rangle$.

\item For $\alpha, \beta < \omega_{1}$, interval-branching subtrees $\mathcal{E}$ of $\mathcal{T}_\alpha$ and $\mathcal{F}$ of $\mathcal{T}_\beta$, and a distributional embedding scheme $\Psi$  of $R_\mathcal{E}$ in $R_\mathcal{F}$, we denote by $\widetilde T_\Psi:\widetilde R_\mathcal{E}\to\widetilde R_\mathcal{F}$ the extension of $T_\Psi$ given by letting $\widetilde T_\Psi\chi_{[0,1]} = \chi_{[0,1]}$.
 
 \end{enumerate}
\end{ntt}

Clearly, $T_\Psi$ is a distributional embedding if and only if $\widetilde T_\Psi$ is.

\begin{ntt}
For this subsection, fix $\alpha, \beta < \omega_{1}$, with $\alpha = \alpha_0+n$, where $\alpha_0$ is a non-zero limit ordinal number and $n\in\mathbb{N}$, interval-branching subtrees $\mathcal{E}$ of $\mathcal{T}_\alpha$ and $\mathcal{F}$ of $\mathcal{T}_\beta$ and a distributional embedding scheme $\Psi = (\mathcal{M}_{\lambda}, (J_{\mu})_{\mu \in \mathcal{M}_{\lambda}})_{ \lambda \in  \mathcal{E}}$  of $R_\mathcal{E}$ in $R_\mathcal{F}$.
\end{ntt}

\begin{prop}
\label{distributional embedding successor proof}
If $P(\alpha_0)$ is true, then $T_\Psi$ is a distributional embedding.
\end{prop}

Before proving this, we will disassemble the operator $\widetilde T_\Psi$. It is worth noting that the ensuing formula is not valid in its current form if we do not use the augmented operators.

\begin{prop}
\label{successor disassembly}
Let $\lambda_0 = (\alpha_0+n)$ denote the unique root of $\mathcal{E}$. Then,
\[\widetilde T_\Psi = \sum_{\mu\in\mathcal{M}_{\lambda_0}}\sum_{I\in\mathcal{D}^{\kappa(\lambda_0)}}\sum_{J\in\mathcal{D}_I^{\kappa(\mu)}}T_\mu T_J \widetilde T_{\Psi^{\mu^\smallfrown J}_{{\lambda_0}^\smallfrown I}}Q_I.\]

\end{prop}

\begin{proof}
We denote by $S$ the map on the right of the above equation. We initially show that, for $f\in \langle \{\chi_{[0,1]}\}\cup X_{\lambda_0}\rangle = L_p^{\kappa(\lambda_0)}$,  $Sf = \widetilde T_\Psi f$. It suffices to verify this for $f = \chi_{I_0}$, $I_0\in\mathcal{D}^{\kappa(\lambda_0)}$. By a standard identity already known to Haar,
\[\chi_{I_0} = \chi_{[h_{\pi(I_0)}=\epsilon(I_0)]}  = 2^{-\kappa(\lambda_0)}\chi_{[0,1]}+\sum_{r=1}^{\kappa(\lambda_0)}2^{-\kappa(\lambda_0)+r-1}\epsilon(\pi^{r-1}(I_0)) h_{\pi^r(I_0)}.\]
This is established by a straightforward induction argument using that, for all $I\in\mathcal{D}\setminus\mathcal{D}^0$, $\chi_I = (1/2)(\chi_{\pi(I)}+\epsilon(I)h_{\pi(I)})$. The same identity holds if we replace $(h_I)_{I\in\mathcal{D}_{\kappa(\lambda_0)-1}}$ with a distributional copy of it. We now compute:
\begin{align*}
S\chi_{I_0} &= \sum_{\mu\in\mathcal{M}_{\lambda_0}}\sum_{I\in\mathcal{D}^{\kappa(\lambda_0)}}\sum_{J\in\mathcal{D}_I^{\kappa(\mu)}}T_{\mu}T_J\widetilde T_{\Psi^{\mu^\smallfrown J}_{{\lambda_0}^\smallfrown I}} Q_I\chi_{I_0} = \sum_{\mu\in\mathcal{M}_{\lambda_0}}\sum_{J\in\mathcal{D}_{I_0}^{\kappa(\mu)}}T_{\mu}T_J\widetilde T_{\Psi^{\mu^\smallfrown J}_{{\lambda_0}^\smallfrown I}}\chi_{[0,1]}\\
&= \sum_{\mu\in\mathcal{M}_{\lambda_0}}\sum_{J\in\mathcal{D}_{I_0}^{\kappa(\mu)}}T_{\mu}T_J\chi_{[0,1]} = \sum_{\mu\in\mathcal{M}_{\lambda_0}}T_{\mu}\sum_{J\in\mathcal{D}_{I_0}^{\kappa(\mu)}}\chi_{J} = \sum_{\mu\in\mathcal{M}_{\lambda_0}} T_\mu\chi_{[J_\mu h_{\pi(I_0)} = \epsilon(I_0)]}\\
& =\sum_{\mu\in\mathcal{M}_{\lambda_0}} T_\mu\Big(2^{-\kappa(\lambda_0)}\chi_{[0,1]}+\sum_{r=1}^{\kappa(\lambda_0)}2^{-\kappa(\lambda_0)+r-1}\epsilon(\pi^{r-1}(I_0)) J_\mu h_{\pi(I_0)}\Big)\\
& =\Big(\sum_{\mu\in\mathcal{M}_{\lambda_0}} T_\mu\Big)2^{-\kappa(\lambda_0)}\chi_{[0,1]}+\sum_{r=1}^{\kappa(\lambda_0)}2^{-\kappa(\lambda_0)+r-1}\epsilon(\pi^{r-1}(I_0)) \Big(\sum_{\mu\in\mathcal{M}_{\lambda_0}}T_\mu J_\mu h_{\pi(I_0)}\Big)\\
&=\widetilde T_\Psi2^{-\kappa(\lambda_0)}\chi_{[0,1]} + \sum_{r=1}^{\kappa(\lambda_0)}2^{-\kappa(\lambda_0)+r-1}\epsilon(\pi^{r-1}(I_0))\widetilde T_\Psi h^{\lambda_0}_{\pi^r(I_0)} = \widetilde T_\Psi(\chi_{I_0}).
\end{align*}

We fix $\lambda\sqsupsetneq\lambda_0$, i.e., there are $I_0\in\mathcal{D}^{\kappa(\lambda_0)}$ and $\xi\in\mathcal{E}^{{\lambda_0}^\smallfrown I_0}$ such that $\lambda = {\lambda_0}^\smallfrown {I_0}^\smallfrown \xi$. We will show $S|_{X_\lambda} = T_\Psi|_{X_\lambda}$, or, equivalently, $ST_\lambda|_{L_p^{\kappa(\lambda),0}} = T_\Psi T_\lambda|_{L_p^{\kappa(\lambda),0}}$. Fix $f\in L_p^{\kappa(\lambda),0}$. By Remark \ref{remark about partitioning a successor} \ref{remark about partitioning a successor b}, the sets
\[\Big\{\Gamma_\lambda^{\mu^\smallfrown P}:\mu\in\mathcal{D}^{\kappa(\lambda_0)}\text{ and }P\in\mathcal{D}_{I_0}^\mu\Big\}\]
form a partition of $\mathcal{M}_\lambda$.
We thus calculate
\begin{align*}
ST_\lambda f & = \sum_{I\in\mathcal{D}^{\kappa(\lambda_0)}}\Big(\sum_{\mu\in\mathcal{M}_{\lambda_0}} \sum_{J\in\mathcal{D}_I^{\kappa(\mu)}}T_\mu T_J \widetilde T_{\Psi^{\mu^\smallfrown J}_{{\lambda_0}^\smallfrown I}}\Big)Q_I (T_{I_0}T_\xi f)= \Big(\sum_{\mu\in\mathcal{M}_{\lambda_0}} \sum_{J\in\mathcal{D}_{I_0}^{\kappa(\mu)}}T_\mu T_J \widetilde T_{\Psi^{\mu^\smallfrown J}_{{\lambda_0}^\smallfrown {I_0}}}\Big) T_\xi f\\
&= \Big(\sum_{\mu\in\mathcal{M}_{\lambda_0}} \sum_{J\in\mathcal{D}_{I_0}^{\kappa(\mu)}}T_\mu T_J\Big)\sum_{\nu\in\mathcal{M}_\xi^{({\lambda_0}^\smallfrown {I_0},\mu^\smallfrown J)}}T_\nu J^{({\lambda_0}^\smallfrown {I_0},\mu^\smallfrown J)}_\nu f\text{ (definition of $T_{\Psi^{\mu^\smallfrown J}_{{\lambda_0}^\smallfrown {I_0}}}$)}\\
&= \sum_{\mu\in\mathcal{M}_{\lambda_0}} \sum_{J\in\mathcal{D}_{I_0}^{\kappa(\mu)}}\sum_{\nu\in\mathcal{M}_\xi^{({\lambda_0}^\smallfrown {I_0},\mu^\smallfrown J)}}T_{\mu^\smallfrown J^\smallfrown\nu} J_{\mu^\smallfrown J^\smallfrown\nu} f\text{ (definition of $J_\nu^{({\lambda_0}^\smallfrown {I_0},\mu^\smallfrown J)}$)} \\
& = \sum_{\mu\in\mathcal{M}_{\lambda_0}} \sum_{J\in\mathcal{D}_{I_0}^{\kappa(\mu)}}\sum_{\nu\in\Gamma_\lambda^{\mu^\smallfrown J}}T_\nu J_\nu f =\sum_{\nu\in\mathcal{M}_\lambda} T_\nu J_\nu f= T_\Psi T_\lambda f.
\end{align*}
\end{proof}

\begin{lem}
\label{successors subspaces reconstruct space}
We have
\[\widetilde R_\mathcal{E} = \langle\{T_I\big(\widetilde R_{\mathcal{E}^{{\lambda_0}^\smallfrown I}}\big):I\in\mathcal{D}^n\}\rangle.\]
\end{lem}

\begin{proof}
Noting, $\langle \{\chi_{[0,1]}\}\cup X_{\lambda_0}\rangle = L_p^{\kappa(\lambda_0)}$, we have
\begin{align*}
\widetilde R_\mathcal{E} &= \big[ L_p^{\kappa(\lambda_0)}\cup \bigcup_{I\in\mathcal{D}^{\kappa(\lambda_0)}}\bigcup_{\mu\in\mathcal{E}^{{\lambda_0}^\smallfrown I}} X_{{\lambda_0}^\smallfrown I^\smallfrown \mu}\big] = \big[ (T_I\chi_{[0,1]})_{I\in\mathcal{D}^{\kappa(\lambda_0)}} \cup \bigcup_{I\in\mathcal{D}^{\kappa(\lambda_0)}}T_I\big([(X_\mu)_{\mu\in \mathcal{E}^{{\lambda_0}^\smallfrown I}}]\big)\big]\\
&= \big[\bigcup_{I\in\mathcal{D}^{\kappa(\lambda_0)}}T_I\big([\{\chi_{[0,1]}\} \cup R_{\mathcal{E}^{{\lambda_0}^\smallfrown I}}]\big)\big] = \big\langle \bigcup_{I\in\mathcal{D}^{\kappa(\lambda_0)}}T_I\big(\widetilde R_{\mathcal{E}^{{\lambda_0}^\smallfrown I}}\big)\big\rangle.
\end{align*}
\end{proof}

\begin{proof}[Proof of Proposition \ref{distributional embedding successor proof}]
We will prove that $\widetilde T_\Psi$ is a distributional embedding, which is equivalent to the conclusion. Using Proposition \ref{successor disassembly}, write
\[\widetilde T_\Psi =\sum_{\mu\in\mathcal{M}_{\lambda_0}} T_\mu\Big[\sum_{I\in\mathcal{D}^{\kappa(\lambda_0)}}\underbrace{\Big(\sum_{J\in\mathcal{D}_I^{\kappa(\mu)}}T_J \widetilde T_{\Psi^{\mu^\smallfrown J}_{{\lambda_0}^\smallfrown I}}\Big)}_{=:S_I^\mu}Q_I.\Big]\]
For $\mu\in\mathcal{M}_{\lambda_0}$, $I\in\mathcal{D}^{\kappa(\lambda_0)}$, and $J\in\mathcal{D}^{\kappa(\mu_0)}_I$, because $P(\alpha_0)$ is true, $\widetilde T_{\Psi^{\mu^\smallfrown J}_{{\lambda_0}^\smallfrown I}}$ is a distributional embedding with domain $\widetilde R_{\mathcal{E}^{{\lambda_0}^\smallfrown I}}$. It follows that $(T_J\widetilde T_{\Psi^{\mu^\smallfrown J}_{{\lambda_0}^\smallfrown I}})_{J\in\mathcal{D}_I^{\kappa(\mu)}}$ are $2^{-\kappa(\mu)}$-compressions with disjointly supported ranges. By Proposition \ref{common domain disjointly supported compressions can be added}, for $\mu\in\mathcal{M}_{\lambda_0}$ and $I\in\mathcal{D}^{\kappa(\lambda_0)}$,
\[S_I^\mu = \sum_{J\in\mathcal{D}_I^{\kappa(\mu)}}T_J \widetilde T_{\Psi^{\mu^\smallfrown J}_{{\lambda_0}^\smallfrown I}}\]
is a $|I|$-compression with domain $\widetilde R_{\mathcal{E}^{{\lambda_0}^\smallfrown I}}$ and range supported in $\cup\mathcal{D}^\mu_I$. By Proposition \ref{assembly that increases order},
\[S_\mu = \sum_{I\in\mathcal{D}^{\kappa(\lambda_0)}}S_I^\mu Q_I\]
is a distributional embedding with domain $\langle\{T_I[\widetilde R_{\mathcal{E}^{{\lambda_0}^\smallfrown I}}]:I\in\mathcal{D}^n\}\rangle = \widetilde R_\mathcal{E}$, by Lemma \ref{successors subspaces reconstruct space}. By Proposition \ref{common domain disjointly supported compressions can be added}, $\widetilde T_\Psi =\sum_{\mu\in\mathcal{M}_{\lambda_0}} T_\mu S_\mu$ is a distributional embedding.
\end{proof}

\subsection{The inductive step for limit ordinals}
\label{limit step distributional embedding} The limit ordinal case of the inductive step is more involved than the successor case. We first establish some notation.

\begin{ntt}
For $\alpha<\omega_1$ and an interval-branching subtree $\mathcal{E}$ of $\mathcal{T}_\alpha$, we denote
\begin{enumerate}[label=(\alph*)]

\item $\mathpzc{Root}(\mathcal{E}) = \{\lambda\text{ is a root of }\mathcal{E}\}$ and,

\item for $\lambda\in\mathpzc{Root}(\mathcal{E})$, $Y^\mathcal{E}_\lambda = \big[(X_{\lambda'})_{\lambda'\in\mathcal{E},\lambda'\sqsupseteq\lambda}\big]$.

\end{enumerate}
\end{ntt}

\begin{rem}
\label{remark about splitting the roots in E along the roots in F}
By definition, we have $R_\mathcal{E} = [(Y^\mathcal{E}_\lambda)_{\lambda\in\mathpzc{Root}(\mathcal{E})}]$. Furthermore, for every $\lambda\in\mathpzc{Root}(\mathcal{E})$, $Y^\mathcal{E}_\lambda\subset Y_\lambda$, and therefore by Proposition \ref{independent sum of successor ordinals}, the spaces $(Y^\mathcal{E}_\lambda)_{\lambda\in\mathpzc{Root}(\mathcal{E})}$ are independent.
\end{rem}

\begin{rem}
\label{Relation to Elambda}
For $\lambda\in\mathpzc{Root}(\mathcal{E})$, there exists a relation between the spaces $R_{\mathcal{E}^\lambda}$ and $Y^\mathcal{E}_\lambda$:
\[T_\lambda[R_{\mathcal{E}^\lambda}] = Y_\lambda^\mathcal{E}.\]
Therefore,
\[R_\mathcal{E} \equiv^\mathrm{dist}\Big(\sum_{\lambda\in\mathpzc{Root}(\mathcal{E})}R_{\mathcal{E}^\lambda}\Big)_{\mathrm{Ind},p}.\]

Indeed, for $\mu\in \mathcal{E}^\lambda$,
\[T_\lambda[X_\mu] = T_\lambda T_\mu[L_p^{\kappa(\mu),0}] =  T_\lambda T_\mu[L_p^{\kappa(\lambda\oslash\mu),0}] =  T_{\lambda\oslash\mu}[L_p^{\kappa(\lambda\oslash\mu),0}] = X_{\lambda\oslash\mu}.\]
Because $\mathcal{E}^\lambda = \{\mu\in\mathcal{T}_{b(\lambda)+\kappa(\lambda)}:\lambda\oslash\mu \in\mathcal{E}\}$, the map $\mu\mapsto \lambda\oslash \mu$ defines a bijection from $\mathcal{E}^\lambda$ to $\{\lambda'\in\mathcal{E}:\lambda\sqsubseteq \lambda'\}$. We conclude
\[T_\lambda[R_{\mathcal{E}^\lambda}] = [(T_\lambda[X_\mu])_{\mu\in\mathcal{E}^\lambda}] = [(X_{\lambda\oslash\mu})_{\mu\in\mathcal{E}^\lambda}] =  [(X_{\lambda'})_{\lambda'\in\mathcal{E}\text{ and }\lambda\sqsubseteq\lambda'}] = Y_\lambda^\mathcal{E}.\]
\end{rem}

\begin{ntt}
For this subsection, fix $\alpha, \beta < \omega_{1}$, where $\alpha$ is a limit ordinal, interval-branching subtrees $\mathcal{E}$ of $\mathcal{T}_\alpha$ and $\mathcal{F}$ of $\mathcal{T}_\beta$ and a distributional embedding scheme $\Psi = (\mathcal{M}_{\lambda}, (J_{\mu})_{\mu \in \mathcal{M}_{\lambda}})_{ \lambda \in  \mathcal{E}}$  of $R_\mathcal{E}$ in $R_\mathcal{F}$.
\end{ntt}

We will prove the following.

\begin{prop}
\label{successor step long proof}
If, for all successor ordinals $\gamma < \alpha$, $P(\gamma)$ is true then $T_\Psi$ is a distributional embedding.
\end{prop}

We begin by outlining a sequence of disassemblies of operators associated with sub-distributional embedding schemes derived from $\Psi$.

\begin{lem}
\label{operator reduction to roots}
Let $\lambda_0\in\mathpzc{Root}(\mathcal{E})$ and let $\mu_0$ be the unique root of $\mathcal{F}$ such that $\Gamma_{\lambda_0}^{{\mu_0}} = \mathcal{M}_{\lambda_0}$ (which exists by Lemma \ref{roots associated to M lambda} \ref{roots associated to M lambda a}). Then
\[T_\Psi|_{Y^\mathcal{E}_{\lambda_0}} = T_{\mu_0}T_{\Psi^{{\mu_0}}_{\lambda_0}}T^{-1}_{\lambda_0}|_{Y^\mathcal{E}_{\lambda_0}}.\]
\end{lem}

\begin{proof}
The tree structure of $(\mathcal{M}_\lambda)_{\lambda\in\mathcal{E}}$ (see Remark \ref{lemma for existnce of eleme in th epredic} \ref{lemma for existnce of eleme in th epredic i}) yields that, for $\lambda \sqsupsetneq \lambda_0$, $\Gamma_\lambda^{\mu_0} = \mathcal{M}_\lambda$. In particular,  if $\mu\in\mathcal{E}^{\lambda_0}$ is such that $\lambda = \lambda_0\oslash\mu$ then, by the definition of $\mathcal{M}_{\mu}^{(\lambda_0,\mu_0)}$,
\[\mathcal{M}_\lambda = \Big\{\mu_0\oslash\nu:\nu\in\mathcal{M}_{\mu}^{(\lambda_0,\mu_0)}\Big\}.\]
Let us fix $f\in X_\lambda$, i.e., there is $g\in L_p^{\kappa(\lambda),0}$ such that $f = T_\lambda g = T_{\lambda_0\oslash \mu} g = T_{\lambda_0}T_{\mu} g$. Then
\begin{align*}
T_{\mu_0}T_{\Psi^{{\mu_0}}_{\lambda_0}}T^{-1}_{\lambda_0} f = T_{\mu_0}T_{\Psi^{{\mu_0}}_{\lambda_0}}T_{\mu} g =& \sum_{\nu\in\mathcal{M}^{(\lambda_0,\mu_0)}_\mu} T_{\mu_0}T_\nu J_\nu^{(\lambda_0,\mu_0)} T_\mu^{-1}T_\mu g\\
&= \sum_{\nu\in\mathcal{M}^{(\lambda_0,\mu_0)}_\mu} T_{\mu_0\oslash \nu} J_{\mu_0\oslash \nu} g = \sum_{\nu\in\mathcal{M}_\lambda} T_{\nu} J_\nu g = T_\Psi f.
\end{align*}
\end{proof}

\begin{lem}
\label{operator reduction successor to limit}
Let $\lambda_0 \in  \mathpzc{Root}(\mathcal{E})$ and $\mu_0\in\mathcal{F}\setminus\mathcal{M}_{\lambda_0}$ such that $\Gamma^{{\mu_0}}_{\lambda_0}\neq\emptyset$. Then
\[T_{\Psi_{\lambda_0}^{\mu_0}} = \sum_{J\in\mathcal{D}^{\kappa(\mu_0)}}T_JT_{\Psi^{{\mu_0}^\smallfrown J}_{\lambda_0}}.\]
\end{lem}

\begin{proof}
Because $\Gamma^{{\mu_0}}_{\lambda_0}\neq\emptyset$ and $\mu_0\notin\mathcal{M}_{\lambda_0}$, we have, for all $\lambda\sqsupseteq \lambda_0$, the sets
\begin{equation}
\label{operator reduction successor to limit eq 1}
\big\{\Gamma_{\lambda}^{{\mu_0}^\smallfrown J}: J\in \mathcal{D}^{\kappa(\mu_0)}\big\}
\end{equation}
form a partition of $\Gamma^{{\mu_0}}_{\lambda}$. We fix $\lambda\in\mathcal{E}^{\lambda_0}$ and $f\in X_\lambda$, i.e., $f = T_\lambda g$, for some $g\in L_p^{\kappa(\lambda),0}$. By definition,
\[T_{\Psi_{\lambda_0}^{{\mu_0}}} f= \sum_{\nu\in\mathcal{M}_\lambda^{(\lambda_0,\mu_0)}}T_\nu J_\nu^{(\lambda_0,\mu_0)}T_\lambda^{-1}Tg =  \sum_{\nu\in\mathcal{M}_\lambda^{(\lambda_0,\mu_0)}}T_\nu J_{\mu_0\oslash \nu}g\]
and, for $J\in\mathcal{D}^{\kappa(\mu_0)}$
\[T_{\Psi_{\lambda_0}^{{\mu_0}^\smallfrown J}}f = \sum_{\nu\in\mathcal{M}_\lambda^{(\lambda_0,{\mu_0}^\smallfrown J)}}T_\nu J_\nu^{(\lambda_0,{\mu_0}^\smallfrown J)} T_\lambda^{-1}T_\lambda g = \sum_{\nu\in\mathcal{M}_\lambda^{(\lambda_0,{\mu_0}^\smallfrown J)}}T_\nu J_{{\mu_0}^\smallfrown J^\smallfrown \nu} g.\]
We associate the index sets in the rightmost above sums as follows:
\begin{equation}
\label{operator reduction successor to limit eq 2}
\mathcal{M}_\lambda^{(\lambda_0,\mu_0)} = \bigcup_{J\in\mathcal{D}^{\kappa(\mu_0)}}\Big\{(b(\mu_0)+\kappa(\mu_0))^\smallfrown J^\smallfrown \nu:\nu\in\mathcal{M}_\lambda^{(\lambda_0,{\mu_0}^\smallfrown J)}\Big\}.
\end{equation}
Indeed, by Remark \ref{second bijection between M and Gamma}, for $J\in \mathcal{D}^{\kappa(\mu_0)}$, the map $\phi^J_\lambda:\mathcal{M}_\lambda^{(\lambda_0,{\mu_0}^\smallfrown J)}\to \Gamma^{{\mu_0}^\smallfrown J}_{\lambda_0\oslash\lambda}$ given by
\[\phi_\lambda(\nu) = {\mu_0}^\smallfrown J^\smallfrown \nu\]
is a bijection. Similarly, the map $\psi_\lambda:\mathcal{M}_\lambda^{(\lambda_0,{\mu_0})}\to \Gamma^{{\mu_0}}_{\lambda_0\oslash\lambda}$ given by
\[\psi_\lambda(\nu) = {\mu_0}\oslash \nu\]
is a bijection. Therefore, by \eqref{operator reduction successor to limit eq 1}, for $J\in\mathcal{D}^{\kappa(\mu_0)}$ the map
\[\phi_\lambda^{-1}\psi_\lambda^J(\nu) = (b(\mu_0)+\kappa(\mu_0))^\smallfrown J^\smallfrown \nu\]
defines an injection from $\mathcal{M}_\lambda^{(\lambda_0,{\mu_0}^\smallfrown J)}$ to $\mathcal{M}_\lambda^{(\lambda_0,\mu_0)}$, and by \eqref{operator reduction successor to limit eq 1}, the images of $(\phi_\lambda^{-1}\psi_\lambda^J)_{J\in\mathcal{D}^{\kappa(\mu_0)}}$ partition $\mathcal{M}_\lambda^{(\lambda_0,\mu_0)}$, which establishes \eqref{operator reduction successor to limit eq 2}. We conclude the proof as follows:
\begin{align*}
T_{\Psi^{\mu_0}_{\lambda_0}}f &= \sum_{J\in\mathcal{D}^{\kappa(\mu_0)}}\sum_{\nu\in\mathcal{M}_\lambda^{(\lambda_0,{\mu_0}^\smallfrown J)}} T_{(b(\mu_0)+\kappa(\mu_0))^\smallfrown J^\smallfrown \nu} J_{{\mu_0}^\smallfrown J^\smallfrown \nu}g\\
&=\sum_{J\in\mathcal{D}^{\kappa(\mu_0)}}\sum_{\nu\in\mathcal{M}_\lambda^{(\lambda_0,{\mu_0}^\smallfrown J)}} T_JT_\nu J_{{\mu_0}^\smallfrown J^\smallfrown \nu}g = \sum_{J\in\mathcal{D}^{\kappa(\mu_0)}}T_JT_{\Psi^{{\mu_0}^\smallfrown J}_{\lambda_0}}f.
\end{align*}
\end{proof}

\begin{lem}
\label{operator reduction going down the path}
Let $\lambda_0 \in  \mathpzc{Root}(\mathcal{E})$, $\mu_0\in\mathcal{F}$, and $J\in\mathcal{D}^{\kappa(\mu_0)}$ such that $\Gamma^{{\mu_0}^\smallfrown J}_{\lambda_0}\neq\emptyset$. Let $\mu$ be the unique root of $\mathcal{F}^{{\mu_0}^\smallfrown J}$ such that $\Gamma_{\lambda_0}^{{\mu_0}^\smallfrown J} = \Gamma_{\lambda_0}^{{\mu_0}^\smallfrown J^\smallfrown \mu}$ (which exists by Lemma \ref{roots associated to M lambda} \ref{roots associated to M lambda b}). Then
\[T_{\Psi^{{\mu_0}^\smallfrown J}_{\lambda_0}} = T_{\mu}T_{\Psi_{\lambda_0}^{{\mu_0}^\smallfrown J^\smallfrown\mu}}.\]
\end{lem}

\begin{proof}
The ensuing argument is along the lines of the proof of Lemma \ref{operator reduction successor to limit}, but we include it for completeness. Similarly to Remark \ref{remark about partitioning a successor} \ref{remark about partitioning a successor c}, for $\lambda \sqsupsetneq \lambda_0$,
\begin{equation}
\label{operator reduction going down the path eq1}
\Gamma_{\lambda}^{{\mu_0}^\smallfrown J} = \bigcup_{\mu\in\Gamma_{\lambda_0}^{{\mu_0}^\smallfrown J}}\Gamma_{\lambda}^{\mu} = \bigcup_{\mu\in\Gamma_{\lambda_0}^{{\mu_0}^\smallfrown J^\smallfrown\mu}}\Gamma_{\lambda}^{\mu} = \Gamma_{\lambda}^{{\mu_0}^\smallfrown J^\smallfrown \mu}.
\end{equation}
We fix $\lambda\in\mathcal{E}^{\lambda_0}$ and $f\in X_\lambda$, i.e., $f = T_\lambda g$, for some $g\in L_p^{\kappa(\lambda),0}$. By definition,
\[T_{\Psi_{\lambda_0}^{{\mu_0}^\smallfrown J}}f = \sum_{\nu\in\mathcal{M}_\lambda^{(\lambda_0,{\mu_0}^\smallfrown J)}}T_\nu J_\nu^{(\lambda_0,{\mu_0}^\smallfrown J)} T_\lambda^{-1}T_\lambda g = \sum_{\nu\in\mathcal{M}_\lambda^{(\lambda_0,{\mu_0}^\smallfrown J)}}T_\nu J_{{\mu_0}^\smallfrown J^\smallfrown \nu} g.\]
and
\[T_{\Psi_{\lambda_0}^{{\mu_0}^\smallfrown J^\smallfrown\mu}}f = \sum_{\nu\in\mathcal{M}_\lambda^{(\lambda_0,{\mu_0}^\smallfrown J^\smallfrown\mu)}}T_\nu J_\nu^{(\lambda_0,{\mu_0}^\smallfrown J^\smallfrown \mu)} T_\lambda^{-1}T_\lambda f =  \sum_{\nu\in\mathcal{M}_\lambda^{(\lambda_0,{\mu_0}^\smallfrown J^\smallfrown\mu)}}T_\nu J_{{\mu_0}^\smallfrown J^\smallfrown \mu\oslash\nu}f\]
We associate the index sets of the rightmost above sums as follows:
\begin{equation}
\label{operator reduction going down the path eq2}
\mathcal{M}_\lambda^{(\lambda_0,{\mu_0}^\smallfrown J)} = \Big\{\mu\oslash\nu: \nu\in\mathcal{M}_\lambda^{(\lambda_0,{\mu_0}^\smallfrown J^\smallfrown\mu)}\Big\}.
\end{equation}
Indeed, by Remark \ref{second bijection between M and Gamma}, the map $\phi_\lambda:\mathcal{M}_\lambda^{(\lambda_0,{\mu_0}^\smallfrown J)}\to \Gamma^{{\mu_0}^\smallfrown J}_{\lambda_0\oslash\lambda}$ given by
\[\phi_\lambda(\nu) = {\mu_0}^\smallfrown J^\smallfrown \nu\]
is a bijection. Similarly, the map $\psi_\lambda:\mathcal{M}_\lambda^{(\lambda_0,{\mu_0}^\smallfrown J^\smallfrown \mu)}\to \Gamma^{{\mu_0}^\smallfrown J^\smallfrown \mu}_{\lambda_0\oslash\lambda}$ given by
\[\psi_\lambda(\nu) = {\mu_0}^\smallfrown J^\smallfrown\mu\oslash \nu\]
is a bijection. Given $\Gamma^{{\mu_0}^\smallfrown J}_{\lambda_0\oslash\lambda} = \Gamma^{{\mu_0}^\smallfrown J^\smallfrown \mu}_{\lambda_0\oslash\lambda}$, the map $\phi_\lambda^{-1}\psi_\lambda(\nu) = \mu\oslash\nu$ is a bijection witnessing \eqref{operator reduction going down the path eq2}. We conclude as follows:
\begin{align*}
T_{\Psi_{\lambda_0}^{{\mu_0}^\smallfrown J}}f &= \sum_{\nu\in\mathcal{M}_\lambda^{(\lambda_0,{\mu_0}^\smallfrown J)}}T_\nu J_{{\mu_0}^\smallfrown J^\smallfrown \nu} g= \sum_{\nu\in\mathcal{M}_\lambda^{(\lambda_0,{\mu_0}^\smallfrown J^\smallfrown\mu)}} T_{\mu\oslash \nu} J_{{\mu_0}^\smallfrown J^\smallfrown\mu\oslash \nu} g = T_\mu T_{\Psi_{\lambda_0}^{{\mu_0}^\smallfrown J^\smallfrown\mu}}f .
\end{align*}
\end{proof}

We now perform the inductive step. To reach the intended conclusion, we carry out an induction on a specific sub-tree $\mathcal{S}$ of $\mathcal{F}$, associated with $\Psi$. As we progress from the leaves to the roots of $\mathcal{S}$, the intention is to show that for each $\mu\in\mathcal{S}$, a certain operator $S_\mu$ associated with $\mu$ and $\Psi$ sufficiently preserves independence.

\begin{proof}[Proof of Proposition \ref{successor step long proof}]
For every $\mu\in\mathcal{F}$ (whether a root or otherwise) denote
\[\mathpzc{Root}(\mathcal{E},\mu) = \{\lambda\in\mathpzc{Root}(\mathcal{E}):\Gamma_{\lambda}^\mu\neq\emptyset\}\]
and $Z_\mu = [(Y^\mathcal{E}_{\lambda})_{\lambda\in\mathpzc{Root}(\mathcal{E},\mu)}]$. Note that, by Lemma \ref{roots associated to M lambda} \ref{roots associated to M lambda b}, the sets $\big(\mathpzc{Root}(\mathcal{E},\mu)\big)_{\mu\in\mathpzc{Root}(\mathcal{F})}$ form a partition of the roots of $\mathpzc{Root}(\mathcal{E})$. By Remark \ref{remark about splitting the roots in E along the roots in F}, $R_{\mathcal{E}} = [(Z_\mu)_{\mu\in\mathpzc{Root}(\mathcal{F})}]$, and the spaces $(Z_\mu)_{\mu\in\mathpzc{Root}(\mathcal{F})}$ are independent.

We formulate the following claim.
\begin{clm}
For all $\mu\in\mathcal{F}$ for which $\mathpzc{Root}(\mathcal{E},\mu)\neq\emptyset$, there exists a distributional embedding
\[S_\mu: Z_\mu \to R_{\mathcal{F}^\mu}\]
such that, for all $\lambda\in\mathpzc{Root}(\mathcal{E},\mu)$,
\[S_\mu|_{Y^\mathcal{E}_{\lambda}} = T_{\Psi_\lambda^{\mu}}T^{-1}_\lambda|_{Y^\mathcal{E}_{\lambda}}.\]
\end{clm}
Assuming the above is true, we will first conclude the proof by applying the claim. For $\mu\in\mathpzc{Root}(\mathcal{F})$, because $T_\mu$ is a distributional embedding, we have that $T_\mu S_\mu:Z_\mu\to R_\beta^{p,0}$ is a distributional embedding with range inside $Y_\mu$, and thus, by Proposition \ref{independent sum of successor ordinals}, $(T_\mu S_\mu)_{\mu\in\mathpzc{Root}(\mathcal{F})}$ have independent ranges. By Proposition \ref{independent extension to distributional isomorphism}, the operator $S:R_\mathcal{E} = [(Z_\mu)_{\mu\in\mathpzc{Root}(\mathcal{F})}]\to R_\beta^{p,0}$ given by $S|_{Z_\mu} = T_\mu S_\mu$, $\mu\in\mathpzc{Root}(\mathcal{F})$, is a distributional embedding. We will use Lemma \ref{operator reduction to roots} to show that $S = T_\Psi$. Indeed, for $\lambda\in\mathpzc{Root}(\mathcal{E})$ and $\mu\in\mathpzc{Root}(\mathcal{F})$ such that $\lambda\in\mathpzc{Root}(\mathcal{E},\mu)$,
\[S|_{Y^\mathcal{E}_{\lambda}} = T_\mu S_\mu|_{Y^\mathcal{E}_{\lambda}} = T_\mu T_{\Psi_\lambda^{\mu}}T^{-1}_\lambda|_{Y^\mathcal{E}_{\lambda}} = T_\Psi|_{Y^\mathcal{E}_{\lambda}}.\]

We proceed to prove the claim by induction on the sub-tree $\mathcal{S} = \{\mu\in\mathcal{F}:\mathpzc{Root}(\mathcal{E},\mu)\neq\emptyset\}$ of $\mathcal{F}$ with maximal nodes the members of $\cup_{\lambda\in\mathpzc{Root}(\mathcal{E})}\mathcal{M}_\lambda$. The inductive step will essentially be a refinement of the preceding argument.

In the basis step, let $\mu\in\mathcal S$ be a terminal node. Then $\mathpzc{Root}(\mathcal{E},\mu) = \{\lambda\}$, for some $\lambda\in\mathpzc{Root}(\mathcal{E})$. The claim can be rephrased as follows: $T_{\Psi_\lambda^\mu}T_\lambda^{-1}:Y^\mathcal{E}_{\lambda}\to R_{\mathcal{F}^\mu}$ is a distributional embedding, which we will now verify. Because $P(b(\lambda)+\kappa(\lambda))$ is true, $T_{\Psi^\mu_\lambda}:R_{\mathcal{E}^\lambda}\to R_{\mathcal{F}^\mu}$ is a distributional embedding. By Remark \ref{Relation to Elambda}, $T _\lambda^{-1}:Y^\mathcal{E}_{\lambda}\to R_{\mathcal{E}^\lambda}$ is a distributional isomorphism. We conclude that $T_{\Psi_\lambda^\mu}T_\lambda^{-1}:Y^\mathcal{E}_{\lambda}\to R_{\mathcal{F}^\mu}$ is a distributional embedding.

Let $\mu$ be non-maximal with the property $\mathpzc{Root}(\mathcal{E},\mu)\neq\emptyset$ and assume that all successors of $\mu$ satisfy the claim. We will first prove that, for $J\in\mathcal{D}^{\kappa(\mu)}$,
\begin{equation}
\label{roots united never divided}
\mathpzc{Root}(\mathcal{E},\mu) = \bigcup_{\nu\in\mathpzc{Root}(\mathcal{F}^{\mu^\smallfrown J})}\mathpzc{Root}(\mathcal{E},\mu^\smallfrown J^\smallfrown \nu).
\end{equation}
The ``$\supset$'' inclusion is trivial. Let $\lambda\in \mathpzc{Root}(\mathcal{E},\mu)$. Because $\mu$ is non-maximal in $\mathcal{S}$, $\mu\notin \mathcal{M}_\lambda$, and thus, there is $P\in\mathcal{D}^{\kappa(\mu)}$ such that $\Gamma_\lambda^{\mu^\smallfrown P}\neq\emptyset$. By Proposition \ref{such and such} \ref{such and such b}, $\sum_{\xi\in \Gamma_\lambda^{\mu^\smallfrown J}} \theta_\xi= \sum_{\xi\in \Gamma_\lambda^{\mu^\smallfrown P}}\theta_\xi$, and thus $\Gamma_\lambda^{\mu^\smallfrown J}\neq\emptyset$. We deduce that there exists $\nu\in\mathpzc{Root}(\mathcal{F}^{\mu^\smallfrown J})$ such that $\Gamma_\lambda^{\mu^\smallfrown J^\smallfrown\nu}\neq\emptyset$, i.e., $\lambda\in\mathpzc{Root}(\mathcal{E},\mu^\smallfrown J^\smallfrown \nu)$.

Fix momentarily $J\in\mathcal{D}^{\kappa(\mu)}$. By the inductive hypothesis, for every $\nu\in\mathpzc{Root}(\mathcal{F}^{\mu^\smallfrown J})$, there exists a distributional embedding
\[S_{\mu^\smallfrown J^\smallfrown \nu}:Z_{\mu^\smallfrown J^\smallfrown \nu}\to R_{\mathcal{F}^{\mu^\smallfrown J^\smallfrown \nu}}\]
such that, for all $\lambda\in \mathpzc{Root}(\mathcal{E},\mu^\smallfrown J^\smallfrown \nu)$,
\[S_{\mu^\smallfrown J^\smallfrown \nu}|_{Y^\mathcal{E}_\lambda} = T_{\Psi_\lambda^{\mu^\smallfrown J^\smallfrown \nu}}T_\lambda^{-1}|_{Y^\mathcal{E}_\lambda}.\]
Therefore, for $\nu\in\mathpzc{Root}(\mathcal{F}^{\mu^\smallfrown J})$, the operator $T_\nu S_{\mu^\smallfrown J^\smallfrown \nu}:Z_{\mu^\smallfrown J^\smallfrown \nu}\to R_{b(\mu)}^{p,0}$ is a distributional embedding with image in $Y_\nu$. By Proposition \ref{independent extension to distributional isomorphism}, there exists a distributional embedding
\[S_J:[(Z_{\mu^\smallfrown J^\smallfrown \nu})_{\nu\in\mathpzc{Root}(\mathcal{F}^{\mu^\smallfrown J})}] = [\big((Y^\mathcal{E}_\lambda)_{\lambda\in\mathpzc{Root}(\mathcal{E},\mu^\smallfrown J^\smallfrown \nu)}\big)_{\nu\in\mathpzc{Root}(\mathcal{F}^{\mu^\smallfrown J})}] \stackrel{\eqref{roots united never divided}}{=} Z_\mu
\to R_{b(\mu)}^{p,0}\]
such that, for all $\nu\in\mathpzc{Root}(\mathcal{F}^{\mu^\smallfrown J})$,
\[S_J|_{Z_{\mu^\smallfrown J^\smallfrown \nu}} = T_\nu S_{\mu^\smallfrown J^\smallfrown \nu}.\]
By using Proposition \ref{common domain disjointly supported compressions can be added}, we deduce that the operator $S = \sum_{J\in\mathcal{D}^{\kappa(\mu)}}T_JS_J:Z_\mu
\to R_{b(\mu)}^{p,0}$ is a distributional embedding. To conclude the proof, it remains to be verified that, for $\lambda\in\mathpzc{Root}(\mathcal{E},\mu)$, $S|_{Y^\mathcal{E}_\lambda} = T_{\Psi_\lambda^\mu}T_\lambda^{-1}|_{Y^\mathcal{E}_\lambda}$. For $J\in\mathcal{D}^{\kappa(\mu)}$, let $\nu_J$ be the unique root of $\mathcal{F}^{\mu^\smallfrown J}$ such that $\lambda\in\mathpzc{Root}(\mathcal{E},\mu^\smallfrown J^\smallfrown \nu_J)$. Then,
\begin{align*}
S|_{Y^\mathcal{E}_\lambda} &= \sum_{J\in\mathcal{D}^{\kappa(\mu)}}T_JS_J|_{Y^\mathcal{E}_\lambda} = \sum_{J\in\mathcal{D}^{\kappa(\mu)}}T_JS_{\mu^\smallfrown J^\smallfrown \nu_J}|_{Y^\mathcal{E}_\lambda} = \sum_{J\in\mathcal{D}^{\kappa(\mu)}}T_J  T_{\Psi_\lambda^{\mu^\smallfrown J^\smallfrown \nu_J}}T_\lambda^{-1}|_{Y^\mathcal{E}_\lambda}\\
&= \sum_{J\in\mathcal{D}^{\kappa(\mu)}}T_J  T_{\Psi_\lambda^{\mu^\smallfrown J}}T_\lambda^{-1}|_{Y^\mathcal{E}_\lambda}\text{ (by Lemma \ref{operator reduction going down the path})}\\
&= T_{\Psi_\lambda^\mu}T_{\lambda}^{-1}|_{Y^\mathcal{E}_\lambda}\text{ (by Lemma \ref{operator reduction successor to limit})}.
\end{align*}
\end{proof}
}
}


\section{The orthogonal reduction to a scalar FDD-diagonal operator}

\label{reduction to FDD}


{
We show that for any bounded linear operator $T:R_p^{\alpha,0}\to R_p^{\alpha,0}$, with $\alpha$ a countable limit ordinal, and any $\epsilon>0$, there exist a distributional embedding scheme $\Psi$ of $R_\alpha^{p,0}$ into itself and a scalar FDD-diagonal operator $R:R_p^{\alpha,0}\to R_p^{\alpha,0}$ such that $\|T_\Psi^\dagger TT_\Psi-R\|<\epsilon$. The entries of $R$ are averages of the diagonal entries of $T$, and hence, as an application, we obtain the factorization property of $R_\alpha^{p,0}$.  The construction of $\Psi$ is inductive. Enumerate $\mathcal{T}_\alpha$ as $(\lambda_n)_{n=1}^\infty$, compatible with the tree order. At each step we select $(\mathcal{M}_{\lambda_n},(J_\mu)_{\mu\in\mathcal{M}_{\lambda_n}})$ and a scalar $r_n$ such that, letting $S_n=\sum_{\mu\in\mathcal{M}_{\lambda_n}}T_\mu J_\mu:L_p^{\kappa(\lambda_n),0}\to L_p^0$, $Y_n$ its image, and $E_n$ the orthogonal projection onto $Y_n$, we have
\[
\Big\|\big(E_nT-r_n\cdot id\big)|_{Y_n}\Big\|<\frac{\epsilon}{3\cdot 2^n},\quad 
\Big\|\big(\sum_{m=1}^{n-1}E_m\big)TE_n\Big\|<\frac{\epsilon}{3\cdot 2^n},\quad 
\Big\|E_nT\big(\sum_{m=1}^{n-1}E_m\big)\Big\|<\frac{\epsilon}{3\cdot 2^n}.
\]
Provided Definition~\ref{DES} is respected, this yields $\|T_\Psi^\dagger T T_\Psi-R\|<\epsilon$, where $R|_{X_{\lambda_n}}=r_n\cdot id|_{X_{\lambda_n}}$.  The inductive choice of $(\mathcal{M}_{\lambda_{n+1}},(J_\mu)_{\mu\in\mathcal{M}_{\lambda_{n+1}}})$ rests on a finite-dimensional stabilization argument depending on the parameters $\dim(X_{\lambda_{n+1}})$, $\sum_{m=1}^n\dim(X_{\lambda_m})$, $\|T\|$, and the error $\epsilon/2^{n+1}$. These determine a large integer $\kappa_{n+1}$, permitting the choice of $\mathcal{M}_{\lambda_{n+1}}$, respecting $(\mathcal{M}_{\lambda_m},(J_\mu)_{\mu\in\mathcal{M}_{\lambda_m}})_{m=1}^n$, with $\kappa(\mu)=\kappa_{n+1}$ for all $\mu\in\mathcal{M}_{\lambda_{n+1}}$. This provides enough room to carry out a probabilistic choice of the family $(J_\mu)_{\mu\in\mathcal{M}_{\lambda_{n+1}}}$ and define $S_{n+1}$. The probabilistic techniques employed here are by now standard in the study of factorization and related properties, originating in \cite{lechner:2018:1-d} (diagonal reduction) and \cite{lechner:motakis:mueller:schlumprecht:2022,lechner:motakis:mueller:schlumprecht:2023} (scalar reduction). Our method is close to \cite{konstantos:motakis:2025}, one of whose statements we use directly, though non-trivial modifications to the approach are required to handle simultaneously the compressions $(T_\mu)_{\mu\in\mathcal{M}_{\lambda_{n+1}}}$ and to extract the distributional embeddings $(J_\mu)_{\mu\in\mathcal{M}_{\lambda_{n+1}}}$. The section is divided into two parts: the finite-dimensional stabilization and the construction of the distributional embedding scheme, the latter completing the proof.

\subsection{A finite-dimensional stabilization}

We now outline the main result of this subsection and its proof. Given $n,m\in\mathbb{N}$, $\Gamma>0$, and $0<\epsilon<1$, we show the existence of a large integer $k_0''$ with the following property: for every operator $T$ of norm at most $\Gamma$, operators $Q_1$, $Q_2$ of norm at most $\Gamma$ and rank at most $m$, every $k\geq k_0''$, and compressions $S_l:L_p^{k,0}\to L_p^0$, $1\leq l\leq N$, with pairwise disjointly supported ranges ($k_0''$ does not depend on $N$), there exist distributional embeddings $J_l:L_p^{n,0}\to L_p^{k,0}$, $1\leq l\leq N$, such that, letting $S=\sum_{l=1}^N S_lJ_l:L_p^{n,0}\to L_p^0$, with image $Y$ and $E$ the orthogonal projection onto $Y$, the operator $ET|_Y$ is $\epsilon$-close to $r\cdot id|_Y$, with $r$ given by an explicit formula, and $EQ_1$ and $Q_2E$ have norm less than $\epsilon$. This result is directly applicable in the context outlined above. It is analogous to the stabilization of \cite[Theorems 5.7 and 6.5]{konstantos:motakis:2025} (see also \cite{speckhofer:2025}), but here $Y$ is a compressed copy of a finite-dimensional $L_p^0$-space constructed by simultaneously taking pieces from all finite-dimensional compressed $L_p^0$-spaces $S_l[L_p^{k,0}]$, $1\leq l\leq N$.

The proof proceeds in two stages: first, a reduction where $ET|_Y$ is shown to be close to a diagonal operator $R:Y\to Y$, and then a further reduction to a scalar multiple of the identity, each step incurring a dimensional penalty. The first step is proved here in full; although similar arguments appear in the literature, to our knowledge it does not follow directly from previous work. The construction proceeds by enumerating $\mathcal{D}_{n-1}$ lexicographically and, assuming $(J_l h_J)_{l=1}^N$ has been defined for $J<I$, simultaneously defining $(J_l h_I)_{l=1}^N$ by means of a random family $(J_l h_I(\vartheta))_{l=1}^N$, where $\vartheta$ ranges over a suitable probability space. One then shows that, with high probability, sufficient criteria are satisfied. The second step is not proved from scratch: we first dilate an appropriate compressed finite-dimensional $L_p^0$-space to a non-compressed one, apply the stabilization of \cite[Theorem 6.5]{konstantos:motakis:2025}, and finally compress back to the original setting.

To implement this scheme, we establish the probabilistic framework in advance. Vectors in $L_p^0$ are randomized to produce distributional embeddings. The notation together with the variance bounds that follow are essentially from \cite{lechner:2018:1-d}.

\begin{ntt}
\label{random variables}
Fix a finite collection of disjointly supported and symmetric $\{-1,0,1\}$-valued measurable functions $\mathcal{B} = (h_k)_{k=1}^m$. Denote
\[H_k  =\mathrm{supp}(h_k), 1\leq k\leq m,\text{ and }H = \cup_{k=1}^mH_k.\]
Considering the probability space $\Omega = \{-1,1\}^m$ equipped with the uniform probability measure, we define the random vector $b_\mathcal{B}:\Omega\to\langle\{h_k:1\leq k\leq m\}\rangle$ given by
\[b_\mathcal{B}(\vartheta) = \sum_{k=1}^m\theta_kh_k,\]
where $\vartheta = (\theta_k)_{k=1}^m\in\Omega$. For $g\in L_q$, $f\in L_p$, and linear operator $T:\langle\{h_k:1\leq k\leq m\}\rangle\to L_p$, we define the random variables $Y_{\mathcal{B},g},W_{\mathcal{B},f}, Z_{\mathcal{B},T}:\Omega\to\mathbb R$ given by
\begin{align*}
 Y_{\mathcal{B},g}(\vartheta) = \langle g,  b_{\mathcal{B}}(\vartheta) \rangle\text{ and }W_{\mathcal{B},f}(\vartheta) = \langle  b_{\mathcal{B}}(\vartheta), f \rangle,
\end{align*}
and
\begin{align*}
 Z_{\mathcal{B},T}(\vartheta) = \langle b_{\mathcal{B}}(\vartheta), Tb_{\mathcal{B}}(\vartheta) \rangle - \sum_{k=1}^m \langle h_{k}, Th_{k}\rangle.
\end{align*}
\end{ntt}

\begin{lem} \label{aupper estimates} 
Following Notation \ref{random variables}, we have $\mathbb{E}(Y_{\mathcal{B},g}) = \mathbb{E}(W_{\mathcal{B},f}) = \mathbb{E}(Z_{\mathcal{B},T}) = 0$ and
\begin{align*}
\mathbb{V} (Y_{\mathcal{B},g}) &\leq \| g \|_{q}^{2}\Big(|H|\max_{1\leq k\leq m}|H_K|\Big)^{1/p} ,\\
\mathbb{V}(W_{\mathcal{B},f}) &\leq \| f \|_{p}^{2}\Big(|H|\max_{1\leq k\leq m}|H_K|\Big)^{1/q} ,\text{ and}\\
\mathbb{V}(Z_{\mathcal{B},T}) &\leq 2\|T\|^2|H|^{1+1/p}\Big(\max_{1\leq k\leq m}|H_k|\Big)^{1/q}.
\end{align*}
\end{lem}

\begin{proof}
All of the above rely on the fact that the functions $\theta_k:\Omega\to\{-1,1\}$, $1\leq k\leq m$, are independent with mean zero and satisfy $\theta_k^2 = 1$. For illustrative purposes, we prove the last bound, as it is the least straightforward.
\begin{align*}
\mathbb{V}(Z_{\mathcal{B},T}) &= \mathbb{E}\Big(\sum_{k,l=1\atop k\neq l}^m\theta_k\theta_m\langle h_k,Th_l\rangle\Big)^2 = \underbrace{\sum_{k,l=1\atop k\neq l}^m\langle h_k,Th_l\rangle\langle h_l,Th_k\rangle}_{=:A} + \underbrace{\sum_{k,l=1\atop k\neq l}^m\langle h_k,Th_l\rangle^2}_{=:B}.
\end{align*}
We compute bounds for $A$ and $B$ separately:
\begin{align*}
A & = \sum_{l=1}^m\Big\langle h_l, T\Big(\sum_{k=1\atop k\neq l}^m\langle  h_k,Th_l\rangle h_k\Big)\Big\rangle  \leq \|T\| \sum_{l=1}^m\|h_l\|_{L_q}\Big\|\sum_{k=1\atop k\neq l}^m\langle h_k,Th_l\rangle h_k\Big\|_{L_p}\\
&\leq \|T\|^2\sum_{l=1}^m\|h_l\|_{L_q}\max_{1\leq k\leq m}\big(\|h_k\|_{L_q}\|h_l\|_{L_p}\big)\Big\|\sum_{k=1}^mh_k\Big\|_{L_p}= \|T\|^2\sum_{l=1}^m|H_l|^{1/q}|H_l|^{1/p}\big(\max_{1\leq k\leq m}|H_k|^{1/q}\big)|H|^{1/p}\\
&= \|T\|^2\Big(\max_{1\leq l\leq m}|H_l|\Big)^{1/q}|H|^{1+1/p}
\end{align*}
and
\begin{align*}
B &=\sum_{k=1}^m\Big\langle h_k, T\Big(\sum_{l=1\atop l\neq k}^m\langle  h_k,Th_l\rangle h_l\Big)\Big\rangle \leq \|T\|\sum_{k=1}^m\|h_k\|_{L_q}\Big\|\sum_{l=1\atop l\neq k}^m\langle  h_k,Th_l\rangle h_l\Big\|_{L_p}\\
&\leq \|T\|^2\sum_{k=1}^m\|h_k\|_{L_q} \max_{1\leq l\leq m}\big(\|h_k\|_{L_q}\|h_l\|_{L_p}\big) \Big\|\sum_{l=1}^mh_l\Big\|_{L_p} =\|T\|^2|H|^{1/p} \Big(\max_{1\leq l\leq m}|H_l|\Big)^{1/p}\sum_{k=1}^m|H_k|^{2/q}\\
&\leq \|T\|^2|H|^{1/p} \Big(\max_{1\leq l\leq m}|H_l|\Big)^{1/p}\Big(\max_{1\leq k\leq m}|H_k|\Big)^{2/q-1}|H| = \|T\|^2\Big(\max_{1\leq l\leq m}|H_l|\Big)^{1/q}|H|^{1+1/p}.
\end{align*}
By placing the sum on the left-hand side of the bilinear form $\langle \cdot,\cdot\rangle$ in the computations for the bounds of $A$ and $B$, one can interchange the roles of $p$ and $q$ in the resulting expressions. This refinement yields the sharper estimate
$\mathbb{V}(Z_{\mathcal{B},T}) \;\leq\; 2\|T\|^2\,|H|^{1+1/p^*}(\max_{1\leq k\leq m}|H_k|)^{\,1-1/p^*}$,
although this improvement is not required in our application.
\end{proof}

\begin{rem} \label{notation for reduction to FDD} Let $m \in \mathbb{N}$, $\theta\in(0,1]$, and $S:L_p^{m,0}\to L_p^0$ be a $\theta$-compression. The orthogonal projection $E \colon L_{p}^{0}  \to  S[L_p^{m,0}]$ can be expressed as
\begin{align*}
E(f) = \sum_{I \in \mathcal{D}_{m-1}} \theta^{-1}\vert I \vert^{-1} \langle Sh_{I}, f \rangle Sh_{I} = \mathbb{E}(f\chi_{A}|\mathcal{A}),
\end{align*}
where $\mathcal{A} = \sigma (Sh_{I} : I \in \mathcal{D}_{m-1})$ and $A = \cup_{I\in\mathcal{D}^{m-1}}\mathrm{supp}(Sh_I)$ (the above equality of formulas holds in $L_p^0$ but not in $L_p$). We will mainly rely on the former, while the latter reveals that $\|E\| = 1$.
\end{rem}

This first step of the reduction is carried out in full and yields a stabilization where $ET|_Y$ is close to a diagonal operator $R$ on a compressed copy $Y$ of $L_p^{n,0}$, obtained from pieces of disjoint compressed copies of $L_p^{k,0}$ for sufficiently large $k$.

\begin{lem} \label{f.d. diagonalization 2.0} Let $n \in \mathbb{N}$, $m \in \mathbb{N}$, $\Gamma >0$,  and $0<\epsilon<1$. There exists $k_0 := k_0(n, m, \Gamma, \epsilon) \in \mathbb{N}$ such that for every $k \geq k_0 +n$ and
\begin{enumerate}[label=(\alph*)]
\item \label{f.d. diagonalization 2.0 a}  operator $T \colon X \to L^0_p$ of norm at most $\Gamma$, where $X$ is a subspace of $L_p^0$,

\item \label{f.d. diagonalization 2.0 b} operators $Q_1:Z\to L_p^0,Q_2:X\to W$ of norm at most $\Gamma$ and rank at most $m$, where $Z$, $W$ are normed spaces, and

\item \label{f.d. diagonalization 2.0 c} operators $S_1,\ldots,S_N:L_p^{k,0}\to X$ with disjointly supported ranges such that, for $1\leq l\leq N$, $S_l$ is a $\theta_l$-compression, for some $\theta_l\in(0,1]$,

\end{enumerate}
there exist distributional embeddings $J_l:L_p^{n,0}\to L_p^{k,0}$, $1\leq l\leq N$, satisfying the following:
\begin{enumerate}[label=(\roman*)]

\item\label{f.d. diagonalization 2.0 i} For each $1\leq l\leq N$, $0\leq i\leq n-1$ and $I\in\mathcal{D}^{i}$, there are $\mathcal{B}^l_{I} \subset \mathcal{D}^{k_{0}+i}$ and $(\theta^l_{K})_{K \in \mathcal{B}^l_{I}} \in \lbrace -1,1 \rbrace^{\mathcal{B}^l_{I}}$ such that
\[J_lh_I = \sum_{K \in \mathcal{B}^l_{I}} \theta^l_{K} h_{K}.\]

\item\label{f.d. diagonalization 2.0 ii} Denote $S = \sum_{l=1}^NS_lJ_l:L_p^{n,0}\to X$, which is a $\theta$-compression with $\theta = \sum_{l=1}^N\theta_l$, $b_I = Sh_I$, $I\in\mathcal{D}_{n-1}$, $Y = [ (b_{I})_{I \in \mathcal{D}_{n-1}}]$, $E:L_p^0\to Y$ the orthogonal projection, and $R:Y\to Y$ the diagonal operator given by $Rb_I = d_Ib_I$, where
\begin{align*}
d_{I} = \frac{1}{|I|\theta}\sum_{l=1}^N \sum_{K \in \mathcal{B}^l_{I}}\langle S_lh_{K} ,TS_lh_{K} \rangle, \;I\in\mathcal{D}_{n-1}.
\end{align*}   
Then
\begin{align*}
\Big\| \Big( ET - R \Big)|Y \Big\| < \epsilon.
\end{align*}

\item\label{f.d. diagonalization 2.0 iii} We have $\|EQ_1\| <\epsilon$ and $\|Q_2E\| <\epsilon$.

\end{enumerate}
\end{lem}

\begin{proof} 
Let $M = 5^m-1$, which is sufficiently large such that the unit sphere of every $m$-dimensional normed space contains an $(1/2)$-dense set of cardinality at most $M$. Let 
\begin{align} \label{f.d. diagonalization 2.0 Appropriate depth}
k_{0}(n, m, \Gamma, \epsilon)  = \Big\lfloor p^{\ast} \Big( 5n+3 + 2\log(\Gamma)+3m+2\log(\epsilon^{-1})\Big)\Big\rfloor +1,
\end{align}
and fix $k\geq k_{0}+n$ and items such as in \ref{f.d. diagonalization 2.0 a} to \ref{f.d. diagonalization 2.0 c}. Fix $(1/2)$-dense subsets $(x_i)_{i=1}^M$ of the unit sphere of  $\mathrm{range}(Q_1)$ and $(y_i^*)_{i=1}^M$ of the unit sphere of the dual of $\mathrm{range}(Q_2)$.

For $1\leq l\leq N$, we will define $J_l:L_p^{n,0}\to L_p^k$ by choosing a distributional copy $(c^l_I)_{I\in\mathcal{D}_{n-1}}$ of $(h_I)_{I\in\mathcal{D}_{n-1}}$ in $L_p^{k,0}$, and letting $J_lh_I = c_I^l$, $I\in\mathcal{D}_{n-1}$. Following the notation of the statement, the construction will satisfy \ref{f.d. diagonalization 2.0 i} and
\begin{align} \label{f.d. diagonalization 2.0 Imp 1}
\Big| \Big\langle b_{I}, Tb_{J} \Big\rangle \Big| < \frac{\epsilon\theta|I|^{1/q}|J|^{1/p}}{2^{2n+1}}, \;\; I \neq J \in \mathcal{D}_{n-1},  
\end{align}
\begin{align} \label{f.d. diagonalization 2.0 Imp 2}
\Big| \theta^{-1}|I|^{-1}\Big\langle b_{I}, Tb_{I} \Big\rangle - d_I \Big| < \frac{\epsilon}{2^{n+1}}, \;\; I  \in \mathcal{D}_{n-1},
\end{align}
\begin{equation}
\label{f.d. diagonalization 2.0 Imp 3}
|\langle b_I,x_i\rangle|\leq \frac{\epsilon(\theta|I|)^{1/q}}{\Gamma 2^{n+1}}\text{ and }\big| Q_2^*x_i^*(b_I)\big| \leq \frac{\epsilon(\theta|I|)^{1/p}}{2^{n+1}}, \; 1\leq i\leq M.
\end{equation}
We note that \eqref{f.d. diagonalization 2.0 Imp 1} and \eqref{f.d. diagonalization 2.0 Imp 2} yield \ref{f.d. diagonalization 2.0 ii} as follows:
for $y \in Y $, with $\|y\|\leq 1$,
\begin{align*}
\Big\|ETy - Ry\Big\|& = \Big\|\sum_{I\in\mathcal{D}_{n-1}}\Big(\sum_{J\in\mathcal{D}_{n-1}}\theta|I|^{-1}\theta|J|^{-1}\langle b_I,Tb_J\rangle\langle b_J,y\rangle \Big)b_I - \sum_{I\in\mathcal{D}_{n-1}}\theta^{-1}|I|^{-1}d_I\langle b_I,y\rangle b_I\Big\|\\
&\leq \Big\|\sum_{I\in\mathcal{D}_{n-1}}\theta^{-1}|I|^{-1}\langle b_I,y\rangle\Big(\theta^{-1}|I|^{-1}\langle b_I,Tb_I\rangle - d_I\Big)b_I\Big\|\\
&+ \Big\|\sum_{I\in\mathcal{D}_{n-1}}\Big(\sum_{J\in\mathcal{D}_{n-1}\atop J\neq I}\theta^{-2}|I|^{-1}|J|^{-1}\langle b_I,Tb_J\rangle\langle b_J,y\rangle\Big)b_I\Big\|\\
&< \sum_{I\in\mathcal{D}_{n-1}}\theta^{-1}|I|^{-1}\theta^{1/q}|I|^{1/q}\theta^{1/p}|I|^{1/p}\frac{\epsilon}{2^{n+1}}\\
&+ \sum_{I\in\mathcal{D}_{n-1}}\sum_{J\in\mathcal{D}_{n-1}\atop J\neq I}\theta^{-2}|I|^{-1}|J|^{-1}\theta^{1/q}|J|^{1/q}\theta^{1/p}|I|^{1/p} \frac{\epsilon\theta|I|^{1/q}|J|^{1/p}}{2^{2n+1}}\\
&\leq \big(\#\mathcal{D}_{n-1}\big) \frac{\epsilon}{2^{n+1}}+ \big(\#\mathcal{D}_{n-1}\big)^2\frac{\epsilon}{2^{2n+1}} <\epsilon.
\end{align*}
Furthermore, \eqref{f.d. diagonalization 2.0 Imp 3} yields \ref{f.d. diagonalization 2.0 iii} as follows: for $1\leq i\leq M$,
\begin{align*}
\|Ex_i\| & = \Big\|\sum_{I\in\mathcal{D}_{n-1}}\theta^{-1}|I|^{-1}\langle b_I ,x_i\rangle b_I\Big\| \leq \sum_{I\in\mathcal{D}_{n-1}}\theta^{-1}|I|^{-1}\frac{\epsilon(\theta|I|)^{1/q}}{\Gamma 2^{n+1}}(\theta|I|)^{1/p} <\frac{\epsilon}{2\Gamma},
\end{align*}
and $\|EQ_1\|\leq 2\Gamma\max_{1\leq i\leq M}\|Ex_i\| <\epsilon$. An argument of a similar nature yields $\|Q_2E\| <\epsilon$.

We return to the main objective of the proof, namely, the construction of the distributional copies $(c^l_I)_{I\in\mathcal{D}_{n-1}}$ of $(h_I)_{I\in\mathcal{D}_{n-1}}$ in $L_p^{k,0}$, $1\leq l\leq N$, such that item \ref{f.d. diagonalization 2.0 i} and \eqref{f.d. diagonalization 2.0 Imp 1} to \eqref{f.d. diagonalization 2.0 Imp 3} are satisfied. We proceed by induction on the lexicographical order of $\mathcal{D}_{n-1}$, and for each $I\in\mathcal{D}_{n-1}$, if $I\in\mathcal{D}^i$, we will choose, simultaneously for all $1\leq l\leq N$, collections $\mathcal{B}^l_{I} \subset \mathcal{D}^{k_0+i}$, with $\cup\mathcal{B}^l_I = [c^l_{\pi(I)} = \epsilon(I)]$ (or $\mathcal{B}_{[0,1)} = \mathcal{D}^{k_0}$, for $I = [0,1)$),  and $(\theta^l_{K})_{K \in \mathcal{B}_{I}} \in \lbrace -1, 1 \rbrace^{\mathcal{B}^l_{I}}$, and let
\[c_I^l = \sum_{K\in\mathcal{B}_I^l}\theta_K^lh_I.\]
We skip the basis step because the argument is analogous to the general one. Let $I\in\mathcal{D}^i$, with $1\leq i\leq n-1$, and assume that for every $J$ preceding $I$ in the lexicographical order, we have chosen the collection $\mathcal{B}^l_{J}$ and the signs $(\theta^l_K)_{K\in\mathcal{B}^l_J}$, $1\leq l\leq N$, satisfying the inductive hypothesis so far. For $1\leq l\leq N$, define 
\begin{align*}
\mathcal{B}^l_{I} = \Big\lbrace K \in \mathcal{D}^{k_{0}+i}: K \subset [c^l_{\pi(I)} = \epsilon(I)] \Big\rbrace,
\end{align*}
i.e., $\mathcal{B}^l_I$ satisfies $\cup\mathcal{B}^l_I = [c^l_{\pi(I)} = \epsilon(I)]$. We consider the randomized vector
\begin{align*}
b_{I}(\vartheta) = \sum_{l=1}^N\sum_{K \in \mathcal{B}^l_{I}} \theta^l_{K} S_lh_{K}
\end{align*}
where
\[\vartheta = \big((\theta^l_{K})_{K \in \mathcal{B}^l_{I}}\big)_{l=1}^N \in\Omega:= \prod_{l=1}^N\lbrace -1, 1 \rbrace^{\mathcal{B}^l_{I}},\]
equipped with the uniform measure. Note that $\big((S_lh_K)_{K\in\mathcal{B}^l_I}\big)_{l=1}^N$ are disjointly supported and symmetric $\{-1,0,1\}$-valued vectors. We introduce the following random variables based on Notation \ref{random variables}. For $\vartheta \in\Omega$, we let
\begin{gather*}
Y_{J}(\vartheta) := Y_{\mathcal{B}_{I}, T^{*}b_{J}}(\vartheta)  =\Big\langle b_{J}, Tb_{I}(\vartheta) \Big\rangle, \;\;  J < I \in \mathcal{D}_{n-1},\\
W_{J}(\vartheta) := W_{\mathcal{B}_{I}, Tb_{J}} (\vartheta) = \Big\langle b_{I}(\vartheta), Tb_{J} \Big\rangle, \;\;  J < I \in \mathcal{D}_{n-1},\\
Z_{I}(\vartheta) := Z_{\mathcal{B}_{I}, T}(\vartheta) =  \Big\langle b_{I}(\vartheta), Tb_{I}(\vartheta) \Big\rangle -  \sum_{l =1}^N \sum_{K \in \mathcal{B}^l_{I}} \Big\langle S_{l}h_{K} ,TS_{l}h_{K} \Big\rangle, \text{ and}\\
W_{j}(\vartheta) := W_{\mathcal{B}_{I}, f_{j}} (\vartheta) =  \langle b_{I}(\vartheta), x_{j} \rangle\;\; \text{and}\;\; Y_{j}(\vartheta) := Y_{\mathcal{B}_{I}, g_{j}} (\vartheta) =  Q_2^*x_j^*\big( b_{I}(\vartheta)\big), \;\;  1 \leq j \leq M.
\end{gather*}

We relate \eqref{f.d. diagonalization 2.0 Imp 1} to \eqref{f.d. diagonalization 2.0 Imp 3} to these random variables as follows: there exists $\vartheta\in\Omega$ such that
\begin{gather}
\label{f.d. diagonalization 2.0 what signs to choose restated in rv language}
\begin{gathered}
\Big| Y_{J} (\vartheta) \Big| < \frac{\epsilon\theta|J|^{1/q}|I|^{1/p}}{2^{2n+1}}\text{ and }\Big| W_{J}(\vartheta) \Big| < \frac{\epsilon\theta|I|^{1/q}|J|^{1/p}}{2^{2n+1}}, \;\;  J < I \in \mathcal{D}_{n-1},\\
\Big| Z_{I}(\vartheta) \Big| <  \frac{\epsilon\theta|I|}{2^{n+1}},\text{ and}\\
\vert W_{j} (\vartheta) \vert < \frac{\epsilon{(\theta|I|)^{1/q}}}{\Gamma 2^{n+1}}\;\; \text{and}\;\; \vert Y_{j} (\vartheta) \vert < \frac{\epsilon{(\theta|I|)^{1/p}}}{2^{n+1}}, \;\;  1 \leq j \leq M.
\end{gathered}
\end{gather}
Once the existence of such $\vartheta$ has been established, we can let, for $1\leq l\leq N$, $c_I^l =\sum_{K\in\mathcal{B}_I^l}\theta_K^lh_K$, and the proof will be complete.  We proceed by examining the following events in $\Omega$:
\begin{gather*}
\mathcal{Y}_{J} = \Big[ \big| Y_{J} (\vartheta) \big| \geq \frac{\epsilon\theta|J|^{1/q}|I|^{1/p}}{2^{2n+1}} \Big]\text{ and }\mathcal{W}_{J} = \Big[ \big| W_{J} (\vartheta) \big| \geq \frac{\epsilon\theta|I|^{1/q}|J|^{1/p}}{2^{2n+1}} \Big],\;\;  J < I \in \mathcal{D}_{n-1},\\
\mathcal{Z}_{I} = \Big[ \big| Z_{I} (\vartheta) \big| \geq \frac{\epsilon\theta|I|}{2^{n+1}}\Big], \text{ and}\\
\mathcal{W}_{j} = \Big[ \big| W_{j} (\vartheta) \big| \geq \frac{\epsilon{(\theta|I|)^{1/q}}}{\Gamma 2^{n+1}} \Big] \text{ and }\;\; \mathcal{Y}_{j} = \Big[ \big| Y_{j} (\vartheta) \big| \geq \frac{\epsilon {(\theta|I|)^{1/p}}}{2^{n+1}}\Big], \;\;  1 \leq j \leq M.
\end{gather*}
Denoting $\mathcal{E}$ the union of the above events, it suffices to show $\mathbb{P}(\mathcal{E})<1$. Taking the probability of $\mathcal{E}$ and applying Chebyshev's inequality, we obtain:
\begin{align*}   
\mathbb{P} ( \mathcal{E} ) &\leq \sum_{J < I \in \mathcal{D}_{n-1}}\mathbb{P}(\mathcal{Y}_{J})  +  \sum_{J < I \in \mathcal{D}_{n-1}}\mathbb{P}(\mathcal{W}_{J}) + \mathbb{P} (\mathcal{Z}_{I} ) + \sum_{j=1}^{m} \mathbb{P} (\mathcal{W}_{j})  +  \sum_{j=1}^{m}\mathbb{P} (\mathcal{Y}_{j})\\
& \leq \sum_{J < I \in \mathcal{D}_{n-1}} \frac{2^{4n+2}}{\epsilon^2\theta^2|J|^{2/q}|I|^{2/p}}\mathbb{V}(Y_{J})  + \sum_{J < I \in \mathcal{D}_{n-1}} \frac{2^{4n+2}}{\epsilon^2\theta^2|I|^{2/q}|J|^{2/p}} \mathbb{V}(W_{J})  + \frac{2^{2n+2}}{\epsilon^2\theta^2|I|^2} \mathbb{V}(Z_{I})\\
&+ \sum_{j=1}^{M} \frac{\Gamma^22^{2n+2}}{\epsilon^{2}{(\theta|I|)^{2/q}}}\mathbb{V} (W_{j}) +  \sum_{j=1}^{M} \frac{2^{2n+2}}{\epsilon^{2}{(\theta|I|)^{2/p}}}\mathbb{V} (Y_{j}).
\end{align*}
Lemma \ref{aupper estimates} yields the following estimates for the above variances:
\begin{align*}
\mathbb{V}(Y_{J}) &\leq \|T\|^2\big(\theta|J|\big)^{2/q}\Big(\theta\frac{1}{2^i} \big(\max_{1\leq l\leq N}\theta_l\big)\frac{1}{2^{k_0+i}}\Big)^{1/p} \leq \frac{\theta^2\Gamma^2|J|^{2/q}}{2^{(k_0+2i)/p}},\;\;  J < I \in \mathcal{D}_{n-1},  
\end{align*}
\begin{align*}
\mathbb{V}(W_{J})  &\leq \|T\|^2 \big(\theta|J|\big)^{2/p}\Big(\theta\frac{1}{2^i} \big(\max_{1\leq l\leq N}\theta_l\big)\frac{1}{2^{k_0+i}}\Big)^{1/q}\leq \frac{\theta^2\Gamma^2|J|^{2/p}}{2^{(k_0+2i)/q}},\;\;  J < I \in \mathcal{D}_{n-1},  
\end{align*}
\begin{align*} 
\mathbb{V}(Z_{I}) &\leq 2\|T\|^2\Big(\theta\frac{1}{2^i}\Big)^{1+1/p}\Big(\max_{1\leq l\leq N}\theta_l\frac{1}{2^{k_0+i}}\Big)^{1/q}\leq \frac{2\theta^2\Gamma^2}{2^{2i+k_0/q}},\text{ and}
\end{align*}
\begin{align*}
\mathbb{V}( W_{j})&\leq \Big(\theta\frac{1}{2^i} \big(\max_{1\leq l\leq N}\theta_l\big)\frac{1}{2^{k_0+i}}\Big)^{1/q}\leq \frac{\theta^{2/q}}{2^{(k_0+2i)/q}} \text{ and}\\
\mathbb{V}( Y_{j})&\leq \Gamma^2\Big(\theta\frac{1}{2^i} \big(\max_{1\leq l\leq N}\theta_l\big)\frac{1}{2^{k_0+i}}\Big)^{1/p}\leq \frac{\Gamma^2\theta^{2/p}}{2^{(k_0+2i)/p}},\;\; 1 \leq j \leq M.
\end{align*}
By combining the upper bound for $\mathbb P(\mathcal E)$ with the variance bounds we conclude:
\begin{align*}   
\mathbb{P} ( \mathcal{E} ) &\leq  \sum_{J < I \in \mathcal{D}_{n-1}} \frac{2^{4n+2}}{\epsilon^2\theta^2|J|^{2/q}|I|^{2/p}}\frac{\theta^2\Gamma^2|J|^{2/q}}{2^{(k_0+2i)/p}} + \sum_{J < I \in \mathcal{D}_{n-1}} \frac{2^{4n+2}}{\epsilon^2\theta^2|I|^{2/q}|J|^{2/p}}\frac{2\theta^2\Gamma^2|J|^{2/p}}{2^{(k_0+2i)/q}}\\
&+\frac{2^{2n+2}}{\epsilon^2\theta^2|I|^2}\frac{2\theta^2\Gamma^2}{2^{2i+k_0/q}} + M\frac{\Gamma^22^{2n+2}}{\epsilon^2{(\theta|I|)^{2/q}}}\frac{\theta^{2/q}}{2^{(k_0+2i)/q}} + M\frac{2^{2n+2}}{\epsilon^2{(\theta|I|)^{2/p}}}\frac{\Gamma^2\theta^{2/p}}{2^{(k_0+2i)/p}}\\
&\leq 2^i\frac{2^{4n+2}\Gamma^2}{\epsilon^22^{k_0/p}} + 2^i\frac{2^{4n+2}\Gamma^2}{\epsilon^22^{k_0/q}} + \frac{2^{2n+2}2\Gamma^2}{\epsilon^22^{k_0/q}} + \frac{2^{2n+3}M\Gamma^2}{\epsilon^22^{k_0/p^*}} \leq \frac{2^{5n+3}\Gamma^2(1+M)}{\epsilon^22^{k_0/p^*}}.
\end{align*}
From \eqref{f.d. diagonalization 2.0 Appropriate depth}, $\mathbb{P} ( \mathcal{E} ) < 1$, finishing the proof.
\end{proof}

The following lemma is a finite-dimensional stabilization and it is, in essence, \cite[Theorem 6.5]{konstantos:motakis:2025}. The latter statement does not specify the exact form of the vectors $b_I$ and the scalar $r$, but these are evident in the proof provided there.

\begin{lem} \label{f.d.stab. 2.0}  Let $n \in \mathbb{N}$, $\Gamma >0$ and $0<\epsilon <1$. There exists $k_0' = k_0(n,\Gamma,\epsilon)\in\mathbb{N}$ such that for every $k\geq k_0'$ and every diagonal operator $R \colon L_{p}^{k,0} \to  L_{p}^{k,0} $ with $\| R \| \leq \Gamma$, there exists a distributional embedding $J:L_p^{n,0}\to L_{p}^{k,0}$ satisfying the following:
\begin{enumerate}[label=(\roman*)]

\item\label{f.d.stab. 2.0 i} There are $0\leq m_0<m_1<\cdots<m_{n-1} \leq k-1$ such that, for $0\leq i\leq n-1$ and $I\in\mathcal{D}^i$,
\[Jh_{I} = \sum_{K \in \mathcal{B}_{I}} \theta_{K} h_{K},\]
where $\mathcal{B}_{I} \subset \mathcal{D}^{m_i}$ and $(\theta_{K})_{K \in \mathcal{B}_{I}} \in \lbrace -1,1 \rbrace^{\mathcal{B}_{I}}$.

\item\label{f.d.stab. 2.0 ii} Denote $Y = J(L_p^{n,0})$, $E$ the conditional expectation operator onto the space $Y$, and $r = \langle Jh_{[0,1]}, RJh_{[0,1]}\rangle$. Then,

\[\Big\| \Big( ER - r \cdot id \Big)|Y \Big\| < \epsilon.\]

\end{enumerate}
\end{lem}

This is the main result of the subsection. Its proof applies Lemma \ref{f.d. diagonalization 2.0}, translates the outcome to the assumptions of Lemma \ref{f.d.stab. 2.0}, and then expresses the conclusion in the language of the present setting.

\begin{prop} \label{crucial lemma for the FDD-reduction 2.0}
Let $n \in \mathbb{N}$, $m \in \mathbb{N}$, $\Gamma >0$,  and $0<\epsilon<1$. There exists $k_0'' := k''_0(n, m, \Gamma, \epsilon) \in \mathbb{N}$ such that for every $k \geq k_0''$ and
\begin{enumerate}[label=(\alph*)]
\item \label{crucial lemma for the FDD-reduction 2.0 a}  operator $T \colon X \to L_{p}$, where $X$ is a subspace of $L_p$, with $\| T \| \leq \Gamma$,

\item \label{crucial lemma for the FDD-reduction 2.0 b} operators $Q_1:Z\to L_p^0,Q_2:X\to W$ of norm at most $\Gamma$ and rank at most $m$, where $Z$ and $W$ are normed spaces, and

\item \label{crucial lemma for the FDD-reduction 2.0 c} operators $S_1,\ldots,S_N:L_p^{k,0}\to X$ with disjointly supported ranges such that, for $1\leq l\leq N$, $S_l$ is a $\theta_l$-compression, for some $\theta_l\in(0,1]$,

\end{enumerate}
there exist distributional embeddings $J_l:L_p^{n,0}\to L_p^{k,0}$, $1\leq l\leq N$ such that the following hold.
\begin{enumerate}[label=(\roman*)]

\item\label{crucial lemma for the FDD-reduction 2.0 i} Denote $S = \sum_{l=1}^NS_lJ_l:L_p^{n,0}\to X$, which is a $\theta$-compression with $\theta = \sum_{l=1}^N\theta_l$, $b_I = Sh_I$, $I\in\mathcal{D}_{n-1}$, $Y = [ (b_{I})_{I \in \mathcal{D}_{n-1}}]$, and $E:L_p^0\to Y$ the orthogonal projection. Then, for some $0\leq i\leq k-1$ and
\[r = \theta^{-1}\sum_{l=1}^N\sum_{K\in\mathcal{D}^i}\langle S_lh_K,TS_l h_K\rangle,\]
we have
\begin{align*}
\Big\| \Big( ET - r \cdot id \Big)|Y \Big\| < \epsilon.
\end{align*}

\item\label{crucial lemma for the FDD-reduction 2.0 ii} We have $\|EQ_1\| <\epsilon$ and $\|Q_2E\| <\epsilon$.
\end{enumerate}
\end{prop}

\begin{proof}
Let
\begin{align*}
k_1 =  k_{0}^{\prime}(n, \Gamma+\epsilon/2, \epsilon/2)\text{ and }k_0'' = k_0(k_1,m,\Gamma,\epsilon / 2)+k_1,
\end{align*}
where $k_{0}$ is as in Lemma \ref{f.d. diagonalization 2.0}  and $k_0^{\prime}$ is as in Lemma \ref{f.d.stab. 2.0}. Fix $k\geq k_0''$ and items such as in \ref{crucial lemma for the FDD-reduction 2.0 a} to \ref{crucial lemma for the FDD-reduction 2.0 c}. Applying Lemma \ref{f.d. diagonalization 2.0} to these parameters yields distributional embeddings $\hat J_1,\ldots,\hat J_N:L_1^{k_1,0}\to L_1^{k,0}$ satisfying its conclusion. For clarity we adopt the hat notation for all objects associated with them. Specifically, for each $1\leq l\leq N$, $0\leq i\leq k_1-1$, and $I\in\mathcal{D}^i$ there are $\hat{\mathcal{B}}_I^l\subset\mathcal{D}^{k_0+i}$ and signs $(\hat\theta_K^l)_{K\in \hat{\mathcal{B}}_I^l}$ such that
\[\hat J_lh_I = \sum_{K\in \hat{\mathcal{B}}_I^l}\hat\theta_K^l h_K.\]
Defining the $\theta$-compression $\hat S = \sum_{l=1}^N S_l\hat J_l:L_1^{k,0}\to X$, $\hat b_I = \hat S h_I$,
\[\hat d_I = (|I|\theta)^{-1}\sum_{l=1}^N\sum_{K\in\hat{\mathcal{B}}_K^l}\langle S_lh_K,TS_lh_K\rangle,\;I\in\mathcal{D}_{k_1-1},\]
$\hat Y = [(\hat b_I)_{I\in\mathcal{D}_{k_1-1}}]$, $\hat E:L_p^0\to \hat Y$ the orthogonal projection, and $\hat R :Y\to Y$ given by $\hat R\hat b_I = \hat d_I\hat b_I$, we have that Lemma \ref{f.d. diagonalization 2.0} \ref{f.d. diagonalization 2.0 ii} and \ref{f.d. diagonalization 2.0 iii} hold for $\epsilon/2$.

Define
\[R = (\hat S)^{-1}\hat R\hat S:L_1^{k_1,0}\to L_1^{k_1,0}\]
and note that $Rh_I = \hat d_Ih_I$, for $I\in\mathcal{D}_{k_1-1}$, i.e., $R$ is a diagonal operator, and $\|R\| = \|\hat R\| \leq \Gamma+\epsilon/2$. We may therefore apply Lemma \ref{f.d.stab. 2.0} to $R$ to obtain a distributional embedding $J':L_p^{n,0}\to L_p^{k_1,0}$ satisfying its conclusion. For clarity we adopt the hat notation for all objects associated with it. More precisely, there exist $0\leq m_0<m_1<\cdots<m_{n-1} \leq k-1$ such that for $0\leq i\leq n-1$ and $I\in\mathcal{D}^i$, there are $\mathcal{B}'_{I} \subset \mathcal{D}^{m_i}$ and signs $(\theta'_{K})_{K \in \mathcal{B}'_{I}}$ such that
\[J'h_{I} = \sum_{K \in \mathcal{B}'_{I}} \theta'_{K} h_{K}.\]
Denote  $Y' = J'(L_p^{n,0})$, $E'$ the conditional expectation operator onto $Y'$, and
\begin{align*}
r &= \langle Jh_{[0,1]},RJh_{[0,1]} \rangle = \langle  \sum_{L\in\mathcal{D}^{m_0}}\theta'_Lh_L,\sum_{L\in\mathcal{D}^{m_0}}\hat d_L\theta'_Lh_L\rangle = \frac{1}{2^{m_0}}\sum_{L\in\mathcal{D}^{m_0}}\hat d_L\\
& = \frac{1}{2^{m_0}}\sum_{L\in\mathcal{D}^{m_0}}\frac{2^{m_0}}{\theta}\sum_{l=1}^N\sum_{K\in\hat{\mathcal{B}}^l_{L}}\langle S_lh_K,TS_lh_K\rangle = \frac{1}{\theta}\sum_{l=1}^N\sum_{K\in\mathcal{D}^{k_0+m_0}}\langle S_lh_K,TS_lh_K\rangle.
\end{align*}
Then,
\begin{equation}
\label{crucial lemma for the FDD-reduction 2.0 eq1}
\Big\|\Big(E'R - r\cdot id\Big)\Big|_{Y'}\Big\| <\frac{\epsilon}{2}.
\end{equation}

For $1\leq l\leq N$, we define $J_l = \hat J_l J'$, and we will show that the desired conclusion is satisfied. Denote $S = \sum_{l=1}^NS_lJ_l = \hat SJ':L_p^{n,0}\to X$, which is a $\theta$-compression with $\theta = \sum_{l=1}^N\theta_l$, $b_I = Sh_I$, $I\in\mathcal{D}_{n-1}$, $Y = [ (b_{I})_{I \in \mathcal{D}_{n-1}}]$, and $E:L_p^0\to Y$ the orthogonal projection.  Because $\hat S$ is a $\theta$-compression, $\mathrm{dist}(\hat Sf,\hat Sg) = \theta\mathrm{dist}(f,g)+(1-\theta)\delta_{(0,0)}$, for all $f,g$ in its domain. We obtain that for every $I\in\mathcal{D}_{n-1}$, $f\in L_p^{k_1,0}$, 
\[\langle \hat SJ'h_I,\hat Sf\rangle = \theta\langle J'h_I,f\rangle.\]
We deduce that, for $f\in \hat Y$,
\begin{align*}
E'(\hat S)^{-1}f &= \sum_{I\in\mathcal{D}_{n-1}}|I|^{-1}\langle J'h_I,(\hat S)^{-1}f\rangle J'h_I = \sum_{I\in\mathcal{D}_{n-1}}\theta^{-1}|I|^{-1}\langle \hat SJ'h_I,f\rangle J'h_I\\
&= \sum_{I\in\mathcal{D}_{n-1}}\theta^{-1}|I|^{-1}\langle Sh_I,f\rangle (\hat S)^{-1}Sh_I = (\hat S)^{-1}E(f).
\end{align*}
Therefore, $\hat S E'(\hat S)^{-1}|_{\hat Y} = E$. We use this in conjunction with \eqref{crucial lemma for the FDD-reduction 2.0 eq1}:
\begin{align*}
\frac{\epsilon}{2} &> \Big\| \Big(E' (\hat S)^{-1}\hat R\hat S - r\cdot id\Big)\Big|_{Y'}\Big\| = \Big\| \Big(\hat SE' (\hat S)^{-1}\hat R\hat S - r\cdot \hat S\Big)\Big|_{Y'}\Big\| = \Big\| \Big(E\hat R - r\cdot id\Big)\Big|_{\hat S(Y')}\Big\|\\
& = \Big\| \Big(E\hat R - r\cdot id\Big)\Big|_{Y}\Big\|.
\end{align*}
We combine this with the conclusion of Lemma \ref{f.d. diagonalization 2.0} \ref{f.d. diagonalization 2.0 ii} for $T$, $\hat Y$, and $\hat E$, which states
\[\frac{\epsilon}{2} > \Big\|\Big(\hat E T - \hat R\Big)\Big|_{\hat Y}\Big\| \geq \Big\|\Big(E\hat ET- E\hat E\hat R\Big)\Big|_Y\Big\| = \Big\|\Big(ET- E\hat R\Big)\Big|_Y\Big\|,\]
and therefore, $\|(ET-r\cdot id)|_Y\| <\epsilon$.

It remains only to verify \ref{crucial lemma for the FDD-reduction 2.0 ii}, which follows at once from $\|\hat E Q_i\|<\epsilon/2$, $\|Q_2\hat E\|<\epsilon/2$, and $E\hat E=\hat EE=E$.
\end{proof}

\subsection{The stabilization to a scalar FDD-diagonal operator}
We now address the reduction of an arbitrary bounded linear operator $T:R_\alpha^{p,0}\to R_\alpha^{p,0}$ to a scalar FDD-diagonal operator $R$ whose entries are averages of the diagonal entries of $T$. With Proposition \ref{crucial lemma for the FDD-reduction 2.0} in place, we turn to the inductive construction of a distributional embedding scheme $\Psi$, as described at the beginning of Section \ref{reduction to FDD}.

\begin{ntt}
For $\alpha < \omega_{1}$ and a bounded linear operator $T : R_{\alpha}^{p,0} \to R_{\alpha}^{p,0}$, we denote by $\mathcal{A}(D(T))$ the set of averages of distinct terms of the family of diagonal entries of $T$, that is, of $(\theta_{\lambda}^{-1} |I|^{-1} \langle h_{I}^{\lambda}, T h_{I}^{\lambda} \rangle)_{I \in \mathcal{D}_{\kappa(\lambda)-1},\; \lambda \in \mathcal{T}_{\alpha}}$.
\end{ntt}

\begin{thm} \label{reduction to FDD diagonal}
Let $\alpha$ be a limit ordinal number, $T \colon R_{\alpha}^{p,0} \to R_{\alpha}^{p,0}$ be a bounded linear operator and $\epsilon>0$. Then there exist a distributional embedding scheme $\Psi$ from $R_\alpha^{p,0}$ to itself and a scalar FDD-diagonal operator $R \colon R_{\alpha}^{p,0} \to R_{\alpha}^{p,0}$ with diagonal entries in $\mathcal{A}(D(T))$ such that
\[\|T_\Psi^\dagger T T_\Psi - R\|<\epsilon.\]
In particular, $T$ is an orthogonal factor of $R$ with error $\epsilon$.
\end{thm}
\begin{proof}
Let $(\lambda_{n})_{n=1}^{\infty}$ be an enumeration of the elements of $\mathcal{T}_{\alpha}$ compatible with the order of the initial segments, that is $\lambda_{n} \sqsubsetneq \lambda_{m}$ implies $n < m $. We will construct a distributional embedding scheme $\Psi$ of $R_{\alpha}^{p,0}$ in $R_{\alpha}^{p,0}$ and a scalar FDD-diagonal operator $R$ on $R_{\alpha}^{p,0}$ such that $\| T_{\Psi}^{\dagger} T T_{\Psi} -R\| < \epsilon$, and furthermore, for every $n \in \mathbb{N}$, there is an $r_{n}$ in $\mathcal{A}(D(T))$ such that $R|_{X_{\lambda_{n}}} = r_{n} \cdot id$. The construction of $\Psi = ( \mathcal{M}_{\lambda_{n}}, ( J_{\mu} )_{  \mu \in \mathcal{M}_{\lambda_{n}} } )_{n=1}^{\infty}$ and choice of $(r_n)_{n=1}^\infty$ will be carried out inductively on $n\in\mathbb{N}$ such that after $n$ steps, the partial family $( \mathcal{M}_{\lambda_{k}}, ( J_{\mu} )_{  \mu \in \mathcal{M}_{\lambda_{k}} } )_{k=1}^n$ satisfies Definition \ref{DES}, up to that point (the set $\{\lambda_1,\ldots,\lambda_n\}$ is a backward closed subtree of $\mathcal{T}_\alpha$), and the following bounds hold: for every $n\in\mathbb{N}$, denote $S_n = \sum_{\mu\in\mathcal{M}_{\lambda_n}}T_\mu J_\mu$, which is a $\theta_{\lambda_n}$-compression,
\[Y_n = S_n(L_p^{\kappa(\lambda_n),0}) = T_\Psi(X_{\lambda_n}),\]
and $E_n:R_{\alpha}^{p,0}\to Y_n$ the orthogonal projection. Then, for all $n\in\mathbb{N}$,
\begin{gather}
\label{reduction to FDD diagonal eq1}\Big\|\Big(E_nT - r_n\cdot id\Big)\Big|_{Y_n}\Big\| < \frac{\epsilon}{3\cdot2^{n}},\\
\label{reduction to FDD diagonal eq2}\Big\|\big(\sum_{m=1}^{n-1}E_m\big)TE_n\Big\| <\frac{\epsilon}{3\cdot2^{n}},\text{ and}\\
\label{reduction to FDD diagonal eq3}\Big\|E_nT\big(\sum_{m=1}^{n-1}E_m\big)\Big\| <\frac{\epsilon}{3\cdot2^{n}}.
\end{gather}
Before initiating the construction, we demonstrate how the above bounds yield the conclusion. Consider the space $Y = T_\Psi[R_\alpha^{p,0}]$ that admits the Schauder decomposition $(Y_n)_{n=1}^\infty$ with corresponding coordinate projections $(E_n)_{n=1}^\infty$, which have norm one. Denote the projection
\[P_\Psi = T_\Psi T_\Psi^\dagger:R_\alpha^{p,0}\to Y,\qquad P_\Psi f = \sum_{n=1}^\infty E_nf,\]
as in Remark \ref{IMPORTANT PROJECTION} (see also  Example \ref{mds orthocomplemented} and Proposition \ref{distro copy of orthocomplemented is orthocomplemented and others}), recalling also $T_\Psi^\dagger T_\Psi = id:R_\alpha^{p,0}\to R_\alpha^{p,0}$. Consider the operator $\tilde R = T_\Psi R T_\Psi^{-1}$, which is diagonal with respect to the aforementioned decomposition, and more precisely, for all $n\in\mathbb{N}$, $\tilde RE_n = r_nE_n$. Then,
\begin{align*}
\Big\|T_\Psi^\dagger TT_\Psi - R\Big\| &= \Big\|T_\Psi \big(T_\Psi^\dagger TT_\Psi - R\big)T_\Psi^{-1}\Big\| =  \Big\|P_\Psi T|_Y - \tilde R\Big\| \\
&= \Big\|\mathrm{SOT}-\sum_{n=1}^\infty\Big(\big(\sum_{m=1}^{n-1}E_m\big)TE_n + E_nT\big(\sum_{m=1}^{n-1} E_m\big) + E_nTE_n - E_n\tilde RE_n\Big)\Big\| \\
&\leq \sum_{n=1}^\infty\Big( \Big\|\big(\sum_{m=1}^{n-1}E_m\big)TE_n\Big\| + \Big\|E_nT\big(\sum_{m=1}^{n-1} E_m\big)\Big\| + \Big\|\Big(E_nT - r_n\cdot id\Big)\Big|_{Y_n}\Big\|\Big) < \epsilon.
\end{align*}

We return to the inductive construction of the family $(\mathcal{M}_{\lambda_{n}}, ( J_{\mu} )_{\mu \in \mathcal{M}_{\lambda_{n}}} )_{n=1}^\infty$ and the scalars $(r_n)_{n=1}^\infty$. The inductive hypothesis up to $n$ consists of
\begin{enumerate}[label=(\roman*)]
    \item\label{reduction to FDD diagonal i} the partial family $( \mathcal{M}_{\lambda_{m}}, ( J_{\mu} )_{\mu \in \mathcal{M}_{\lambda_{m}}} )_{m=1}^n$ satisfies Definition~\ref{DES},
    \item\label{reduction to FDD diagonal ii} \eqref{reduction to FDD diagonal eq1} to \eqref{reduction to FDD diagonal eq3} hold up to $n$,
    \item\label{reduction to FDD diagonal iii} there exist $\kappa_n\in\mathbb{N}$ and $\theta_n\in(0,1]$ such that for all $\mu\in\mathcal{M}_{\lambda_n}$,
    \[
    \kappa(\mu)=\kappa_n, \;\theta_\mu=\theta_n,\text{ and }b(\mu) = b(\lambda_n),\text{ and}
    \]
    \item\label{reduction to FDD diagonal iv} if $\lambda_n$ is a root then $\mathcal{M}_{\lambda_n} = \{\lambda_n'\}$, for some root $\lambda_n'$, and otherwise, for every $\mu\in\mathcal{M}_{\lambda_n}$, $\pi(\mu)\in\mathcal{M}_{\pi(\lambda_n)}$.
\end{enumerate}
We omit the basis step, since the same method applies. Assume the construction is complete up to $n$, and set $\lambda=\lambda_{n+1}$. There are two cases: $\lambda$ is a root, or a successor. The root case is analogous to the successor case, so we only record that then $\mathcal{M}_\lambda=\{\lambda'\}$, where $\lambda'$ is the root of $\mathcal{T}_\alpha$ obtained from $\lambda$ by repeating all entries except that the last, $b(\lambda)+\kappa(\lambda)$, is replaced with $b(\lambda)+\kappa_{n+1}$ for some sufficiently large $\kappa_{n+1}$. While this is analogous to the successor case, this is the only point where the fact that $\alpha$ is a limit ordinal matters; otherwise $\mathcal{T}_\alpha$ would have a unique root. We now turn to the successor case, where
\[\lambda=\pi(\lambda)^{\frown} O^{\frown}\nu,\]
with $\pi(\lambda) = \lambda_{n_0}$, for some $1\leq n_0\leq n$, $O\in\mathcal{D}^{\kappa(\pi(\lambda))}$ and $\nu\in\mathcal{T}_{b(\pi(\lambda))}$ is a root. We apply Proposition \ref{crucial lemma for the FDD-reduction 2.0} to obtain an integer
\[k_0'' = k_0''\Big(\kappa(\lambda),\sum_{m=1}^{n}2^{\kappa(\lambda_m)}, \|T\|\cdot\big\|\sum_{m=1}^nE_m\big\|,\frac{\epsilon}{3\cdot 2^{n+1}}\Big)\]
satisfying its conclusion, and we define
\[\kappa_{n+1} = k_0''\vee\kappa(\nu)\vee \Big(\max_{1\leq m\leq n}\big(\kappa_m + 1\big)\Big)\text{ and }\theta_{n+1} = \frac{\theta_{n_0}}{2^{\kappa_{n_0}}}.\]
We define the root $\nu'$ of $\mathcal{T}_{b(\pi(\lambda))}$ by taking the entries of $\nu$ unchanged, except that the last entry, $b(\nu) + \kappa(\nu)$, is replaced with $b(\nu) + \kappa_{n+1}$. Letting $\mathcal{D}^{\kappa(\mu)}_O = \{P\in\mathcal{D}^{\kappa(\mu)}:P\subset [J_\mu h_{\pi(O)} = \epsilon(O)]\}$, $\mu\in\mathcal{M}_{\pi(\lambda)}$, as in Notation \ref{notation for Ds with direction given by interval}, we define
\[\mathcal{M}_\lambda = \Big\{\mu^\smallfrown P^\smallfrown \nu':\mu\in\mathcal{M}_{\pi(\lambda)}\text{ and }P\in\mathcal{D}^{\kappa(\mu)}_O\Big\}.\]
It is clear that, for $\mu\in\mathcal{M}_\lambda$, $b(\mu) = b(\lambda)$, $\kappa(\mu) = \kappa_{n+1}$ and $\theta_\mu = \theta_{n+1}$. By the inductive hypothesis, the members of $\mathcal{M}_{\pi(\mu)}$ are pairwise disjointly supported; hence, by Proposition \ref{remark for supp and lenghts}, so are the members of $\mathcal{M}_\lambda$ and
\[\sum_{\mu\in\mathcal{M}_\lambda}\theta_\mu = \sum_{\mu\in\mathcal{M}_{\pi(\lambda)}}\sum_{P\in\mathcal{D}^{\kappa(\mu)}_O}\theta_\mu|P| =  \sum_{\mu\in\mathcal{M}_{\pi(\lambda)}}\frac{\#\mathcal{D}^{\kappa(\mu)}_O}{2^{\kappa_{n_0}}}\theta_\mu = \sum_{\mu\in\mathcal{M}_{\pi(\lambda)}}\frac{|I|2^{\kappa_{n_0}}}{2^{\kappa_{n_0}}}\theta_\mu = |I|\theta_{\pi(\lambda)} = \theta_\lambda.\]
Therefore, Definition \ref{DES} \ref{DES part a} is satisfied. Clearly, Definition \ref{DES} \ref{DES part c} holds by design. The verification of Definition~\ref{DES}\ref{DES part b} relies on the tree structure of $(\mathcal{M}_{\lambda_k})_{k=1}^n$, on \ref{reduction to FDD diagonal iv}, which holds by construction, and on the fact that, by the choice of $\kappa_{n+1}$, $\mathcal{M}_\lambda$ is disjoint from $\mathcal{M}_{\lambda_1},\ldots,\mathcal{M}_{\lambda_n}$. It remains to choose distributional embeddings $J_{\mu}:L_p^{\kappa(\lambda),0}\to L_p^{\kappa(\mu),0}$, $\mu \in \mathcal{M}_{\lambda}$, and a scalar $r_{n+1}\in\mathcal{A}(D(T))$ satisfying \ref{reduction to FDD diagonal ii}. We now note that the numerical restrictions tied to the integer $k_0''$ are satisfied by the operators $T$, $Q_1=T\Big(\sum_{m=1}^nE_m\Big)$, $Q_2=(\sum_{m=1}^nE_m)T$, and by the $\theta_{n+1}$-compressions $(T_\mu)_{\mu\in\mathcal{M}_\lambda}$, whose ranges are disjointly supported. Proposition \ref{crucial lemma for the FDD-reduction 2.0} yields distributional embeddings $J_\mu:L_p^{\kappa(\lambda),0}\to L_p^{\kappa(\mu),0}$, $\mu\in\mathcal{M}_{\lambda}$, such that letting $S_{n+1} = \sum_{\mu\in\mathcal{M}_\lambda}T_\mu J_\mu$, $Y_{n+1} = S_{n+1}[L_p^{\kappa(\lambda),0}]$, and $E_{n+1}$ the orthogonal projection onto $Y_{n+1}$, then, for some $0\leq i\leq \kappa(\lambda)-1$ and
\begin{align*}
r_{n+1} &= \theta_\lambda ^{-1}\sum_{\mu\in\mathcal{M}_\lambda}\sum_{K\in\mathcal{D}^i}\langle T_\mu h_K,TT_\mu h_K\rangle = 2^{-i}\frac{\theta_{n+1}}{\theta_\lambda}\sum_{\mu\in\mathcal{M}_\lambda}\sum_{K\in\mathcal{D}^i}|K|^{-1}\theta^{-1}_\mu\langle h^\mu_K,T h^\mu_K\rangle\\
&= \frac{1}{\big(\#\mathcal{M}_\lambda\big)\cdot\big(\#\mathcal{D}^i\big)}\sum_{\mu\in\mathcal{M}_\lambda}\sum_{K\in\mathcal{D}^i}|K|^{-1}\theta^{-1}_\mu\langle h^\mu_K,T h^\mu_K\rangle\in\mathcal{A}(D(T))
\end{align*}
we have
\begin{gather*}
\Big\| \Big( E_{n+1}T - r_{n+1} \cdot id \Big)|Y_{n+1} \Big\| < \frac{\epsilon}{3\cdot 2^{n+1}},\\
\Big\|\big(\sum_{m=1}^{n}E_m\big)TE_{n+1}\Big\| <\frac{\epsilon}{3\cdot2^{n+1}},\text{ and }\Big\|E_{n+1}T\big(\sum_{m=1}^{n}E_m\big)\Big\| <\frac{\epsilon}{3\cdot2^{n+1}}.
\end{gather*}
\end{proof}

\begin{rem}
As the proof of Theorem \ref{reduction to FDD diagonal} suggests, the ``error'' operator $T_\Psi^\dagger T T_\Psi - R$ is compact.
\end{rem}

Theorem \ref{reduction to FDD diagonal} and Proposition \ref{Prop for factorization propery} yield the factorization property of $R_\alpha^{p,0}$ for a countable limit ordinal number $\alpha$.

\begin{thm}
\label{fact property of BRS spaces} The standard MDS basis of $R_{\alpha}^{p,0}$ has the factorization property, whenever $\alpha$ is a countable limit ordinal number.
\end{thm}

\begin{rem}
Let $\alpha$ be a countable limit ordinal and $n\in\mathbb{N}$. Every bounded linear operator $T:R_{\alpha+n}^{p,0}\to R_{\alpha+n}^{p,0}$ is a factor of a bounded linear operator $\tilde T:R_{\alpha}^{p,0}\to R_{\alpha}^{p,0}$ whose diagonal entries are a subcollection of the diagonal entries of $T$ with respect to the standard MDS basis. This is achieved, e.g., by fixing $I\in\mathcal{D}^n$ and taking $A = (id - \mathbb E)Q_I :R_{\alpha+n}^{p,0}\to R_\alpha^{p,0}$ and $B=T_I:R_{\alpha}^{p,0}\to R_{\alpha+n}^{p,0}$ and letting $\tilde T = ATB$. Together with Theorem \ref{fact property of BRS spaces} and the isomorphism $R_{\alpha}^{p,0}\simeq R_{\alpha+n}^{p,0}$ from \cite{alspach:1999}, this yields the factorization property for the standard MDS bases of all spaces $R_{\alpha}^{p,0}$, $\omega\leq \alpha<\omega_1$. Note that the operator-norm bounds for the factoring operators deteriorate as $n\to\infty$.
\end{rem}
}


%
%
\bibliographystyle{plain}
\bibliography{bibliography}

\end{document}